\documentclass{amsart}

\usepackage{amssymb,color}
\usepackage{amsfonts}
\usepackage{amsmath}
\usepackage{euscript}
\usepackage{enumerate}
\usepackage{pdfsync}
\newtheorem{theorem}{Theorem}[section]
\newtheorem{lemma}[theorem]{Lemma}

\newtheorem{prop}[theorem]{Proposition}
\newtheorem{cor}[theorem]{Corollary}

\newtheorem*{Theorem1'}{Theorem 1'}

\theoremstyle{definition}

\theoremstyle{remark}

\newcommand \Z{{\mathbb Z}}

\newcommand \N{{\mathbb N}}

\def\a{\alpha}
\def\b{\beta}
\def\l{\lambda}
\def\d{\delta}
\def\g{\gamma}
\def\m{\mu}

\begin{document}

\title[Sylow subgroups of the Macdonald group on 2 parameters]{Sylow subgroups of the Macdonald group on 2 parameters}


\author{Fernando Szechtman}
\address{Department of Mathematics and Statistics, University of Regina, Canada}
\email{fernando.szechtman@gmail.com}
\thanks{The author was partially supported by NSERC grant 2020-04062}

\subjclass[2020]{20D15, 20D20}

\keywords{Macdonald group, nilpotent group, Sylow subgroup, group extension, deficiency zero}

\begin{abstract} Consider the Macdonald group $G(\alpha,\beta)=\langle A,B\,|\, A^{[A,B]}=A^\alpha,\, B^{[B,A]}=B^\beta\rangle$,
where $\alpha$ and $\beta$ are integers different from one. We fill a gap in Macdonald's original proof that $G(\alpha,\beta)$ is nilpotent,
and find the order and nilpotency class of each Sylow subgroup of $G(\alpha,\beta)$.
\end{abstract}

\maketitle

\section{Introduction}

A finite group is said to have deficiency zero if it has a finite presentation with as many generators as relations.
Families of finite groups defined by 2 generators and 2 relations have been known for a long time, see \cite{Mi}, for instance.
The first example of a finite group of deficiency zero requiring 3 generators was
$M(a,b,c)=\langle x,y,z\,|\, x^y=x^a, y^{z}=y^b, z^x=z^c\rangle$, found in 1959 by Mennicke \cite{Me},
who proved that $M(a,b,c)$ is finite when $a=b=c\geq 2$.
It is easy to see that $M(a,b,c)$ does require 3 generators whenever $a-1,b-1,c-1$ share a common prime factor.
A sufficient condition for the finiteness of $M(a,b,c)$ is $a,b,c\notin\{-1,1\}$, shown by Jabara \cite{Ja} in 2009.
Upper bounds for the order of $M(a,b,c)$ can be found in \cite{JR,AA,Ja}. The actual order of $M(a,b,c)$ is known only in certain cases
(see \cite{Me,A,AA,Ja}). The attention received by the Mennicke groups spurred the search for other finite groups of deficiency zero,
and many of these have been found since then. See \cite{M,W,P,J,CR,CRT,AS,AS2}, for instance.
The structure of the Sylow subgroups of the Wamsley groups $G_2(\a,\b,\gamma)$ from \cite{W} has just been
elucidated in \cite{PS} when $\a=\b$ and $\gamma>0$ by means of considerable machinery.
As exemplified by the Mennicke, Wamsley, and other groups, it may be quite difficult to find the order and other structural
properties of the members of a given family of finite groups of deficiency zero.

In this paper we determine the order and nilpotency class of the members of one such family,
namely the 2-parameter Macdonald groups $G(\a,\b)$ from \cite{M}, defined by
$$
G(\a,\b)=\langle A,B\,|\, A^{[A,B]}=A^{\a},\, B^{[B,A]}=B^{\b}\rangle,
$$
where $\a$ and $\b$ are integers different from one that will be fixed throughout the entire paper.

We begin by filling a gap in \cite{M} and prove that
$G(\a,\b)$ is nilpotent. Macdonald showed that $G(\a,\b)$ is finite,
so $G(\a,\b)$ is the direct product of its Sylow subgroups and these are, in fact, our main objects of investigation.
For each prime number $p$ that divides the order
of $G(\a,\b)$, we write $G(\a,\b)_p$ for the Sylow $p$-subgroup of
$G(\a,\b)$ as well as $a,b,c$ for the images of $A,B,C=[A,B]$, respectively,
under the canonical projection $G(\a,\b)\to G(\a,\b)_p$. In this notation,
we find the order and nilpotency class of $G(\a,\b)_p$,
as well as the orders of $a,b,c$.
This is achieved by: making use of known of relations among $A,B,C$ taken from \cite{M};
appealing to new relations among $A,B,C$ derived in Section \ref{secfr}; dividing the analysis of the structure
of $G(\a,\b)_p$ into various cases depending on the nature of $p$, the $p$-valuations $v_p(\a-1)$,
$v_p(\b-1)$, $v_p(\a-\b)$, as well as further parameters, and finding yet more relations
among $a,b,c$ valid in each specific case, until sharp bounds on the order and class of $G(\a,\b)_p$, and the orders of $a,b,c$ emerge;
constructing an image of $G(\a,\b)_p$, by means a sequence of group extensions with cyclic quotients, that attains
these sharp bounds. Even after taking into the account the isomorphism $G(\a,\b)\cong G(\b,\a)$ and rapidly discarding
the case when $G(\a,\b)_p$ is cyclic, the above procedure still breaks up into 19 different cases.
Most of the above analysis is made under the assumption that $\a>1$ and $\b>1$, which is a requirement of
the foregoing relations, but in Section \ref{gc} we show
that all of our structural results remain valid without this assumption.

It is shown in \cite[Section 2]{M} that $A$ and $B$ have finite orders, which implies \cite[p. 603]{M}
that $C$ has finite order. As mentioned in \cite[p. 603]{M}, the finiteness of $o(A),o(B),$ and $o(C)$
can be used to show that $G(\a,\b)$ is finite.  Details can be found in \cite[Lemma 6.1]{MS}, a result that
implies that $|G(\a,\b)|$ is a factor of $o(A)o(B)o(C)$. According to \cite[p. 609]{M}, $o(C)$ divides
both $o(A)$ and $o(B)$. Moreover, \cite[Section 4]{M} establishes the non-trivial result that
the prime factors of $o(A)$ (resp. $o(B)$) are exactly the same as those of $\a-1$ (resp. $\b-1$).
It follows that the prime factors of $|G(\a,\b)|$ are those of $(\a-1)(\b-1)$. Macdonald did not
attempt to compute the order of $G(\a,\b)$ and this was left open as a ``complicated question".
The special case when $\a=\b$ was recently settled in \cite{MS}. In this paper we determine the
order of $G(\a,\b)$ for arbitrary $\a$ and $\b$. 

Macdonald \cite[Section 5]{M} also showed that $G(\a,\b)$ is nilpotent
of class 7 or less, provided $\gcd(\a-1,6)=1=\gcd(\b-1,6)$. A few details are in order regarding this subtle point.
A key fact used by Macdonald, proved in \cite[p. 611]{M}, is that if $\a>1$ and $\gcd(\a-1,6)=1$,
and we set $\gamma_\a=\a^\a-(1+\a+\cdots+\a^{\a-1})$, then for any prime factor $p$ of $\a-1$, we have
\begin{equation}
\label{aze0}
v_p((\a-1)\gamma_\a)=3v_p(\a-1).
\end{equation}
A second key fact used by Macdonald  \cite[Eqs. (2.15) and (2.16)]{M} is that if $\a>1$, then
\begin{equation}
\label{aze}
A^{(\a-1)\gamma_\a}\in Z(G(\a,\b)),\; A^{\epsilon (\a-1)\gamma_\a}=1,
\end{equation}
where $\epsilon=\gcd(\a-1,\b-1)$. Combining these two facts with the foregoing result on the prime factors of $o(A)$ yields
that if $\a>1$ and $\gcd(\a-1,6)=1$, then \cite[p. 611]{M}
$$
A^{(\a-1)^3}\in Z(G(\a,\b)),\; A^{\epsilon (\a-1)^3}=1.
$$
These and analogous results for $B$ allow Macdonald to begin the proof that $G(\a,\b)$ is nilpotent when $\gcd(\a-1,6)=1=\gcd(\b-1,6)$.
Macdonald states without proof \cite[p. 612]{M} that $G(\a,\b)$ is nilpotent
in general, and that the proofs in the remaining cases are  essentially similar to the previous case.
He seems to rely on the assertion, made in \cite[p. 611]{M}, that $v_3((\a-1)\gamma_\a)=4$ when $\a>1$ and $\a\equiv 7\mod 9$.
This 3-valuation is wrong. In fact, $v_3((\a-1)\gamma_\a)$ is unbounded for arbitrary $\a>1$ and $\a\equiv 7\mod 9$, as shown
in \cite[Proposition 2.1]{MS}. As a result, Macdonald's argument
for the nilpotence of $G(\a,\b)$, as given in \cite[p. 611-612]{M}, cannot get off the ground.
In Proposition \ref{nuevaca}, we provide a replacement for (\ref{aze}) valid for arbitrary $\a$. 	
The corresponding replacement for (\ref{aze0}) can be found in \cite[Proposition 2.2]{MS}. These replacements
allow us to produce Lemma \ref{casos}, and armed with these tools we establish the nilpotence of $G(\a,\b)$ in Theorem \ref{nilp}.

Macdonald \cite[p. 612]{M} states without proof that the class of $G(\a,\b)$ may be as high as 8, and wonders whether
this bound is ever reached, singling out $G(7,34)$ as a likely candidate. He came back to this question in \cite{M2}, ten years after the publication of \cite{M}, proving by means of a computer calculation that $G(7,34)$ had order $3^{10}$ and class 7.
Macdonald left open as a ``complicated question" the calculation of the exact class of~$G(\a,\b)$.
The special case $\a=\b$ was settled in \cite{MS}.
In this paper we determine the nilpotency class of each of the Sylow subgroups of $G(\a,\b)$
for arbitrary $\a$ and $\b$. We settle the problem of the largest possible class ever attained by $G(\a,\b)$,
which turns out to be 7, and show that this bound is reached if and only if
$\a,\b\equiv 7\mod 9$ and $\a\equiv\b\mod 27$, which is the only case when the class of the Sylow 3-subgroup of $G(\a,\b)$
reaches 7. The class of all other Sylow $p$-subgroups of $G(\a,\b)$ is $\leq 6$, and this is only attained 
under the extreme conditions described in Theorems A and B below.

Set $G=G(\a,\b)$, fix a prime $p\in\N$, and write $G_p=G(\a,\b)_p$ for
the sole Sylow $p$-subgroup of~$G$. We let $v_p(\a-1)=m$ and $v_p(\b-1)=n$,
so that $\a=1+p^m u$ and $\b=1+p^n v$, where $m,n\geq 0$ and $u,v\in\Z$ are relatively prime to $p$.
We further set $\epsilon=\gcd(\a-1,\b-1)$ and $\ell=v_p(\a-\b)$, allowing for the possibility that $\a=\b$,
in which case $\ell=\infty$. If $\a\neq\b$, then $\a-\b=p^\ell k$, where $\ell,k\in\Z$, $\ell\geq 0$, and $p\nmid k$.
We write $e=v_p(|G(\a,\b)|)$ and let $f$ stand for the nilpotency class of~$G_p$. As $G(\a,\b)\cong G(\b,\a)$,
we may assume without loss throughout this section that $m\geq n$.

Macdonald \cite[p. 612]{M} states that if $\gcd(\a-1,6)=1=\gcd(\b-1,6)$ then $e\leq 10n$, which is false, as well as $e\leq 9n+m$, which is correct. Macdonald \cite[p. 612]{M} also states without proof that, in general, $e\leq 9n+m$, except when $p=3$, in which case
$e\leq 9n+m+3$. These bounds are correct, but not sharp, except when $n=0$ and $p\neq 3$. If $n=0$ then
$G_p$ is cyclic of order~$p^m$,
so when $p\neq 3$, we do get $e=m=9n+m$ in this trivial case. We assume next that $n>0$.

Suppose first that $p>3$ or that $p=3$ and $\a,\b\not\equiv 7\mod 9$.
If $n=\ell$, then $e=4n+m$ and $f=3$ by Theorem \ref{teo1}. Suppose next $\ell>n$, which can only happen if $m=n$. 
If $\ell\geq 2m$, then $e=7m$
and $f=5$ by Theorem \ref{teo2}.
This includes the case $\a=\b$ considered in \cite{MS}. If $m<\ell<2m$ the values of $e$ and $f$ are much subtler to determine.
Indeed, if $2\ell<3m$, then $e=2m+3\ell<\frac{13}{2} m$ and $f=5$ by Theorem \ref{teo3}. 
If $2\ell>3m$, then $e=5m+\ell<7m$ and $f=5$ by Theorem \ref{teo4}. The remaining case, namely $2\ell=3m$, is found
in Theorem \ref{teo5}, and can be stated as follows.

\medskip

\noindent{\bf Theorem A.} {\it Suppose $2\ell=3m$, and set $s=v_p(2k^2-u^3)$. 
Then $f=5$ if $s=0$ and $f=6$ if $s>0$. 
Moreover, if $0\leq s\leq m/2$ then
$e=s+13m/2<7m$, $o(a)=p^{s+5m/2}=o(b)$, and
$o(c)=p^{2m+s}$, while if $s\geq m/2$ then $e=7m$, $o(a)=p^{3m}=o(b)$, and
$o(c)=p^{5m/2}$.}

\medskip

We found it difficult to detect the relevance of the parameter $v_p(2k^2-u^3)$ to the structure of $G_p$ when $2\ell=3m$.
Also challenging was being able to determine the precise relations required to sharply bound the order and nilpotency class of $G_p$,
and to construct an image of $G_p$ that attains these bounds. This construction alone takes up 4 pages!

Perusing the case $p>3$, or $p=3$ and $\a,\b\not\equiv 7\mod 9$, when $\ell>n$, we find that we always have
$e\leq 7m<10m=9n+m$ and $f\leq 6$.

Suppose next that $p=3$ and that $\a\equiv 7\mod 9$ or
$\b\equiv 7\mod 9$. If $\a,\b\equiv 7\mod 9$ and $\a\equiv\b\mod 27$, then $e=10$ and $f=7$ by Theorem \ref{teo6}
(which includes the case $\a=\b$ considered in \cite{MS}), 
and this is the closest $e$ and $f$ ever are to the upper bounds proposed in \cite{M}.
It turns out that the factor of 27 appearing in Macdonald's expression $27(\a-1)(\b-1)\epsilon^8$ is not needed and that $f$ never reaches 8.
Thus, if $\a=\b$
is chosen from the list
$$
7,16,25,34,43,52,61,70,79
$$
or $(\a,\b)$ is taken from the list
$$
(7,34), (16,43), (25,52), (34,61), (43,70), (52,79), (7,61), (16,70), (25,79),
$$
then $e=10$ and $f=7$ (it is easy to see that all other cases reduce to
those listed above). In particular, the Sylow 3-subgroup of $G(7,34)$ has order $3^{10}$ and nilpotency class 7.

If $\a,\b\equiv 7\mod 9$ but $\a\not\equiv\b\mod 27$, then $e=8$ and $f=5$ by Theorem \ref{teo7}. If
$\b\equiv 7\mod 9$ and $\a\equiv 4\mod 9$, then $e=5$ and $f=3$ by Theorem \ref{teo8}. If
$\b\equiv 7\mod 9$ and $\a\equiv 1\mod 9$, then $e=4+m$ and $f=3$ by Theorem \ref{teo9}. In all these cases, except when
$\a,\b\equiv 7\mod 9$ and $\a\equiv\b\mod 27$, we have $e<9n+m$ and $f\leq 5$.

Suppose finally that $p=2$.
If $n=1$ and $m=1$ then $G_2\cong Q_{16}$, the generalized quaternion group of order 16 and class 3, by Theorem \ref{teo10}.
If $n=1$ and $m>2$, then $e=m+4$ and $f=3$ by Theorem \ref{teo11}, while if 
$n=1$ and $m=2$, then $e=7$ and $f=4$ by Theorem \ref{teo12}.
Suppose next that $m,n>1$. If $\ell=n$ (so that $m>n$, for $m=n$ forces $\ell>n$), then $e=m+4n$ and $f=3$ by Theorem \ref{teo13}.
Suppose from now on that $m=n>1$. If $\ell\geq 2m$, then $e=7m-3$ and $f=5$ by Theorem \ref{teo14}
(the special case $\a=\b$ was considered in \cite{MS}).
If $\ell=2m-1$ or $\ell=2m-2$ then $e=7m-3$ and $f=5$ by Theorems \ref{teo15} and \ref{teo16}.
Suppose from now on $m<\ell<2m$ and $\ell\leq 2m-3$. Three cases arise: $2(\ell+1)\leq 3m$, $2(\ell+1)=3m+1$, and $2(\ell+1)>3m+1$.
The case $2(\ell+1)=3m+1$, which forces $m\geq 5$ to be odd, is handled in Theorem \ref{teo17}, and reads as follows.

\medskip

\noindent{\bf Theorem B.} {\it Suppose that $m=n\geq 5$ and $2\ell+2=3m+1$, and set $s=v_2(k^2-u^3)$.
Then $f=6$. If $s<(m-3)/2$, then $m\geq 7$, $e=(13m+2s-3)/2<7m-3$,
$o(a)=2^{(5m+2s+1)/2}=o(b)$,
and $o(c)=2^{2m+s}$. If $s\geq (m-3)/2$,
then $e=7m-3$, $o(a)=2^{3m-1}=o(b)$, and $o(c)=2^{(5m-3)/2}$.}

\medskip

The same comments made about the challenges involved in Theorem A apply to Theorem B.

If $2(\ell+1)>3m+1$, then $e=5m+\ell-1$ and $f=5$ by Theorem~\ref{teo18}.
If $2(\ell+1)\leq 3m$ then $e=2m+3\ell$ and $f=5$ by Theorem~\ref{teo19}. 

This completes the description of the order and class of $G_p$ in all cases. 
Our proofs are theoretical and computer-free, although the results are confirmed by GAP and Magma calculations.
We are very grateful to A. Previtali for this verification.


In terms of notation, given a group $T$, we~set
$$
[x,y]=x^{-1}y^{-1}xy,\; y^x=x^{-1}yx,\; {}^x y=xyx^{-1},\quad x,y\in T.
$$
If $x\in T$ has finite order $r$, and $s,t\in\Z$, with $t\neq 0$ and $\gcd(r,t)=1$,
we set $x^{s/t}=x^{st_0}$, where $tt_0\equiv 1\mod r$, which is easily seen to be well-defined.

For an integer $a>1$, we define the integers $\d_a$ and $\l_a$ by
$$
\d_{a}=(a-1)(a+2a^2+\cdots+(a-1)a^{a-1}),
$$
$$
 \l_{a}=(a-1)(a+2a^2+\cdots+(\d_{a}-1)a^{\d_{a}-1}),
$$
as well as the integer the $\g_a$, appearing in \cite[p. 604]{M}, by
$$
\g_a=a^a-(1+a+\cdots+a^{a-1}).
$$
Note that
\begin{equation}\label{lomismo}
\g_a=\d_a/a.
\end{equation}

In addition, for an integer $a$, we define the integer $\m_a$, appearing in \cite[Section 2]{MS}, by
$$\mu_a=a^{a^2+2}-a(1+a+\cdots+a^{a^2-1}),$$
where the right hand side should be interpreted as 0 if $a\in\{-1,0,1\}$.

\section{Further relations in the Macdonald group}\label{secfr}

We keep throughout the paper the notation defined in the Introduction. Observe that for
$$
G(\a,\b)=\langle A,B\,|\, A^{[A,B]}=A^{\a},\, B^{[B,A]}=B^{\b}\rangle, G(\b,\a)=\langle X,Y\,|\, X^{[X,Y]}=X^{\b},\, Y^{[Y,X]}=Y^{\a}\rangle,
$$
and $C=[A,B]$, $Z=[X,Y]$, we have inverse isomorphisms $G(\a,\b)\leftrightarrow G(\b,\a)$, given by $A\leftrightarrow Y$ and
$B\leftrightarrow X$, with $C\leftrightarrow Z^{-1}$. This allows us to transform valid results in $G(\a,\b)$ to other valid results
in $G(\a,\b)$ via the replacements $A\leftrightarrow B$, $C\leftrightarrow C^{-1}$, and $\a\leftrightarrow \b$. For instance,
for $i>0$, we see, as in \cite[Eq. (1.4)]{M}, that
$$
(A^i)^B=C^i A^{\a(1+\a+\cdots+\a^{i-1})}
$$
is valid in $G(\a,\b)$. As this is true in every Macdonald group,
$$
(X^i)^Y=Z^i Y^{\b(1+\b+\cdots+\b^{i-1})}
$$
holds in $G(\b,\a)$, so the isomorphism $G(\b,\a)\to G(\a,\b)$ yields that
$$
(B^i)^A=C^{-i} B^{\b(1+\b+\cdots+\b^{i-1})}
$$
is valid in $G(\a,\b)$. In the sequel we will derive such consequences automatically.


\begin{prop}\label{nuevaca} We have
$$
A^{(\alpha-1)\mu_\alpha}=B^{(\beta-1)\mu_\beta},\; A^{\epsilon(\alpha-1)\mu_\alpha}=1=B^{\epsilon(\beta-1)\mu_\beta}.
$$
\end{prop}

\begin{proof} We will repeatedly and implicitly use \cite[Eqs. (1.4) and (1.6)]{M}.
As conjugation by $C^2$ is an automorphism of $G$, the defining relations of $G$ yield
\begin{equation}\label{iq} [A,B^{\beta^2}]^{C^2}=[A^{\alpha^2},B].
\end{equation}

Regarding the left hand side of (\ref{iq}), we have
$$
(B^{\beta^2})^A=C^{-\beta^2}B^{\beta(1+\beta+\cdots+\beta^{\beta^2-1})},
$$
which successively implies
$$
(B^{-\beta^2})^A=B^{-\beta(1+\beta+\cdots+\beta^{\beta^2-1})}C^{\beta^2},
$$
$$
[A,B^{\beta^2}]=(B^{-\beta^2})^A B^{\beta^2}=B^{-\beta(1+\beta+\cdots+\beta^{\beta^2-1})}C^{\beta^2}B^{\beta^2},
$$
\begin{equation}\label{iq3}
[A,B^{\beta^2}]^{C^2}=C^{-2}B^{-\beta(1+\beta+\cdots+\beta^{\beta^2-1})}C^{\beta^2}B^{\beta^2}C^2.
\end{equation}

As for the right hand side of (\ref{iq}), we have
$$
(A^{\alpha^2})^B=C^{\alpha^2}A^{\alpha(1+\alpha+\cdots+\alpha^{\alpha^2-1})},
$$
\begin{equation}\label{iq4}
[A^{\alpha^2},B]=A^{-\alpha^2}(A^{\alpha^2})^B=A^{-\alpha^2}C^{\alpha^2}A^{\alpha(1+\alpha+\cdots+\alpha+\alpha^{\alpha^2-1})}.
\end{equation}
It follows from (\ref{iq}) that the right hand sides of (\ref{iq3}) and (\ref{iq4}) are equal. Thus
\begin{align*}
B^{-\beta(1+\beta+\cdots+\beta^{\beta^2-1})}C^{\beta^2}B^{\beta^2}C^2 &=
C^2A^{-\alpha^2}C^{\alpha^2}A^{\alpha(1+\alpha+\cdots+\alpha^{\alpha^2-1})}\\ &=
C^2A^{-\alpha^2}C^{-2}C^{\alpha^2+2}A^{\alpha(1+\alpha+\cdots+\alpha^{\alpha^2-1})}\\ &=
A^{-1}C^{\alpha^2+2}A^{\alpha(1+\alpha+\cdots+\alpha^{\alpha^2-1})}.
\end{align*}
On the other hand,
\begin{align*}
B^{-\beta(1+\beta+\cdots+\beta^{\beta^2-1})}C^{\beta^2}B^{\beta^2}C^2 &=
B^{-\beta(1+\beta+\cdots+\beta^{\beta^2-1})}C^{\beta^2+2}C^{-2}B^{\beta^2}C^2\\ &=
B^{-\beta(1+\beta+\cdots+\beta^{\beta^2-1})}C^{\beta^2+2}B\\ &=
B^{-\beta(1+\beta+\cdots+\beta^{\beta^2-1})}C^{\beta^2+2}B C^{-(\beta^2+2)}C^{\beta^2+2}\\ &=
B^{-\beta(1+\beta+\cdots+\beta^{\beta^2-1})}B^{\beta^{\beta^2+2}}C^{\beta^2+2}\\&=
B^{\mu_\beta}C^{\beta^2+2},
\end{align*}
\begin{align*}
A^{-1}C^{\alpha^2+2}A^{\alpha(1+\alpha+\cdots+\alpha^{\alpha^2-1})} &=
C^{\alpha^2+2}C^{-(\alpha^2+2)}A^{-1}C^{\alpha^2+2}A^{\alpha(1+\alpha+\cdots+\alpha^{\alpha^2-1})}\\ &=
C^{\alpha^2+2}A^{-\alpha^{\alpha+2}}A^{\alpha(1+\alpha+\cdots+\alpha^{\alpha^2-1})}\\ &=
C^{\alpha^2+2} A^{-\mu_\alpha},
\end{align*}
so
$$
B^{\mu_\beta}C^{\beta^2+2}=C^{\alpha^2+2}A^{-\mu_\alpha}.
$$
Let $\alpha_0$ be the inverse of $\alpha$ modulo the order of $A$. We then have
\begin{equation}\label{izq}
B^{\mu_\beta}=C^{\alpha^2+2}A^{-\mu_\alpha}C^{-(\alpha^2+2)}C^{\alpha^2+2}C^{-(\beta^2+2)}=
A^{-\mu_\alpha \alpha_0^{\alpha^2+2}}C^{\alpha^2-\beta^2}.
\end{equation}
Conjugating both sides by $C^{-1}$ yields
\begin{equation}\label{izq2}
B^{\beta\mu_\beta}=A^{-\mu_\alpha \alpha_0^{\alpha^2+3}}C^{\alpha^2-\beta^2}.
\end{equation}
Multiplying (\ref{izq2}) by the inverse of (\ref{izq}) gives
\begin{equation}\label{izq3}
B^{\mu_\beta(\beta-1)}=A^{\mu_\alpha \alpha_0^{\alpha^2}(1-\alpha_0)},
\end{equation}
which is a central element of $G$. Thus, conjugating (\ref{izq3}) by $C^{\alpha^2+1}$ yields
$$
B^{\mu_\beta(\beta-1)}=A^{\mu_\alpha(\alpha-1)}.
$$
Conjugating $A^{\mu_\alpha(\alpha-1)}\in Z(G)$ by $C$ and $B^{\mu_\beta(\beta-1)}\in Z(G)$ by $C^{-1}$, we obtain
$$
B^{\mu_\beta(\beta-1)(\alpha-1)}=A^{\mu_\alpha(\alpha-1)^2}=1=B^{\mu_\beta(\beta-1)^2}=A^{\mu_\alpha(\alpha-1)(\beta-1)},
$$
whence
\[
A^{\mu_\alpha(\alpha-1)\epsilon}=1=B^{\mu_\beta(\beta-1)\epsilon}.\qedhere
\]
\end{proof}

We proceed to review material from \cite[Section 2]{M} analogous to the above, as well as to obtain further relations.
Assume for the remainder of this section that $\a,\b>1$.

As conjugation by $C$ is an automorphism of $G$, the defining relations of $G$ give
$$
[A^\a,B]=[A,B^\b]^C.
$$
Here $[A^\a,B]=(B^{-1})^{A^\a} B=(B^{A^\a})^{-1} B$ and induction shows that
\begin{equation}\label{indi}
B^{A^i}=BA^{(\a-1)(\a+2\a^2+\cdots+(i-1)\a^{i-1})}C^{-i},\quad i\geq 1.
\end{equation}
The right hand side should be interpreted as $BC^{-1}$ when $i=1$. Applying (\ref{indi}) with $i=\a$ gives
$$
[A^\a,B]=(B^{A^\a})^{-1} B=C^\a A^{-\d_\a} B^{-1}B=C^\a A^{-\d_\a}.
$$
Likewise, $[A,B^\b]=A^{-1} A^{B^\b}$, where
\begin{equation}\label{indi2}
A^{B^i}=AB^{(\b-1)(\b+2\b^2+\cdots+(i-1)\b^{i-1})}C^{i},\quad i\geq 1,
\end{equation}
and the right hand side should interpreted as $AC$ when $i=1$. Applying (\ref{indi2}) with $i=\b$ gives
$$
[A,B^\b]=A^{-1} A^{B^\b}=A^{-1}AB^{\d_\b}C^{\b}=B^{\d_\b}C^{\b}.
$$
Thus, if $\b_0\in\Z$ satisfies $\b\b_0\equiv 1\mod o(B)$, then
$$
[A,B^\b]^C=B^{\b_0 \d_\b} C^\b,
$$
and therefore
$$
C^\a A^{-\d_\a}=[A^\a,B]=[A,B^\b]^C=B^{\b_0 \d_\b} C^\b,
$$
which implies
$$
A^{-\d_\a}=C^{\b-\a} B^{\b_0^{\b+1} \d_\b},
$$
or
\begin{equation}\label{adabc}
A^{\d_\a}=B^{-\b_0^{\b+1} \d_\b}C^{\a-\b}.
\end{equation}
Let $\a_0\in\Z$ satisfy $\a\a_0\equiv 1\mod o(A)$. Then the transformation $A\leftrightarrow B$, $C\leftrightarrow C^{-1}$, $\a\leftrightarrow \b$ yields
\begin{equation}\label{bdac}
B^{\d_\b}=A^{-\a_0^{\a+1} \d_\a}C^{\a-\b}.
\end{equation}
The operator $[A,-]$ applied to (\ref{bdac}) and the identity $[x,yz]=[x,z][x,y]^z$, valid in any group, give
\begin{equation}\label{bdac2}
B^{\l_\b}C^{\d_\b}=A^{\a^{\a-\b}-1},
\end{equation}
where, if $\a<\b$ the right hand side should interpreted by means of $A^{\a^{-1}}=A^{\a_0}$, as indicated by the end of the Introduction.
The transformation $A\leftrightarrow B$, $C\leftrightarrow C^{-1}$, $\a\leftrightarrow \b$ now yields
\begin{equation}\label{adabc2}
A^{\l_\a}C^{-\d_\a}=B^{\b^{\b-\a}-1},
\end{equation}
where, if $\b<\a$ the right hand side should interpreted by means of $B^{\b^{-1}}=B^{\b_0}$. Conjugating (\ref{adabc}) by $C$ produces
\begin{equation}\label{adabc3}
A^{\d_\a\a}=B^{-\b_0^{\b+2} \d_\b}C^{\a-\b}.
\end{equation}
Multiplying (\ref{adabc3}) on the right by the inverse of (\ref{adabc}) gives
$$
A^{\d_\a(\a-1)}=B^{\d_\b \b_0^{\b+1}(1-\b_0)}.
$$
This is a central element of $G$, so conjugating it by $C^{-(\b+2)}$ results in
\begin{equation}\label{adabc4}
A^{\d_\a(\a-1)}=B^{\d_\b(\b-1)}\in Z(G),
\end{equation}
which implies
\begin{equation}\label{adabc5}
A^{\d_\a(\a-1)\epsilon}=1=B^{\d_\b(\b-1)\epsilon}.
\end{equation}

We proceed to justify \cite[Eq. (2.18)]{M}. Raising (\ref{bdac}) to the $(\b-1)$th power yields
\begin{equation}\label{h99}
B^{\d_\b(\b-1)}=(A^{-\a_0^{\a+1} \d_\a}C^{\a-\b})^{\b-1}.
\end{equation}
As $C$ normalizes $\langle A\rangle$, it follows that
$$
B^{\d_\b(\b-1)}=C^{(\a-\b)(\b-1)}A^i,\quad i\in \Z.
$$
Since  $B^{\d_\b(\b-1)}=A^{\d_\a(\a-1)}$ by (\ref{adabc4}), we infer
\begin{equation}\label{h}
C^{(\a-\b)(\b-1)}\in \langle A\rangle.
\end{equation}
The transformation $A\leftrightarrow B$, $C\leftrightarrow C^{-1}$, $\a\leftrightarrow \b$ applied to (\ref{h}) yields
\begin{equation}\label{h2}
C^{(\a-\b)(\a-1)}\in \langle B\rangle.
\end{equation}
From (\ref{h}) and (\ref{h2}), and following the convention stipulated in the Introduction, we obtain
\begin{equation}\label{h3}
A^{\a^{(\a-\b)(\b-1)}-1}=1=B^{\b^{(\a-\b)(\a-1)}-1}.
\end{equation}

\begin{lemma}\label{casos} (a) If $2\nmid (\alpha-1)$ and $\a\not\equiv 7\mod 9$, then
 $$A^{(\alpha-1)^3}\in Z(G),\; A^{(\alpha-1)^4}=1,\;
C^{(\alpha-1)^3}=1.
$$

(b) If $2\nmid (\alpha-1)$ and $\a\equiv 7\mod 9$, then
 $$A^{3(\alpha-1)^3}\in Z(G),\; A^{3(\alpha-1)^4}=1,\;
C^{3(\alpha-1)^3}=1.
$$

(c) If $2|(\alpha-1)$ and $\a\not\equiv 7\mod 9$, then
$$A^{(\alpha-1)^3/2}\in Z(G),\; A^{(\alpha-1)^4/2}=1,\;
C^{(\alpha-1)^3/2}=A^{(\alpha-1)^4/4}\in Z(G),\; C^{(\alpha-1)^3}=1.
$$

(d) If $2|(\alpha-1)$ and $\a\equiv 7\mod 9$, then
$$A^{3(\alpha-1)^3/2}\in Z(G),\; A^{3(\alpha-1)^4/2}=1,\;
C^{3(\alpha-1)^3/2}=A^{3(\alpha-1)^4/4}\in Z(G),\; C^{3(\alpha-1)^3}=1.
$$
\end{lemma}

\begin{proof} We will repeatedly and implicitly use \cite[Eqs. (1.4) and (1.6)]{M}, as well as
the fact \cite[Section 4]{M} that
the prime factors of the order of $A$ are precisely those of $\alpha-1$. In all cases, we have
$A^{(\alpha-1)\gamma_\alpha}\in Z(G)$ by (\ref{lomismo}) and (\ref{adabc4}), and
$A^{(\alpha-1)\mu_\alpha}\in Z(G)$ by Proposition \ref{nuevaca}.

(a) By \cite[Proposition 2.1]{MS}, we have $v_p((\alpha-1)\gamma_\alpha)=3v_p(\alpha-1)$
for all prime factors $p$ of $\alpha-1$, whence $A^{(\alpha-1)^3}\in Z(G)$. Conjugating $A^{(\alpha-1)^3}$ by $C$
yields $A^{(\alpha-1)^4}=1$. Since $A^{(\alpha-1)^3}\in Z(G)$,
$$
A^{(\alpha-1)^3}=(A^{(\alpha-1)^3})^B=
C^{(\alpha-1)^3}A^{\alpha(\alpha^{(\alpha-1)^3}-1)/(\alpha-1)}.
$$
Here
$
(\alpha^{(\alpha-1)^3}-1)/(\alpha-1)\equiv  (\alpha-1)^3 \mod (\alpha-1)^4,
$
so
$
\alpha (\alpha^{(\alpha-1)^3}-1)/(\alpha-1)\equiv  (\alpha-1)^3\mod (\alpha-1)^4.
$
As $A^{(\alpha-1)^4}=1$, we deduce
$
A^{(\alpha-1)^3}=C^{(\alpha-1)^3}A^{(\alpha-1)^3}.
$
This proves that $C^{(\alpha-1)^3}=1$.

(b) Suppose first that
$\alpha=1+3q$, $q\in\N$, $q\equiv -1\mod 3$, and $v_3(q+1)=1$. Then
$v_p((\alpha-1)\gamma_\alpha)=3v_p(\alpha-1)$
for any positive prime factor $p\neq 3$ of $\alpha-1$ and $v_3((\alpha-1)\gamma_\alpha)=4$ by \cite[Proposition 2.1]{MS},
whence $A^{3(\alpha-1)^3}\in Z(G)$. Suppose next that
$\alpha=1+3q$, $q\in\N$, $q\equiv -1\mod 3$, and $v_3(q+1)>1$. Then $v_p((\alpha-1)\mu_\alpha)=3v_p(\alpha-1)$
for any positive prime factor $p\neq 3$ of $\alpha-1$ and $v_3((\alpha-1)\mu_\alpha)=4$ by \cite[Proposition 2.2]{MS},
whence $A^{3(\alpha-1)^3}\in Z(G)$. Thus, $A^{3(\alpha-1)^3}\in Z(G)$ in both cases. Conjugating $A^{3(\alpha-1)^3}$ by $C$
yields $A^{3(\alpha-1)^4}=1$. Since $A^{3(\alpha-1)^3}\in Z(G)$,
$$
A^{3(\alpha-1)^3}=(A^{3(\alpha-1)^3})^B=
C^{3(\alpha-1)^3}A^{\alpha(\alpha^{3(\alpha-1)^3}-1)/(\alpha-1)}.
$$
Now $(\alpha^{3(\alpha-1)^3}-1)/(\alpha-1)\equiv  3(\alpha-1)^3\mod 3(\alpha-1)^4,
$
whence
$
\alpha(\alpha^{3(\alpha-1)^3}-1)/(\alpha-1) \equiv 3(\alpha-1)^3\mod 3(\alpha-1)^4.
$
As $A^{3(\alpha-1)^4}=1$, we infer
$
A^{3(\alpha-1)^3}=C^{3(\alpha-1)^3}A^{3(\alpha-1)^3},
$
which proves $C^{3(\alpha-1)^3}=1$.

(c) By \cite[Proposition 2.1]{MS}, we have $v_p((\alpha-1)\gamma_\alpha)=3v_p(\alpha-1)$
for any positive prime factor $p\neq 2$ of $\alpha-1$ and $v_2((\alpha-1)\gamma_\alpha)=3v_2(\alpha-1)-1$,
whence $A^{(\alpha-1)^3/2}\in Z(G)$. Conjugating
$A^{(\alpha-1)^3/2}$ by $C$ yields $A^{(\alpha-1)^4/2}=1$. Since $A^{(\alpha-1)^3/2}\in Z(G)$,
$$
A^{(\alpha-1)^3/2}=(A^{(\alpha-1)^3/2})^B=
C^{(\alpha-1)^3/2}A^{\alpha(\alpha^{(\alpha-1)^3/2}-1)/(\alpha-1)}.
$$
Here
$
(\alpha^{(\alpha-1)^3/2}-1)/(\alpha-1)\equiv  (\alpha-1)^3/2-(\alpha-1)^4/4\mod (\alpha-1)^4/2,
$
and therefore we have
$
\alpha (\alpha^{(\alpha-1)^3/2}-1)/(\alpha-1) \equiv  (\alpha-1)^3/2-(\alpha-1)^4/4\mod (\alpha-1)^4/2.
$
As $A^{(\alpha-1)^4/2}=1$, we infer
$
A^{(\alpha-1)^3/2}=C^{(\alpha-1)^3/2}A^{(\alpha-1)^3/2-(\alpha-1)^4/4},
$
which proves $C^{(\alpha-1)^3/2}=A^{(\alpha-1)^4/4}\in Z(G)$ and $C^{(\alpha-1)^3}=1$.

(d) Suppose first that
$\alpha=1+3q$, $q\in\N$, $q\equiv -1\mod 3$, and $v_3(q+1)=1$. We then have $v_p((\alpha-1)\gamma_\alpha)=3v_p(\alpha-1)$
for any positive  prime factor $p\notin\{2,3\}$ of $\alpha-1$,
$v_3((\alpha-1)\gamma_\alpha)=4$, and $v_2((\alpha-1)\gamma_\alpha)=3v_2(\alpha-1)-1$, by \cite[Proposition 2.1]{MS},
whence $A^{3(\alpha-1)^3/2}\in Z(G)$. Suppose next that
$\alpha=1+3q$, $q\in\N$, $q\equiv -1\mod 3$, and $v_3(q+1)>1$. Then $v_p((\alpha-1)\mu_\alpha)=3v_p(\alpha-1)$
for any positive  positive prime factor $p\notin\{2,3\}$ of $\alpha-1$, $v_3((\alpha-1)\mu_\alpha)=4$,
and $v_2((\alpha-1)\mu_\alpha)\geq 3 v_2(\alpha-1)$, by \cite[Proposition 2.2]{MS}. Moreover, in this
case, we also have $v_p((\alpha-1)\gamma_\alpha)=3v_p(\alpha-1)$
for any positive  prime factor $p\notin\{2,3\}$ of $\alpha-1$,
$v_3((\alpha-1)\gamma_\alpha)\geq 5$, and $v_2((\alpha-1)\gamma_\alpha)=3v_2(\alpha-1)-1$, by \cite[Proposition 2.1]{MS}.
Since $$\gcd(3(\alpha-1)^3 2^t,3^s(\alpha-1)^3/2)=3(\alpha-1)^3/2$$ for any $s,t\in\N$, we infer that
$A^{3(\alpha-1)^3/2}\in Z(G)$ also in this case. Thus $A^{3(\alpha-1)^3/2}\in Z(G)$ in both cases.
Conjugating $A^{3(\alpha-1)^3/2}$ by $C$ we get $A^{3(\alpha-1)^4/2}=1$. Since $A^{3(\alpha-1)^3/2}\in Z(G)$,
$$
A^{3(\alpha-1)^3/2}=(A^{3(\alpha-1)^3/2})^B=
C^{3(\alpha-1)^3/2}A^{\alpha(\alpha^{3(\alpha-1)^3/2}-1)/(\alpha-1)}.
$$
Now
$$
\frac{\alpha^{3(\alpha-1)^3/2}-1}{\alpha-1}\equiv  3(\alpha-1)^3/2-3(\alpha-1)^4/4\mod 3(\alpha-1)^4/2,
$$
and therefore
$$
\alpha\frac{\alpha^{3(\alpha-1)^3/2}-1}{\alpha-1}\equiv  3(\alpha-1)^3/2-3(\alpha-1)^4/4\mod 3(\alpha-1)^4/2.
$$
As $A^{3(\alpha-1)^4/2}=1$, we deduce
$$
A^{3(\alpha-1)^3/2}=C^{3(\alpha-1)^3/2}A^{3(\alpha-1)^3/2-3(\alpha-1)^4/4}.
$$
This shows that $C^{3(\alpha-1)^3/2}=A^{3(\alpha-1)^4/4}\in Z(G)$ and $C^{3(\alpha-1)^3}=1$.
\end{proof}

\section{Nilpotence of the Macdonald group}

We are ready prove that $G$ is nilpotent. In this section, we will
write $Z=Z_1,Z_2,\cdots$ for the terms of the upper central series of $G$.

\begin{theorem}\label{nilp} The group $G(\a,\b)$ is nilpotent.
\end{theorem}

\begin{proof} As indicated in \cite[pp. 603]{M}, we may assume that $\alpha>1$ and $\beta>1$
and we make this assumption. We will repeatedly and implicitly use \cite[Eqs. (1.4) and (1.6)]{M}.

\medskip

\noindent{\sc Case I:} $\gcd(\epsilon,6)=1$. This strictly includes
the case analyzed in \cite[Section 5]{M}.

\medskip

At least one of $\a-1,\b-1$ is relatively prime to 3. Suppose first that $\b\not\equiv 1\mod 3$. Then
Lemma \ref{casos} and the transformation $A\leftrightarrow B$, $C\leftrightarrow C^{-1}$, $\a\leftrightarrow \b$ give
$$B^{(\b-1)^3}\in Z,\; B^{(\b-1)^4}=1,\;
C^{(\b-1)^3}=1.
$$
Moreover, whether $\a\equiv 1\mod 3$ or not, Lemma \ref{casos}, implies
$$A^{3(\alpha-1)^3}\in Z,\; A^{3(\alpha-1)^4}=1,\;
C^{3(\alpha-1)^3}=1.
$$
As $\gcd(\epsilon,3)=1$, we infer $C^{\epsilon^3}=1$, and therefore $A^{\a^{\epsilon^3}-1}=1=B^{\b^{\epsilon^3}-1}$.
Looking at the $p$th valuation of each prime factor $p$ of $\a-1$ and $\b-1$, we find that
$$
\gcd(3(\alpha-1)^4,\a^{\epsilon^3}-1)=\epsilon^3(\a-1),\; \gcd((\b-1)^4,\b^{\epsilon^3}-1)=\epsilon^3(\b-1),
$$
which implies $A^{\epsilon^3(\a-1)}=1=B^{\epsilon^3(\b-1)}$. The case when $\a\not\equiv 1\mod 3$ leads to the same outcome.

We claim that $A^{\epsilon^3},B^{\epsilon^3}\in Z$.
Indeed, we have
$$
(A^{\epsilon^3})^B=
C^{\epsilon^3}A^{\alpha(\alpha^{\epsilon^3}-1)/(\alpha-1)},
$$
$$
\frac{\alpha^{\epsilon^3}-1}{\a-1}=\epsilon^3+{{\epsilon^3}\choose{2}}(\alpha-1)+{{\epsilon^3}\choose{3}}(\alpha-1)^2+\cdots
$$
From $\gcd(\epsilon,6)=1$, we deduce
$
(\alpha^{\epsilon^3}-1)/(\a-1)\equiv \epsilon^3\mod \epsilon^3(\a-1),
$
and therefore we have
$
\a(\alpha^{\epsilon^3}-1)/(\a-1)(\alpha^{\epsilon^3}-1)/(\a-1)\equiv \epsilon^3\mod \epsilon^3(\a-1).
$
As $C^{\epsilon^3}=1$ and $A^{\epsilon^3(\a-1)}=1$, we infer $A^{\epsilon^3}\in Z$. Likewise
we see that $B^{\epsilon^3}\in Z$.

From $\a^{\epsilon^2}\equiv 1\mod \epsilon^3$ and $\b^{\epsilon^2}\equiv 1\mod \epsilon^3$
we deduce $C^{\epsilon^2}\in Z_2$.
We next claim that $A^{\epsilon^2},B^{\epsilon^2}$ are in $Z_3$. Indeed, we have
$$
(A^{\epsilon^2})^B=
C^{\epsilon^2}A^{\alpha(\alpha^{\epsilon^2}-1)/(\alpha-1)}.
$$
From $\gcd(\epsilon,6)=1$, we infer
$
(\alpha^{\epsilon^2}-1)/(\alpha-1)\equiv  \epsilon^2\mod \epsilon^3,
$
so
$
\alpha (\alpha^{\epsilon^2}-1)/(\alpha-1)\equiv \epsilon^2\mod \epsilon^3.
$
Since $C^{\epsilon^2}\in Z_2$ and $A^{\epsilon^3}\in Z$, it follows that
$A^{\epsilon^2}\in Z_3$. Likewise we see that $B^{\epsilon^2}\in Z_3$.

From $\alpha^{\epsilon}\equiv 1\mod \epsilon^2$
and $\beta^{\epsilon}\equiv 1\mod \epsilon^2$ we deduce $C^{\epsilon}\in Z_4$.
We next claim that $A^{\epsilon},B^{\epsilon}\in Z_5$.
Indeed, we have
$$
(A^{\epsilon})^B=
C^{\epsilon}A^{\alpha(\alpha^{\epsilon}-1)/(\alpha-1)}.
$$
From $\gcd(\epsilon,6)=1$, we infer
$
(\alpha^{\epsilon}-1)/(\alpha-1)\equiv  \epsilon\mod \epsilon^2,
$
so
$
\alpha (\alpha^{\epsilon}-1)/(\alpha-1)\equiv  \epsilon\mod \epsilon^2.
$
Since $C^{\epsilon}\in Z_4$ and $A^{\epsilon^2}\in Z_3$, it follows that
$A^{\epsilon}\in Z_5$. Likewise we can see that $B^{\epsilon}\in Z_5$.

From $\alpha\equiv 1\mod \epsilon$
and $\beta\equiv 1\mod \epsilon$ we deduce $C\in Z_6$. As $C=[A,B]\in Z_6$, we infer
$A,B\in Z_7$.

\medskip

\noindent{\sc Case II:} $\gcd(\epsilon,2)=1$ and $3\mid \epsilon$.

\medskip

Lemma \ref{casos} and the transformation $A\leftrightarrow B$, $C\leftrightarrow C^{-1}$, $\a\leftrightarrow \b$ ensure that
$$A^{3(\alpha-1)^3}, B^{3(\b-1)^3}\in Z,\;A^{3(\alpha-1)^4}=1=B^{3(\b-1)^4},\;
C^{3(\alpha-1)^3}=1=C^{3(\b-1)^3}.
$$
It follows that $C^{3\epsilon^3}=1$. Therefore $A^{\a^{3\epsilon^3}-1}=1=B^{\b^{3\epsilon^3}-1}$.
Looking at the $p$th valuation of each prime factor $p$ of $\a-1$ and $\b-1$, we find that
$$
\gcd(3(\alpha-1)^4,\a^{3\epsilon^3}-1)=3\epsilon^3(\a-1),\; \gcd(3(\b-1)^4,\b^{3\epsilon^3}-1)=3\epsilon^3(\b-1),
$$
which implies $A^{3\epsilon^3(\a-1)}=1=B^{3\epsilon^3(\b-1)}$. Arguing as in Case I, we successively deduce:
$$
A^{3\epsilon^3},B^{3\epsilon^3}\in Z; C^{3\epsilon^2}\in Z_2; A^{3\epsilon^2},B^{3\epsilon^2}\in Z_3;
C^{3\epsilon}\in Z_4; A^{3\epsilon},B^{3\epsilon}\in Z_5;
$$
$$
C^3\in Z_6;
A^3,B^3\in Z_7; C\in Z_8; A,B\in Z_9.
$$

\medskip

\noindent{\sc Case III:} $\gcd(\epsilon,3)=1$ and $2|\epsilon$.

\medskip

Exactly the same argument given in Case I shows that $C^{\epsilon^3}=1$ and
$A^{\epsilon^3(\a-1)}=1=B^{\epsilon^3(\b-1)}$. We claim that $A^{2\epsilon^3},B^{2\epsilon^3}\in Z$.
Indeed, we have
$$
\frac{\alpha^{2\epsilon^3}-1}{\a-1}=2\epsilon^3+{{2\epsilon^3}\choose{2}}(\alpha-1)+{{2\epsilon^3}\choose{3}}(\alpha-1)^2+\cdots
$$
Making use of $2\mid \epsilon$ and $\gcd(\epsilon,3)=1$, we see that
$$
\frac{\alpha^{2\epsilon^3}-1}{\a-1}\equiv 2\epsilon^3\mod \epsilon^3(\a-1),
$$
and therefore
$$
\a\frac{\alpha^{2\epsilon^3}-1}{\a-1}\equiv 2\epsilon^3\mod \epsilon^3(\a-1).
$$
Now
$$
(A^{2\epsilon^3})^B=
C^{2\epsilon^3}A^{\alpha(\alpha^{2\epsilon^3}-1)/(\alpha-1)},
$$
where $C^{\epsilon^3}=1$, and $A^{\epsilon^3(\a-1)}=1$, so $A^{2\epsilon^3}\in Z$. Likewise
we see that $B^{2\epsilon^3}\in Z$. Arguing as in Case~I,
we successively obtain:
$$
C^{2\epsilon^2}\in Z_2; A^{4\epsilon^2},B^{4\epsilon^2}\in Z_3; C^{4\epsilon}\in Z_4; A^{8\epsilon},B^{8\epsilon}\in Z_5;
C^8\in Z_6; A^{16},B^{16}\in Z_7.
$$
We may
now appeal to \cite[Lemma 6.1]{MS} and the finiteness of $G$ to conclude that $G/Z_7$ is a finite 2-group.
Thus $G/Z_7$ is nilpotent, and therefore $G$ is nilpotent.

\medskip

\noindent{\sc Case IV:} $6|\epsilon$.

\medskip

Lemma \ref{casos} and the transformation $A\leftrightarrow B$, $C\leftrightarrow C^{-1}$, $\a\leftrightarrow \b$ ensure that
$$
A^{3(\alpha-1)^3}\in Z,\; A^{3(\alpha-1)^4}=1,\;
C^{3(\alpha-1)^3}=1,\; B^{3(\b-1)^3}\in Z,\; B^{3(\b-1)^4}=1,\;
C^{3(\b-1)^3}=1.
$$
Arguing as in Case~I,
we successively obtain:
$$
C^{3\epsilon^3}=1; A^{\a^{3\epsilon^3}-1}=1=B^{\b^{3\epsilon^3}-1};
A^{3\epsilon^3(\a-1)}=1=B^{3\epsilon^3(\b-1)}; A^{6\epsilon^3},B^{6\epsilon^3}\in Z;
C^{6\epsilon^2}\in Z_2;
$$
$$
A^{12\epsilon^2},B^{12\epsilon^2}\in Z_3; C^{12\epsilon}\in Z_4;
A^{24\epsilon},B^{24\epsilon}\in Z_5; C^{24}\in Z_6;
A^{48},B^{48}\in Z_7; C^{16}\in Z_8; A^{32},B^{32}\in Z_9.
$$
We may now appeal to \cite[Lemma 6.1]{MS} and the finiteness of $G$ to conclude that $G/Z_9$ is a finite 2-group.
Thus $G/Z_9$ is nilpotent, and therefore $G$ is nilpotent.
\end{proof}

By Theorem \ref{nilp}, $G(\a,\b)$ is the direct product of its Sylow subgroups, so we have a canonical projection
$\pi:G(\a,\b)\to G(\a,\b)_p$, and we set $a=A^\pi$, $b=B^\pi$, and $c=C^\pi$.
Expressions such as $a^{p^{\infty}}$ or $a^{p^{m+\infty}}$ will be interpreted as 1.

Theorem \ref{nilp} and Proposition \ref{preex} below can be used to obtain a presentation for $G_p=G(\a,\b)_p$.

\begin{prop}\label{preex} Let $T=\langle X\,|\,  R\rangle$ be a finite nilpotent group.
For $x\in X$, suppose that $a_x>0$ and $x^{a_x}\in \overline{R}$, the normal closure of $R$ in the free group $F(X)$.
Set $V=\{x^{v_p(a_x)}\,|\, x \in X\}$. Then the Sylow $p$-subgroup
of $T$ has presentation $\langle X\,|\,  R\cup V\rangle$.
\end{prop}

\begin{proof} See \cite[Corollary 5.2]{MS}. Alternatively,
$\langle X\,|\,  R\cup V\rangle$ is a finite nilpotent group, hence it has a projection
onto its Sylow $p$-group. This projection is the identity, as it maps each given generator to itself,
whence $\langle X\,|\,  R\cup V\rangle$ is a finite $p$-group.
Thus, the map $\langle X\,|\,  R\rangle\to \langle X\,|\,  R\cup V\rangle$
is trivial on all other Sylow subgroups, yielding an epimorphism from the Sylow $p$-sugroup of $\langle X\,|\,  R\rangle$
to $\langle X\,|\,  R\cup V\rangle$. The definition of $V$ allows us to define an epimorphism in the opposite direction.
\end{proof}

\begin{cor}\label{pres} Suppose $A^{p^r g}=1=B^{p^s h}$ holds in $G(\a,\b)$, where $r,s\geq 0$, $p\nmid g$, and $p\nmid h$
(this means
$a^{p^r}=1=b^{p^s}$). Then $G_p$ has presentation $\langle a,b\,|\, a^{[a,b]}=a^{\a},\, b^{[b,a]}=b^{\b}, a^{p^r}=1=b^{p^s}\rangle$.
\end{cor}

\begin{theorem}\label{finite} The following statements hold:

(a) $p\mid o(A)\Leftrightarrow p\mid (\alpha-1)$, in which case $p^m|o(A)$.

(b) $p\mid o(B)\Leftrightarrow p\mid (\beta-1)$, in which case $p^n|o(B)$.

(c) $p\mid o(C)\Leftrightarrow p\mid (\alpha-1)$ and $p\mid (\beta-1)$.

(d) $G(\a,\b)$ is the product of the subgroups $\langle A\rangle$, $\langle B\rangle$, $\langle C\rangle$ in any fixed order.
In particular, $p\mid |G(\a,\b)| \Leftrightarrow p\mid (\alpha-1)(\b-1)$. Moreover, if $p\nmid (\b-1)$
(resp. $p\nmid (\a-1)$) then $G_p$ is cyclic of order $p^m$ (resp. $p^n$).

(e) $G(\a,\b)$ is cyclic if and only if $\gcd(\a-1,\b-1)=1$, in which case $|G(\a,\b)|=|(\a-1)(\b-1)|$.
\end{theorem}

\begin{proof} (a) There is clearly an epimorphism $G(\a,\b)\to C_{p^m}$, which shows that $p^m|o(A)$.
The fact that $p\mid o(A)$ implies $p\mid (\alpha-1)$ takes considerable effort and is elegantly proven in \cite[Section 4]{M}.

(b) This follows from part (a) via the isomorphism $G(\a,\b)\leftrightarrow G(\b,\a)$.

(c) If $p\mid o(C)$ then $p\mid (\alpha-1)$ and $p\mid (\beta-1)$, by parts (a) and (b), and \cite[p. 603]{M}. The converse
follows by defining epimorphism from $G(\a,\b)$ onto the Heisenberg group over $\Z/p\Z$.

(d) The first statement follows from \cite[Lemma 6.1]{MS}. This and part (c) imply the second statement.
As for third statement, the epimorphism  $G(\a,\b)\to C_{p^m}$ of part (a) yields an epimorphism $G_p\to C_{p^m}$
by Theorem \ref{nilp}. Suppose $p\nmid (\b-1)$. Then $G_p=\langle a\rangle$ by part (b). Since $a=a^\a$,
it follows that $o(a)\mid p^m$, whence $o(a)=p^m$. The case when $p\nmid (\a-1)$ is handled similarly.

(e) This follows from Theorem \ref{nilp} together with parts (c) and (d).
\end{proof}

By Theorem \ref{finite}, in our study of $G_p$ we may assume that $p$ is a common factor of $\a-1$ and $\b-1$,
that is, $m>0$ and  $n>0$, and we do so for the remainder of the paper.

We assume until Section \ref{larga} inclusive that $\a,\b>1$.
Applying the projection $\pi:G\to G_p$ we see all the relations obtained in Section \ref{secfr}, after
Proposition \ref{nuevaca} and before Lemma \ref{casos}, remain valid when $A,B,C$
are replaced by $a,b,c$. In particular, we will make extensive use of
\begin{equation}\label{prin1}
b^{\b_0^{\b+1}\d_\b}a^{\d_\a}=c^{\a-\b}=a^{\a_0^{\a+1} \d_\a}b^{\d_\b},
\end{equation}
which is a consequence of (\ref{adabc}) and (\ref{bdac}), as well as of
\begin{equation}\label{prin2}
b^{\l_\b}c^{\d_\b}=a^{\a^{\a-\b}-1},\; a^{\l_\a}c^{-\d_\a}=b^{\b^{\b-\a}-1},
\end{equation}
which follows from (\ref{bdac2}) and (\ref{adabc2}), and
\begin{equation}\label{prin3}
G_p=\langle a\rangle\langle b\rangle\langle c\rangle=\langle a\rangle\langle c\rangle\langle b\rangle,
\end{equation}
which is a consequence of Theorem \ref{finite}.

We will write $Z=Z_1, Z_2, Z_3,\dots$ for the terms of the upper central series of $G_p$.

The following well-known gadget (cf. \cite[Chapter III, Section 7]{Z}) will be used repeatedly and implicitly to construct
homomorphic images of $G_p$ of suitable orders.

\begin{theorem}\label{Z} Let $T$ be an arbitrary group and $L$ a cyclic group of finite order $n\in\N$. Suppose that $t\in T$
and that $\Omega$ is an automorphism of $T$ fixing $t$ and such that $\Omega^n$ is conjugation by $t$. Then there
is a group $E$ containing $T$ as a normal subgroup, such that $E/T\cong L$, and for some $g\in E$ of order $n$ modulo $T$, we have
$g^n=t$ and $\Omega$ is conjugation by $g$.
\end{theorem}

\section{Generalities of the case when $p>3$, or $p=3$ and $\alpha,\beta\not\equiv 7\mod 9$}\label{case1}

We assume throughout this section that $p>3$, or that $p=3$ and that neither $\alpha$ nor $\beta$ is congruent to 7 modulo 9.
It follows from (\ref{lomismo}) and \cite[Propisiton 2.1]{MS} that
\begin{equation}\label{val}
v_p(\d_\a)=2m,\; v_p(\d_\b)=2n,
\end{equation}
so by (\ref{adabc4}),
\begin{equation}\label{cen}
a^{p^{3m}},b^{p^{3n}}\in Z,
\end{equation}
while (\ref{adabc5}) yields
\begin{equation}\label{cen2}
a^{p^{4m}}=1=b^{p^{4n}}.
\end{equation}
Combining (\ref{cen}) and (\ref{cen2}) we obtain
\begin{equation}\label{cen3}
c^{p^{3m}}=1=c^{p^{3n}}.
\end{equation}
Indeed, by \cite[Eqs. (1.4) and (1.6)]{M}, we have
\begin{equation}\label{unocua}
a^{p^{3m}}=(a^{p^{3m}})^b=c^{p^{3m}}a^{\a(1+\a+\cdots+\a^{p^{3m}-1})}=c^{p^{3m}}a^{p^{3m}},
\end{equation}
\begin{equation}\label{unocuar}
b^{p^{3n}}=(b^{p^{3n}})^a=c^{-p^{3n}}b^{\b(1+\b+\cdots+\b^{p^{3n}-1})}=c^{-p^{3n}}b^{p^{3n}},
\end{equation}
using
\begin{equation}\label{unocuat}
\a(\a^{p^{3m}}-1)/(\a-1)\equiv p^{3m}\mod p^{4m},\; \b(\b^{p^{3n}}-1)/(\b-1)\equiv p^{3n}\mod p^{4n}.
\end{equation}

On the other hand, a routine calculation that makes use of (\ref{val}) shows that
\begin{equation}\label{val2}
v_p(\l_\a)\geq 3m,\; v_p(\l_\b)\geq 3n.
\end{equation}
It follows from (\ref{cen}) and (\ref{val2}) that
$$
a^{\l_\a} \in Z,\; b^{\l_\b}\in Z.
$$
Thus, the operators $[a,-]$ and $[b,-]$ applied to (\ref{prin2}) yield
$$
a^{\a^{\d_\b}-1}=1,\; b^{\b^{\d_\a}-1}=1.
$$
Here
$$
v_p(\a^{\d_\b}-1)=m+2n,\; v_p(\b^{\d_\a}-1)=n+2m,
$$
so
$$
a^{p^{m+2n}}=1,\; b^{p^{n+2m}}=1,
$$
$$
[c^{p^{2n}},a]=1=[c^{p^{2m}},b].
$$
In view of the isomorphism $G(\a,\b)\cong G(\b,\a)$, we may assume without loss that $m\geq n$. Then
\begin{equation}\label{p4}
c^{p^{2m}}\in Z, a^{p^{3m}}=1.
\end{equation}
From (\ref{val2}) and (\ref{p4}) we deduce
$$
a^{\l_\a}=1,
$$
and therefore (\ref{prin2}) gives
\begin{equation}\label{q}
c^{-\d_\a}=b^{\b^{\b-\a}-1}.
\end{equation}
If $m=n$, then the same argument yields
\begin{equation}\label{qqz}
c^{\d_\b}=a^{\a^{\a-\b}-1}.
\end{equation}

Since $v_p(\d_\a)=2m$ and $c^{p^{2m}}\in Z$, it follows from (\ref{q}) that
$
b^{\b^{\b-\a}-1}\in Z,
$
whence
\begin{equation}\label{p5}
b^{p^{n+\ell}}\in Z.
\end{equation}
Therefore, the operator $[c^{-1},-]$ gives
\begin{equation}\label{p6}
b^{p^{2n+\ell}}=1.
\end{equation}


\section{The case when $\ell=n$}\label{elln}

We maintain the hypotheses of Section \ref{case1} and assume further that $\ell=n$. Then (\ref{p5}) and
(\ref{p6}) become
$
b^{p^{2n}}\in Z, b^{p^{3n}}=1.
$
From these two relations, we derive 
$
c^{p^{2n}}=1,
$
through slight modifications of (\ref{cen})-(\ref{unocuat}).
This implies $c^{p^{2m}}=1$, which together with $a^{p^{3m}}=1$ yield
$
a^{p^{2m}}\in Z,
$
using minor variations of (\ref{unocua}) and (\ref{unocuat}). Since $a^{p^{2m}}\in Z$ and $b^{p^{2n}}\in Z$, we see from (\ref{prin1}) that
$
c^{p^n}\in Z,
$
which implies
\begin{equation}\label{dq}
a^{p^{m+n}}=1=b^{p^{2n}}.
\end{equation}
Going back to (\ref{prin1}) we now see that
\begin{equation}\label{dq2}
c^{p^n}=1.
\end{equation}
From (\ref{dq}) and (\ref{dq2})  we easily obtain
$
a^{p^n}, b^{p^n}\in Z,
$
whence
$
c\in Z_2, Z_3=G_p.
$
It follows from (\ref{prin3}), (\ref{dq}), and (\ref{dq2}) that $|G_p|\leq p^{4n+m}$
and the class of $G_p$ is at most 3.


\begin{theorem}\label{teo1} If $m\geq n=\ell$, then $e=4n+m$, $f=3$, $o(a)=p^{m+n}$, $o(b)=p^{2n}$,
and $o(c)=p^{n}$.
\end{theorem}

\begin{proof} We first show that $e=4n+m$. Since $e\leq 4n+m$, it suffices to construct a homomorphic image
of $G_p$ of order $p^{4n+m}$. We begin with a group $T=\langle X,Y,Z\rangle$ of order $p^{3n}$
having defining relations $[X,Y]=[X,Z]=[Y,Z]=1$ and $X^{p^n}=Y^{p^n}=Z^{p^n}=1$. Here $X,Y,Z$ play the roles of $a^{p^m},
b^{p^n},c$, respectively. The assignment $X\mapsto X$, $Y\mapsto Y$, $Z\mapsto Z X^{-u}$ extends to an automorphism $\Omega$
of $T$ (which plays the role of conjugation by $a$) that fixes $X$ and such that $\Omega^{p^m}$ is conjugation by $X$, namely
trivial. Let $E=\langle X_0,Y,Z\rangle$ be the group arising from
Theorem \ref{Z}, so that $E/T\cong C_{p^m}$, $X_0$ has order $p^m$ modulo $T$, $X_0^{p^m}=X$, and $\Omega$ is conjugation by $X_0$.
Then $|E|=p^{m+3n}$, with defining relations $Z^{X_0}=Z X_0^{1-\alpha}$, $[X_0,Y]=[Y,Z]=1$, $X^{p^{m+n}}=Y^{p^n}=Z^{p^n}=1$.
The assignment $X_0\mapsto X_0Z$, $Y\mapsto Y$, $Z\mapsto Y^v Z$ extends to an automorphism $\Psi$
of $E$ (which plays the role of conjugation by $b$) that fixes $Y$ and such that $\Psi^{p^n}$ is conjugation by $Y$.
Let $F=\langle X_0,Y_0,Z\rangle$ be the group arising from
Theorem \ref{Z}, so that $F/E\cong C_{p^n}$, $Y_0$ has order $p^n$ modulo $E$, $Y_0^{p^n}=Y$, and $\Psi$ is conjugation by $Y_0$.
Then $|F|=p^{m+4n}$, $Z=[X_0,Y_0]$, $X_0^Z=X_0^{\a}$ and ${}^Z Y_0=Y^\b$. Thus the $p$-group $F$
is an image of $G(\a,\b)$, and and hence of $G_p$, by Theorem \ref{nilp}.

This shows that $|G_p|=p^{4n+m}$, which clearly implies that the orders of $a,b,c$ are correct. It follows that $f=3$.
Because if $c\in Z$, then $b^{p^n}=1$, against $v_p(o(b))=2n$. Thus, $c\in Z_2\setminus Z$. If $a\in Z_2$ then $c\in Z$, which is false, so
$a\in Z_3\setminus Z_2$.
\end{proof}

In the proof of subsequent theorems, we will just construct an image of $G(\a,\b)$
that is a finite $p$-group of the required order, as all assertions will follow immediately from this.

\section{Preliminary observations of the case $m=n<\ell$}\label{ellmn}

We maintain the hypotheses of Section \ref{case1} and assume further that $m=n<\ell$. As $m\geq n$ and $n\geq m$, it follows from Section 
\ref{case1} that
\begin{equation}
\label{haim}
a^{p^{3m}}=1=b^{p^{3m}}, c^{p^{2m}}\in Z, b^{p^{m+\ell}}\in Z,a^{p^{m+\ell}}\in Z.
\end{equation}
From $a^{p^{3m}}=1=b^{p^{3m}}$ and $\ell>m=n$, we infer $a^{p^{2m+\ell}}=1=b^{p^{2m+\ell}}$. This and
$a^{p^{m+\ell}}, b^{p^{m+\ell}}\in Z$ readily give 
\begin{equation}
\label{haim2}
c^{p^{m+\ell}}=1,
\end{equation}
via conjugation by $a$ or $b$. On the other hand, by \cite[Eq. (1.6)]{M}, we have
\begin{equation}
\label{conjbya0}
(b^{p^m})^a=c^{-p^m} b^{\b(\b^{p^m}-1)/(\b-1)},
\end{equation}
and a routine calculation yields
\begin{equation}
\label{conjbya}
\b(\b^{p^m}-1)/(\b-1)\equiv\begin{cases} p^m(1+vp^m(p^m+1)/2)\mod p^{3m} & \text{ if }p>3,\\
3^m(1+v3^m(3^m+1)/2+3^{2m-1})\mod 3^{3m} & \text{ if }p=3,\end{cases}
\end{equation}
where we have used that $v^2\equiv 1\mod 3$. From (\ref{prin1}) we deduce
\begin{equation}\label{y}
b^{p^{2m}}\in \langle a\rangle\langle c\rangle.
\end{equation}
As $m=n$, (\ref{qqz}) is valid, so 
\begin{equation}\label{y2}
c^{p^{2m}}\in \langle a\rangle.
\end{equation}
We infer from  (\ref{prin3}), (\ref{haim}), (\ref{y}), and (\ref{y2}) that
\begin{equation}\label{cotasup}
|G_p|\leq p^{7m}.
\end{equation}
This bound is actually reached in certain cases below, so a further analysis is  required to sharpen it in other cases. Regarding the
upper central series of $G_p$, we already know from (\ref{haim}) that
\begin{equation}\label{y5}
a^{p^{m+\ell}}, b^{p^{m+\ell}},c^{p^{2m}}\in Z,
\end{equation}
where by (\ref{q}) and (\ref{qqz}), these elements generate the same subgroup, that is
\begin{equation}\label{haim3}
\langle a^{p^{m+\ell}}\rangle=\langle b^{p^{m+\ell}}\rangle=\langle c^{p^{2m}}\rangle.
\end{equation}

It follows easily from (\ref{y5}) that
\begin{equation}\label{y6}
a^{p^{2m}}, b^{p^{2m}},c^{p^{\ell}}\in Z_2.
\end{equation}
From (\ref{y5}) and (\ref{y6}), we deduce
\begin{equation}\label{y7}
a^{p^{\ell}}, b^{p^{\ell}},c^{p^m}\in Z_3.
\end{equation}
By means of (\ref{y6}) and (\ref{y7}), we now infer
\begin{equation}\label{y8}
a^{p^{m}}, b^{p^{m}}\in Z_4.
\end{equation}
Finally, (\ref{y8}) yields
\begin{equation}\label{y9}
c\in Z_5,\; Z_6=G_p.
\end{equation}
Thus the nilpotency class of $G_p$ is at most 6. This bound is actually reached in certain cases below, so a further analysis is  required to sharpen it in other cases. In these latter cases,  all subgroups listed in (\ref{haim3}) are trivial, and the nilpotency class of $G_p$ is actually equal to 5.

Recall from (\ref{prin1}) that $b^{\b_0^{\b+1}\d_\b}a^{\d_\a}=c^{\a-\b}=a^{\a_0^{\a+1} \d_\a}b^{\d_\b}$, where $\a_0,\b_0$
are defined in Section \ref{secfr}, and satisfy $\a\a_0\equiv 1\mod o(a)$ and 
$\b\b_0\equiv 1\mod o(b)$. But $\a,\b\equiv 1\mod p^m$, where
$p^m\mid o(a)$ and $p^m\mid o(b)$ by Theorem \ref{finite}, so $\a_0,\b_0\equiv 1\mod p^m$.  Since
$\d_\a,\d_\b\equiv 0\mod p^{2m}$, $a^{p^{3m}}=1=b^{p^{3m}}$, we deduce from (\ref{prin1}) that
\begin{equation}\label{abc}
a^{\d_\a}b^{\d_\b}=c^{\a-\b}=b^{\d_\b}a^{\d_\a},
\end{equation}
In particular, $[a^{p^{2m}},b^{p^{2m}}]=1$. Select $w_\a,w_\b\in\Z$ so that
\begin{equation}\label{pjk2}
2w_\a\equiv \begin{cases} u^2\mod p^{m}\text{ if }p>3,\\u^2-2\times  3^{m-1}u\mod 3^{m}\text{ if }p=3,\end{cases}
2w_\b\equiv \begin{cases} v^2\mod p^{m}\text{ if }p>3,\\v^2-2\times 3^{m-1}v\mod 3^{m}\text{ if }p=3,\end{cases}
\end{equation}
noting that $p\nmid w_\a$ and $p\nmid w_\b$, even in the extreme case $p=3$ and $m=1$, in which case 
the hypothesis $\a,\b\not\equiv 7\mod 9$ is required to reach this conclusion. Then
\begin{equation}\label{2casos}
\d_\a\equiv p^{2m} w_\a\mod p^{3m},\; a^{\d_\a}=a^{p^{2m} w_\a},\; \d_\b\equiv p^{2m} w_\b\mod p^{3m},\; b^{\d_\b}=b^{p^{2m} w_\b},
\end{equation}
so by (\ref{abc}) 
\begin{equation}\label{model1r1}
c^{p^{\ell}k}=a^{p^{2m} w_\a}b^{p^{2m} w_\b}.
\end{equation}
As $u\equiv v\mod p^{\ell-m}$, we have $3^{m-1}u\equiv 3^{m-1}v\mod 3^m$, so these terms can be used interchangeably  in (\ref{pjk2})
without affecting (\ref{model1r1}). 
If $b^{p^{\ell+m}}=1$ (which means $a^{p^{\ell+m}}=1$ or, equivalently $c^{p^{2m}}=1$), then 
$w_\a$ and $w_\b$ can also be used interchangeably in (\ref{model1r1}) without affecting it.

It will be convenient to set
\begin{equation}\label{hache}
h=\begin{cases} p^m v/2 & \text{ if }p>3,\\ 3^m v/2+3^{2m-1} & \text{ if }p=3.
\end{cases}
\end{equation}

\begin{prop}\label{abcabelian} Let $H$ be a group with elements $x_1,x_2,x_3$ and an automorphism $\Psi$ such that
for $t=1+h$, with $h$ as in (\ref{hache}), and for some integer $0\leq g\leq m$, we have
$$
x_1^{x_3}=x_1^\a, {}^{x_3} x_2=x_2^\b, x_1^{p^{2m}}=x_2^{p^{2m-g}}=1,
$$
$$
x_1^{\Psi}=x_1, x_2^\Psi=x_3^{-p^{m+g}}x_2^{t}, x_3^{\Psi}=x_3x_1^{-u}.
$$
Then $[x_1,x_3^{p^{m}}]=1=[x_2,x_3^{p^{m}}]$, $(x_3^{-p^m})^\Psi=x_1^{p^m u} x_3^{-p^m}$, $x_3^{\Psi^{p^m}}=x_3^{x_1}$, and
\begin{equation}\label{x_p123}
x_2^{\Psi^i}=x_1^{p^{m+g}u(i-1)i/2} x_3^{-p^{m+g}(1+t+\cdots+t^{i-1})} x_2^{t^i},\quad i\geq 1.
\end{equation}
In particular, if $x_3^{p^{2m+g}}=1$, then $x_2^{\Psi^{p^m}}=x_2$, so if $c^{p^{2m+g}}=1$,
then $\langle a^{p^{m}}, b^{p^{m+g}}, c^{p^{m}}\rangle$ is a normal abelian subgroup of $G_p$.
\end{prop}

\begin{proof} Since $\a^{p^{m}},\b^{p^{m}}\equiv 1\mod p^{2m}$, it follows that $[x_1,x_3^{p^{m}}]=1=[x_2,x_3^{p^{m}}]$.
Note that $(x_3^{p^m})^\Psi=(x_3x_1^{-u})^{p^m}= x_3^{p^m}x_1^{-u(\a^{p^m}-1)/(\a-1)}$, where $(\a^{p^m}-1)/(\a-1)\equiv p^m\mod p^{2m}$,
so $(x_3^{-p^m})^\Psi=x_1^{p^m u} x_3^{-p^m}$. Also, $x_3^{\Psi^{p^m}}=x_3 x_1^{-p^m u}=x_3^{x_1}$.
We prove (\ref{x_p123}) by induction. The case $i=1$ is true by hypothesis.
Suppose (\ref{x_p123}) holds for some $i\geq 1$. Then, using $[x_1,x_3^{p^{m}}]=1=[x_2,x_3^{p^{m}}]$,
$(x_3^{-p^{m+g}})^\Psi=x_1^{p^{m+g} u} x_3^{-p^{m+g}}$,
and the effect of $\Psi$ on $x_1,x_2,x_3$, we deduce
$$
x_2^{\Psi^{i+1}}=x_1^{p^{m+g}u(i-1)i/2} x_1^{p^{m+g} u(1+t+\cdots+t^{i-1})} x_3^{-p^{m+g}(1+t+\cdots+t^{i-1})}
x_3^{-p^{m+g} t^i} x_2^{t^{i+1}}.
$$
Since $t\equiv 1\mod p^{m}$ and $x_1^{p^{2m}}=1$, we have $x_1^{p^{m+g} u(1+t+\cdots+t^{i-1})}=x_1^{p^{m+g} u i}$,  so
$$
x_2^{\Psi^{i+1}}=x_1^{p^{m+g}ui(i+1)/2} x_3^{-p^{m+g}(1+t+\cdots+t^{i-1}+t^i)}x_p^{t^{i+1}},
$$
which completes the proof of (\ref{x_p123}). Making use of
$(t^{p^m}-1)/(t-1)\equiv 0\mod p^m$, $t^{p^m}\equiv 1\mod p^{2m}$, and (\ref{x_p123}), we deduce that if  $x_3^{p^{2m+g}}=1$, then
$x_2^{\Psi^{p^m}}=x_2$.

Take $H=G_p$ and $x_1=a^{p^m}$, $x_2=b^{p^{m+g}}$,  $x_3=c$, and $\Psi$ conjugation by $a$,
and suppose that $c^{p^{2m+g}}=1$. It follows
easily from \cite[Eqs. (1.4) and (1.6)]{MS} that $(a^{p^{m}})^b\in \langle a^{p^{m}}, c^{p^{m}}\rangle$
and $(b^{p^{m+g}})^a\in \langle b^{p^{m+g}}, c^{p^{m}}\rangle$. Moreover, $(c^{p^{m}})^a=c^{p^{m}} a^{1-\a^{p^{m}}}$
and $(c^{p^{m}})^b=b^{\b^{p^{m}-1}}c^{p^{m}}$, where $1-\a^{p^{m}}\equiv 0\mod p^{2m}$ and $\b^{p^{m}}-1\equiv 0\mod p^{2m}$.
Since $g\leq m$, the conjugates of $c^{p^{m}}$ by $a$ and $b$ are also in $\langle a^{p^m}, b^{p^{m+g}}, c^{p^{m}}\rangle$,
so this is a normal subgroup of $G_p$. As $b^{p^{3m}}=1$, it follows from (\ref{conjbya0}) and (\ref{conjbya}) that
$(b^{p^{m+g}})^a=c^{-p^{m+g}}b^{p^{m+g}t}$.
We also have $a^{p^{3m}}=1$ and $c^{a}=c [c,a]=c a^{1-\a}=c (a^{p^m})^{-u}$. As all hypotheses imposed on $x_1,x_2,x_3$ and $\Psi$ are met,
it follows that $\langle a^{p^m}, b^{p^{m+g}}, c^{p^{m}}\rangle$ is abelian.
\end{proof}


\section{The case when $m=n$ and $m<\ell<2m$}\label{pobs}

We maintain the hypotheses of Sections \ref{case1} and \ref{ellmn} and assume further that
$m=n$ and $m<\ell<2m$, recalling that $\a-\b=p^\ell k$, where $p\nmid k$, so that
$u-v=p^{\ell-m}k$.

Raising (\ref{abc}) to the $p^{(2m-\ell)}$th power and making use of (\ref{pjk2})-(\ref{model1r1}), we obtain
\begin{equation}\label{abc2}
c^{p^{2m}}=a^{p^{4m-\ell} u^2/2k}b^{p^{4m-\ell} v^2/2k},
\end{equation}
regardless of whether $p>3$ or $p=3$. By (\ref{haim2}), (\ref{2casos}), $u\equiv v\mod p^{\ell-m}$, and $\ell<2m$, we have
\begin{equation}\label{obt}
c^{\d_\a}=c^{p^{2m} u^2/2}=c^{p^{2m} v^2/2}=c^{\d_\b},
\end{equation}
regardless of whether $p>3$ or $p=3$. On the other hand, whether $\a>\b$ or $\a<\b$, we see that
\begin{equation}\label{obt2}
a^{\a^{\a-\b}-1}=a^{p^{m+\ell} uk},\; b^{\b^{\b-\a}-1}=b^{-p^{m+\ell} vk}.
\end{equation}
It now follows from (\ref{q}), (\ref{qqz}), (\ref{haim2}), (\ref{obt}), (\ref{obt2}), and $u\equiv v\mod p^{\ell-m}$ that
\begin{equation}\label{los33}
a^{p^{m+\ell} 2k}=c^{p^{2m}u}=c^{p^{2m}v}=b^{p^{m+\ell} 2k},
\end{equation}
whence
\begin{equation}\label{los3}
a^{p^{m+\ell}}=b^{p^{m+\ell}}.
\end{equation}

Suppose next that $2\ell=3m$. Since $u\equiv v\mod p^{\ell-m}$ and $2\ell=3m$, then (\ref{abc2}) and (\ref{los33}) give
\begin{equation}\label{los4}
a^{p^{m+\ell} u^2/2k}b^{p^{m+\ell} u^2/2k}=c^{p^{2m}}=a^{p^{m+\ell} 2k/u}=b^{p^{m+\ell} 2k/u}.
\end{equation}
From (\ref{los4}), we obtain
\begin{equation}\label{abc14}
b^{p^{m+\ell} u^2/2k}=a^{p^{m+\ell}(2k/u-u^2/2k)},\; a^{p^{m+\ell} u^2/2k}=b^{p^{m+\ell}(2k/u-u^2/2k)}.
\end{equation}
Raising (\ref{abc14}) to the $(2uk)$th power gives
\begin{equation}\label{abc15}
b^{p^{m+\ell} u^3}=a^{p^{m+\ell}(4k^2-u^3)},\; a^{p^{m+\ell} u^3}=b^{p^{m+\ell}(4k^2-u^3)}.
\end{equation}
From (\ref{abc15}), we derive
\begin{equation}\label{abc16}
a^{p^{m+\ell} u^6}=b^{p^{m+\ell}u^3(4k^2-u^3)}=a^{p^{m+\ell}(4k^2-u^3)^2},\;
b^{p^{m+\ell} u^6}=a^{p^{m+\ell}u^3(4k^2-u^3)}=b^{p^{m+\ell}(4k^2-u^3)^2}.
\end{equation}
We deduce from  (\ref{abc16}) that
\begin{equation}\label{abc17}
a^{p^{m+\ell} (u^3-2k^2)}=1=b^{p^{m+\ell} (u^3-2k^2)}.
\end{equation}
Let $s=v_p(u^3-2k^2)$. Then (\ref{abc17}) gives
\begin{equation}\label{len}
a^{p^{s+5m/2}}=a^{p^{m+\ell+s}}=1=b^{p^{m+\ell+s}}=b^{p^{s+5m/2}},
\end{equation}
so (\ref{haim3}) and (\ref{len}) yield
\begin{equation}\label{len2}
c^{p^{2m+s}}=1.
\end{equation}
Recall that $|G_p|\leq p^{7m}$ by (\ref{cotasup}). But we also have
$G_p=\langle a\rangle\langle c\rangle\langle b\rangle$ by (\ref{prin3}), $b^{p^{2m}}\in \langle a\rangle\langle c\rangle$ by (\ref{y}),
and $c^{p^{2m}}\in \langle a\rangle$ by (\ref{y2}), so (\ref{len}) gives
\begin{equation}\label{len3}
|G_p|\leq p^{5m+\ell+s}=p^{13m/2+s}.
\end{equation}
In particular, if $s=0$, then (\ref{len}), (\ref{len2}), and (\ref{len3}) yield
$$
a^{p^{m+\ell}}=b^{p^{m+\ell}}=c^{p^{2m}}=1,\; |G_p|\leq p^{13m/2},
$$ 
and the class of $G_p$ is at most 5 in this case.

\begin{theorem}\label{teo5} Suppose that $n=m$ and $2\ell=3m$. Then $f=5$ if $s=0$ and $f=6$ if $s>0$. Moreover, if $0\leq s\leq m/2$ then
$e=s+13m/2$, $o(a)=p^{s+5m/2}=o(b)$, and
$o(c)=p^{2m+s}$, while if $s\geq m/2$ then $e=7m$, $o(a)=p^{3m}=o(b)$, and
$o(c)=p^{5m/2}$.
\end{theorem}

\begin{proof} Set $q=\min\{s, m/2\}$ and consider the abelian group generated by $x,y,z$ subject to the defining relations
$
[x,y]=[x,z]=[y,z]=1=x^{p^{\ell+q}},
$
as well as
$$
z^{p^{m/2}k}=x^{p^{m}w_\a}y^{p^{m/2}w_\b},\; x^{2 p^{\ell}k}=z^{p^{m}u},\; x^{p^{\ell}}=y^{p^m},
$$
where $x,y,z$ play the roles of $a^{p^m}, b^{p^\ell}, c^{p^{m}}$, respectively. This is a valid choice thanks to
Proposition~\ref{abcabelian} (applied with $g=m/2$). The displayed relations are modeled upon (\ref{model1r1}), (\ref{los33}),
and (\ref{los3}), respectively.  The given relations force $x^{p^{2m}}=y^{p^{m+q}}=z^{p^{m+q}}=1$.

We claim that $\langle x,y,z\rangle$ has order $p^{2\ell+q}=p^{3m+q}$. Indeed, passing to an additive notation, we can view
$\langle x,y,z\rangle$ as the quotient of a free abelian
group with basis $\{X,Y,Z\}$ by the subgroup generated by $p^{\ell+q}X, p^{m}w_\a X+p^{m/2}w_\b Y-p^{m/2}k Z,
p^{\ell} X-p^{m}Y, 2p^{\ell}k X-p^{m}u Z$. Thus, the matrix whose columns are the coordinates of these generators relative
to the basis $\{X,Y,Z\}$ is
$$
M=\begin{pmatrix} p^{\ell+q} & p^{m}w_\a & p^{\ell} & 2p^{\ell}k\\
0 & p^{m/2}w_\b & -p^{m} & 0\\
0 & -p^{m/2}k & 0 & -p^{m}u\end{pmatrix}.
$$
Let $d_1,d_2,d_3$ the determinants of the 3 submatrices $M_1,M_2,M_3$ of $M$ of size $3\times 3$, obtained by deleting
columns 3, 2, and 1, respectively. Then the order of $\langle x,y,z\rangle$ is $d=\gcd\{d_1,d_2,d_3\}$. Up to 
a factor relatively prime to $p$ that is irrelevant (as $\langle x,y,z\rangle$ is clearly a finite $p$-group), we have
$$
d_1=p^{2\ell+q}, d_2=p^{3m+q}, d_3=p^{3m}(u(w_\a+w_\b)-2k^2).
$$
Here $u(w_\a+w_\b)-2k^2\equiv u^3-2k^2\mod p^m$ if $p>3$ and $u(w_\a+w_\b)-2k^2\equiv u^3-2k^2\mod 3^{m-1}$ if $p=3$.
Since $m/2\leq m-1$, it follows that $v_p(p^{3m}(u(w_\a+w_\b)-2k^2))$ is equal to $3m+s=3m+q$ if $s<m/2$, and
is at least $3m+m/2=3m+q$ if $s\geq m/2$. Thus $d=p^{2\ell+q}=p^{3m+q}$, as claimed.

We next construct a cyclic extension $\langle x,y,z_0\rangle$
of $\langle x,y,z\rangle$ of order $p^{4m+q}$, where
$z_0^{p^{m}}=z$, by means of an automorphism $\Omega$
of $\langle x,y,z\rangle$ that fixes $z$ and such that
$\Omega^{p^{m}}$ is conjugation by~$z$, that is, the trivial automorphism. In order to achieve this goal,
we consider the assignment
		\[
    x\mapsto x^\alpha,\;
    y\mapsto y^\gamma,\;
    z\mapsto z,
    \]
		where $\gamma=1-p^m v$ is the inverse of $\beta$ modulo $p^{2m}$. 
		The defining relations  of  $\langle x,y,z\rangle$ are easily seen to be preserved. 
		Thus the above assignment extends to an endomorphism $\Omega$ of $\langle x,y,z\rangle$ which is clearly surjective
		and hence an automorphism of $\langle x,y,z\rangle$. Since $\alpha^{p^{m}}\equiv 1 \mod p^{2m}$ and 
		$\gamma^{p^{m}}\equiv 1 \mod p^{2m}$, we see that $\Omega^{p^{m}}$ is the trivial automorphism.
		This produces the required extension, where $\Omega$ is conjugation by $z_0$. We see that $\langle x,y,z_0\rangle$ has defining relations:
		\[
		xy = yx,\;
    x^{z_0} = x^\alpha,\;
    {}^{z_0}y=y^\beta,\;
		x^{p^{\ell+q}}=1,\]
		\[
   z_0^{p^{\ell}k}=x^{p^{m}w_\a}y^{p^{m/2}w_\b},
   x^{2p^{\ell}k}=z_0^{p^{2m}u},\; x^{p^{\ell}}=y^{p^{m}}.
	\]
		
		We next construct a cyclic extension $\langle x_0,y,z_0\rangle$
of $\langle x,y,z_0\rangle$ of order $p^{5m+q}$ with
$x_0^{p^{m}}=x$, by means of an automorphism $\Psi$
of $\langle x,y,z_0\rangle$ that fixes $x$ and such that
$\Psi^{p^{m}}$ is conjugation by~$x$.
Appealing to (\ref{conjbya0}) and (\ref{conjbya}), this is achieved by
the automorphism 
$$
x\mapsto x,\; y\mapsto z_0^{-p^\ell} y^{1+h}=z^{-m/2} y^{1+h},\; z_0\mapsto z_0 x^{-u},
$$
where $h$ is as defined in (\ref{hache}). All defining relations of $\langle x,y,z_0\rangle$ are easily
seen to be preserved, except perhaps for $z_0^{p^{\ell}k}=x^{p^{m}w_\a}y^{p^{m/2}w_\b}$. Its preservation reduces
to show that $x^{-u(\a^{p^\ell k}-1)/(\a-1)}=z^{-p^m w_\b}$, that is, $x^{p^\ell u k}=z^{p^m w_\b}$. This is true
because $x^{p^\ell u k}=z^{p^{2m}u^2/2}=z^{p^m w_\b}$, since $2 w_\b\equiv v^2\equiv u^2\mod p^{m/2}$.
It follows from Proposition \ref{abcabelian} that $\Psi^{p^m}$ is conjugation by $x$.  This produces
the required extension, where $\Psi$ is conjugation by $x_0$. We see that $\langle x_0,y,z_0\rangle$ has defining relations:
		\[
		y^{x_0} = z_0^{-p^\ell} y^{1+h},\;
    x_0^{z_0} = x_0^\alpha,\;
    {}^{z_0}y=y^\beta,\;
		x_0^{p^{q+5m/2}}=1,\]
		\[
   z_0^{p^{\ell}k}=x_0^{p^{2m}w_\a}y^{p^{m/2}w_\b},
   x_0^{2p^{5m/2}k}=z_0^{p^{2m}u},\; x_0^{p^{5m/2}}=y^{p^{m}}.
	\]
		
We next construct a cyclic extension $\langle x_0,y_0,z_0\rangle$
of $\langle x_0,y,z_0\rangle$ of order $p^{q+11m/2}$ with
$y_0^{p^{m/2}}=y$, by means of an automorphism $\Pi$
of $\langle x,y,z_0\rangle$ that fixes $y$ and such that
$\Pi^{p^{m/2}}$ is conjugation by~$y$. To achieve this we appeal to (\ref{indi}) and consider the assignment
$$
x_0\mapsto x_0 y^g z_0^{p^m},\; y\mapsto y,\; z_0\mapsto y^{p^{2m-\ell}v}z_0,
$$		
where $g=-p^{2m-\ell}v/2$ if $p>3$ and $g=-p^{2m-\ell}v/2-3^{3m-1}$ if $p=3$.

We claim that all defining relations of $\langle x_0,y,z_0\rangle$ are preserved, in which case the given
assignment extends to an endomorphism of $\langle x_0,y,z_0\rangle$, which is then clearly an automorphism.

$\bullet$ ${}^{z_0}y = y^\beta$. This is clearly preserved.

$\bullet$ $x_0^{z_0} = x_0^\alpha$. We need to show that
\begin{equation}
\label{redifp}
(x_0 y^{g} z)^{y^{p^{2m-\ell} v}z_0}=(x_0 y^{g} z)^\a.
\end{equation}
We first compute the right hand side of (\ref{redifp}). We have
\begin{equation}
\label{redif2p}
(x_0 y^{g} z)^\a=x_0^\a (y^{g} z)^{x_0^{\a-1}}(y^{g} z)^{x_0^{\a-2}}\cdots (y^{g} z)^{x_0} (y^{g} z).
\end{equation}
The calculation of (\ref{redif2p}) requires that we know how to conjugate $y^g$ and $z$ by $x_0^i$, $i\geq 1$. From
$
y^{x_0}=z^{-p^{m/2}}y^{1+h},
$
we infer
$$
(y^{g})^{x_0}=z^{-p^{m/2}g}y^{(1+h)g}=z^{p^{m} v/2+t} y^g,
$$
where $t=0$ if $p>3$ and $t=3^{2m-1}$ if $p=3$. Since $[x_0,z^{p^m}]=[x_0,z_0^{p^{2m}}]=1$, given that $x_0^{p^{3m}}=1$, 
\begin{equation}
\label{redif3p}
(y^{g})^{x_0^i}=z^{(p^{m} v/2+t)i} y^g ,\quad i\geq 1.
\end{equation}
On the other hand, from $\a^{p^{m}}\equiv 1+p^{2m} u\mod p^{3m}$, we successively find
$$x_0^{z}=x_0^{z_0^{p^m}}=x_0^{\a^{p^{m}}}=x_0^{1+p^{2m} u},
$$
$$
z^{x_0}=z x_0^{-p^{2m}u},
$$
\begin{equation}
\label{redif4p}
z^{x_0^i}=z x^{p^{m}ui},\quad i\geq 1.
\end{equation}
Combining (\ref{redif2p})-(\ref{redif4p}), we obtain
$$
(x_0 y^{g} z)^\a=x_0^\a y^g z^{\a}.
$$
Regarding the left hand side of (\ref{redifp}), from $y^{x_0}=z_0^{-p^\ell}y^{1+h}$, we successively deduce
$$
(y^{p^{2m-\ell}v})^{x_0}=z_0^{-p^{2m}v}y^{p^{2m-\ell}v},
$$
$$
x_0^{y^{p^{2m-\ell}v}}=x_0 z^{p^{m}v}.
$$
Therefore
$$
(x_0 y^{g} z)^{y^{p^{2m-\ell}v}z_0}=(x_0 z^{p^{m}v} y^{g} z)^{z_0}=(x_0 y^g z^\b)^{z_0}=x_0^\a y^g z^\b=x_0^\a y^g z^\a,
$$
as $[y^g,z_0]=1$ and $\a\equiv\b\mod p^\ell$, with $\ell\geq m+q$. 

$\bullet$ $x_0^{2^{q+5m/2}}=1$ and $x_0^{p^{\ell+m}}=y^{p^{m}}$. These follows easily from  (\ref{redif2p})-(\ref{redif4p}).


$\bullet$ $x_0^{2 p^{\ell+m}k}=z_0^{p^{2m}u}$. From (\ref{redif2p})-(\ref{redif4p}), we see that
$$
(x_0 y^{g} z)^{2 p^{\ell+m}k}=x_0^{2 p^{\ell+m}k}.
$$
On the other hand, since $4m-\ell\geq m+q$, we have
$$
(y^{p^{2m-\ell}v}z_0)^{p^{2m}u}=(y^{p^{2m-\ell}v})^{1+\b+\cdots+\b^{p^{2m}u}}z_0^{p^{2m}u}=z_0^{p^{2m}u}.
$$

$\bullet$ $z_0^{p^{\ell}k}=x_0^{p^{2m}w_\a}y^{2^{m/2}w_\b}$. This follows as above and by appealing to (\ref{redif2p})-(\ref{redif4p}).

$\bullet$ $y^{x_0} = z_0^{-p^{\ell}}y^{1+h}$. Arguing as above, we find that
$$
(y^{p^{2m-\ell}v}z_0)^{-p^\ell}y^{1+h}=z_0^{-p^{\ell}}y^{1+h}.
$$
On the other hand
$$
y^{x_0 y^{g} z}=(z_0^{-p^{\ell}}y^{1+h})^{y^g z}=z_0^{-p^{\ell}}y^{1+h}.
$$

Thus all defining relations of $\langle x_0,y,z_0\rangle$ are preserved. 
We next claim that $\Pi^{p^{m/2}}$ is conjugation by $y$. This is clear for $y$. As for $z_0$, from $z_0 y z_0^{-1}=y^{1+p^m v}$, 
we deduce that $z_0^{y}=y^{p^{m}v}z_0= z_0 \Pi^{p^{m/2}}$. Regarding $x_0$, note that $\Pi$ fixes $z$, using $y^{p^{3m/2}}=1$,
so
$$
x_0 \Pi^{p^{m/2}}= x_0 y^{g p^{m/2}}z^{p^{m/2}}.
$$
On the other hand, from $y^{x_0}=z^{-p^{m/2}}y^{1+h}$ and $-h=g p^{m/2}$, we deduce
$$
x_0^{y}=x_0 y^{-h}z^{p^{m/2}}=x_0 y^{g p^{m/2}}z^{p^{m/2}}.
$$

This produces the required extension, where $\Pi$ is conjugation by $y_0$. We readily verify that $\langle x_0,y_0,z_0\rangle$ has defining relations:
    \[
    x_0^{y_0} =x_0 y_0^{-h} z_0^{p^m},\;
    x_0^{z_0} = x_0^\alpha,\;
    ^{z_0} y_0=y_0^\b,\;
    x_0^{p^{q+5m/2}}=1,
    \]
		\[
		z_0^{p^{\ell}k}=x_0^{p^{2m}w_\a}y_0^{p^{m}w_\b},
x_0^{2p^{\ell+m}k}=z_0^{p^{2m}u},\; x_0^{p^{\ell+m}}=y_0^{p^{\ell}}.
    \]
		
From $x_0^{y_0} =x_0 y_0^{-h} z_0^{p^m}$, we infer $y_0^{x_0} = z_0^{-p^m} y_0^{1+h}$ and
$x_0^{y_0^v} =x_0 y_0^{-h v} z_0^{p^m v}$. Also, from $z_0^{y_0}=y_0^{p^m v} z_0$, we deduce 
$z_0^{y_0^v}=y_0^{p^m v^2} z_0=z_0 y_0^{v(\b-1)}$.

We finally construct a cyclic extension $\langle x_0,y_1,z_0\rangle$
of $\langle x_0,y,z_0\rangle$ of order $p^{q+13m/2}$, such that
$y_1^{p^{m}}=y_0$, by means of an automorphism $\Lambda$
of $\langle x_0,y_0,z_0\rangle$ that fixes $y_0$ and such that
$\Lambda^{p^{m}}$ is conjugation by~$y_0$. For this purpose, we consider the assignment
$$
x_0\mapsto x_0 z_0,\; y_0\mapsto y_0,\; z_0\mapsto y_0^v z_0.
$$
We claim that all defining relations of $\langle x_0,y_0,z_0\rangle$ are preserved, in which case the given
assignment extends to an endomorphism of $\langle x_0,y,z_0\rangle$, which is then clearly an automorphism.

$\bullet$ $^{z_0} y_0=y_0^\b$. This is clear.

$\bullet$ $x_0^{p^{\ell+m}}=y_0^{p^{\ell}}$. From $2m+\ell>3m$ and $\ell+m\geq 2m+q$, we see that
$$
(x_0 z_0)^{p^{\ell+m}}=z_0^{p^{\ell+m}}x_0^{\a(\a^{p^{\ell+m}}-1)/(\a-1)}=x_0^{p^{\ell+m}}.
$$

$\bullet$ $x_0^{p^{q+5m/2}}=1$. We have
$$
(x_0 z_0)^{p^{q+5m/2}}=z_0^{q+5m/2}x_0^{\a(\a^{q+5m/2}-1)/(\a-1)}=1.
$$

$\bullet$ $x_0^{2 p^{\ell+m}k}=z_0^{p^{2m}u}$. By above, $(x_0 z_0)^{2p^{\ell+m}k}=x_0^{2p^{\ell+m}k}$.
On the other hand, as $2m\geq q+3m/2$,
$$
(y_0^v z_0)^{{p^{2m}u}}=y_0^{v (\b^{{p^{2m}u}}-1)/(\b-1)}z_0^{{p^{2m}u}}=z_0^{{p^{2m}u}}.
$$

$\bullet$ $z_0^{p^{\ell}k}=x_0^{p^{2m}w_\a}y_0^{p^{m}w_\b}$. Using $\ell-m\geq q$, we find that 
$$
(y_0^v z_0)^{p^{\ell}k}=y_0^{v (\b^{p^{\ell}k}-1)/(\b-1)}z_0^{p^\ell k}=y_0^{p^{\ell} v k} z_0^{p^\ell k}=y_0^{p^{\ell} u k} z_0^{p^\ell k},
$$
$$
(x_0 z_0)^{p^{2m}w_\a}=z_0^{p^{2m}w_\a}x_0^{\a(\a^{p^{2m}w_\a}-1)/(\a-1)}=z_0^{p^{2m}w_\a}x_0^{p^{2m}w_\a},
$$
and we are reduced to show that $z_0^{p^{2m}w_\a}=y_0^{p^{\ell} u k}$, which is true since 
$y_0^{p^{\ell} u k}=x_0^{p^{\ell_m} u k}=z_0^{p^{2m} u^2/2}$, where $2w_\a\equiv u^2\mod p^{m-1}$, with $m-1\geq m/2$.

$\bullet$ $x_0^{y_0} =x_0 y_0^{-h} z_0^{p^m}$.  We have
$$
(x_0z_0)^{y_0}=x_0 y_0^{-h} z_0^{p^m} y_0^{p^m v}z_0=x_0 y_0^{-h} y_0^{p^m v} z_0^{p^{m+1}},
$$
$$
x_0 z_0 y_0^{-h}  (y_0^{v} z_0)^{p^m}=x_0 z_0 y_0^{-h} y_0^{v(\b^{p^m}-1)/(\b-1)} z_0^{p^m}=x_0 y_0^{-h} 
y_0^{p^m v} z_0^{p^{m+1}}.
$$

$\bullet$ $x_0^{z_0} = x_0^\alpha$. We need to show that $(x_0z_0)^{y_0^v z_0}=(x_0 z_0)^\a$. 
From $x_0^{y_0^v} =x_0 y_0^{-h v} z_0^{p^m v}$ and $z_0^{y_0^v}=z_0 y_0^{v(\b-1)}$, we deduce
$$
(x_0z_0)^{y_0^v}=x_0 y_0^{-h v} z_0^{p^m v} z_0 y_0^{v(\b-1)}=x_0 y_0^{v(-h+(\b-1))} z_0^{\b}.
$$
Here $\b=\a-p^\ell k$ and $z_0^{-p^\ell k}=y_0^{-p^{m}w_\b}x_0^{-p^{2m}w_\a}$. Set $r=0$ if $p>3$, and
$r=3^{2m-1}u$ if $p=3$, noting that if $p=3$, then $3^{2m-1}u\equiv 3^{2m-1}v\mod 3^{2m}$.
Then $v(-h+(\b-1))\equiv p^m v^2/2-r\mod p^{2m}$ and
$2 p^m w_\b\equiv v^2-2 r\mod p^{2m}$, whence
$$
(x_0z_0)^{y_0^v}=x_0 y_0^{v(-h+(\b-1))} y_0^{-p^{m}w_\b}x_0^{-p^{2m}w_\a} z_0^{\a}=x_0 x_0^{-p^{2m}w_\a} z_0^{\a},
$$
$$
(x_0z_0)^{y_0^v z_0}=x_0^\a x_0^{-p^{2m}w_\a} z_0^{\a}=z_0^{\a} x_0^{\a^\a+1} x_0^{-p^{2m}w_\a}.
$$
On the other hand, $(x_0 z_0)^\a=z_0^\a x_0^{\a(\a^\a-1)/(\a-1)}$. Thus $(x_0z_0)^{y_0^v z_0}=(x_0 z_0)^\a$ if and only if
$$
x_0^{\a(\a^\a-1)/(\a-1)+p^{2m}w_\a}=x_0^{\a^\a+1}.
$$
Here
$$
\alpha (\alpha^{\alpha}-1)/(\alpha-1)\equiv \alpha^2+(\alpha-1)^2/2+{{\alpha}\choose{3}}(\alpha-1)^2\mod p^{3m}.
$$
Set $j=3^m\times r$. As $2 p^{2m}w_\a\equiv p^{2m}u^2-2 j\mod p^{3m}$, we infer
$$
x_0^{\a(\a^\a-1)/(\a-1)+p^{2m}w_\a}=x_0^{\alpha^2+(\alpha-1)^2+{{\alpha}\choose{3}}(\alpha-1)^2 -j}.
$$
Since
$$
\alpha^{\alpha+1}\equiv \alpha+\alpha^2(\alpha-1)\equiv \alpha^2+(\alpha-1)^2\mod p^{3m},
\quad {{\alpha}\choose{3}}(\alpha-1)^2\equiv j\mod p^{3m},
$$
we conclude that $x_0^{\a(\a^\a-1)/(\a-1)+p^{2m}w_\a}=x_0^{\a^\a+1}$, as required.

That $\Lambda^{p^m}$ and conjugation by $y_0$ agree on $y_0$ and $z_0$ is clear, and it remains to verify
that $x_0^{y_0}=x_0\Lambda^{p^m}$, that is, $x_0 y_0^{-h} z=x_0y_0^{v\beta(1+2\beta+3\beta^2+\cdots+(p^m-1)\beta^{p^m-2})}z_0^{p^m}$,
which is a consequence of
$$
v\beta(1+2\beta+3\beta^2+\cdots+(p^m-1)\beta^{p^m-2})\equiv -h\mod p^{2m}.
$$
The verification of this congruence is carried out in the proof of \cite[Theorem 11.2]{MS}.

This produces the required extension, where $\Lambda$ is conjugation by $y_1$. From
$x_0^{y_1}=x_0z_0$, we infer $[x_0,y_1]=z_0$, so $\langle x_0,y_1,z_0\rangle=\langle x_0,y_1\rangle$.
Moreover, we have $x_0^{z_0}=x_0^\alpha$ and $z_0^{y_1}=y_0^{v}z_0=
y_1^{p^m v}z_0=y_1^{\beta-1}z_0$, which implies
${}^{z_0} y_1=y^\beta$. Thus
$\langle x_0,y_1\rangle$ is an image of $G_p$ of the required order.
\end{proof}

We suppose next that $2\ell<3m$. 
Set $i=m+\ell$ and $j=4m-\ell$, so that $i<j$. Combining (\ref{abc2}) and (\ref{los33}) yields
$$
a^{p^i}\in \langle b^{p^j}\rangle, b^{p^i}\in \langle a^{p^j}\rangle,
$$
which easily implies
\begin{equation}\label{triv}
a^{p^{m+\ell}}=1=b^{p^{m+\ell}}.
\end{equation}
Raising (\ref{abc}) to the $p^{(\ell-m)}$th power and making use of (\ref{triv}) gives
\begin{equation}\label{triv2}
c^{p^{2\ell-m}}=1.
\end{equation}
It follows from (\ref{prin3}), (\ref{y}), (\ref{triv}), and (\ref{triv2}) that
$$
|G_p|\leq p^{m+\ell}p^{2\ell-m}p^{2m}=p^{2m+3\ell}.
$$
Moreover, due to (\ref{triv}), the class of $G_p$ is at most 5, as explained in Section \ref{ellmn},
with
$$
a^{p^{2m}}, b^{p^{2m}},c^{p^{\ell}}\in Z,
a^{p^{\ell}}, b^{p^{\ell}},c^{p^m}\in Z_2, a^{p^{m}}, b^{p^{m}}\in Z_3,c\in Z_4,\; Z_5=G_p.
$$

\begin{theorem}\label{teo3} If $n=m<\ell<2m$ and $2\ell<3m$, then $e=2m+3\ell$, $f=5$, $o(a)=p^{m+\ell}=o(b)$,
and $o(c)=p^{2\ell-m}$.
\end{theorem}

\begin{proof} Let us construct an image
of $G_p$ of order $p^{2m+3\ell}$. To achieve this, we appeal to Proposition~\ref{abcabelian} and
take (\ref{model1r1}) into account, so we
begin with a group $T=\langle X,Y,Z\rangle$ of order $p^{3\ell-m}$
having defining relations $[X,Y]=[X,Z]=[Y,Z]=1$, $Z^{p^{\ell-m} k}=X^{p^{m} w_\a}Y^{p^{m} w_\b}$,
and $X^{p^{\ell}}=Z^{p^{2(\ell-m)}}=1$. Here $X,Y,Z$ play the roles of $a^{p^m},b^{p^m},c^{p^m}$, respectively. Note that $Y^{p^{\ell}}=1$.

We first construct a cyclic extension $\langle X,Y,Z_0\rangle$
of $\langle X,Y,Z\rangle$ of order $p^{3\ell}$, where
$Z_0^{p^{m}}=Z$, by means of an automorphism $\Omega$ of $\langle X,Y,Z\rangle$ that is conjugation by $Z_0$, namely
$$
X\mapsto X^{\a},\; Y\mapsto Y^{\gamma},\; Z\mapsto Z,
$$
where $\gamma=1-v p^m$ is the inverse of $\beta$ modulo $p^{2m}$. We see that $\langle X,Y,Z_0\rangle$
has defining relations:
$$
X^{Z_0}=X^\alpha,\; {}^{Z_0} Y=Y^\beta,\; XY=YX,\; X^{p^{\ell}}=1,\; Z_0^{p^{\ell}k}=X^{p^{m}w_\a}Y^{p^{m}w_\b}=1,
Z_0^{p^{2\ell-m}}=1.
$$

We next construct a cyclic extension $\langle X_0,Y,Z_0\rangle$
of $\langle X,Y,Z_0\rangle$ of order $p^{3\ell+m}$ with
$X_0^{p^{m}}=X$, by means of an automorphism $\Psi$
of $\langle X,Y,Z_0\rangle$ that is conjugation by $X_0$.
Appealing to (\ref{conjbya0}) and (\ref{conjbya}), this is achieved by
$$
X\mapsto X,\; Y\mapsto Z_0^{-p^m}\, Y^{1+h}=Z^{-1} Y^{1+h},\; Z_0\mapsto Z_0 X^{-u},
$$
where $h$ is as defined in (\ref{hache}). We easily verify that
the given assignment extends to an automorphism, and Proposition \ref{abcabelian} ensures that $\Psi^{p^m}$
is conjugation by $X$, as required. We readily verify that $\langle X_0,Y,Z_0\rangle$ has defining relations
$$
X_0^{p^{\ell+m}}=1,\; X_0^{Z_0}=X_0^\alpha,\;
 Y^{X_0}=Z_0^{-p^m}\, Y^{1+h},\;
{}^{Z_0} Y=Y^\beta,
Z_0^{p^{\ell}k}=X_0^{p^{2m}w}Y^{p^m w}=1,\;
Z_0^{p^{2\ell-m}}=1.
$$

We finally construct a cyclic extension $\langle X_0,Y_0,Z_0\rangle$
of $\langle X_0,Y,Z_0\rangle$ of order $p^{3\ell+2m}$ with
$Y_0^{p^{m}}=Y$, by means of an automorphism $\Pi$
of $\langle X_0,Y,Z_0\rangle$ that is conjugation by $Y_0$, namely
$$
X_0\mapsto X_0Z_0,\; Y\mapsto Y,\; Z_0\mapsto Y^{v} Z_0.
$$
The verification that the given assignment preserves the defining relations is routine,
except for the verification that $(X_0Z_0)^{Y^{v}Z_0}=(X_0Z_0)^\alpha$. The argument
given in the proof of Theorem \ref{teo5} applies. 
That $\Pi^{p^{m}}$ is conjugation by~$Y$ follows as in the proof of Theorem \ref{teo5}. 
Clearly $\langle X_0,Y_0,Z_0\rangle=\langle X_0,Y_0\rangle$ is an image of $G_p$ of the required order.
\end{proof}

Suppose finally that $2\ell>3m$. Raising (\ref{abc}) to the $p^{(\ell-m)}$th power,
we obtain
\begin{equation}\label{abc6}
c^{p^{2\ell-m}k}=a^{p^{\ell+m} u^2/2}b^{p^{\ell+m} v^2/2},
\end{equation}
regardless of whether $p>3$ or $p=3$. In view of (\ref{los3}), we see that (\ref{abc6}) becomes
\begin{equation}\label{abc7}
c^{p^{2\ell-m}k}=a^{p^{\ell+m} (u^2+v^2)/2}.
\end{equation}
But $u\equiv v\mod p^{\ell-m}$ and $2\ell\geq 3m$, so (\ref{abc7}) becomes
\begin{equation}\label{abc8}
c^{p^{2\ell-m}k}=a^{p^{\ell+m} u^2}.
\end{equation}
On the other hand, by (\ref{los33}), we have
\begin{equation}\label{abc9}
c^{p^{2m}u^3/2k}=a^{p^{m+\ell}u^2}.
\end{equation}
Comparison of (\ref{abc8}) and (\ref{abc9}) yields
\begin{equation}\label{abc10}
c^{p^{2m}u^3/2k}=c^{p^{2\ell-m}k}.
\end{equation}
Using $2\ell>3m$ and (\ref{abc10}) produces $c^{p^{2m}}=1$, which implies $a^{p^{m+\ell}}=1=b^{p^{m+\ell}}$ by (\ref{haim3}). Thus
by (\ref{prin3}) and (\ref{y}), we see that $|G_p|\leq p^{5m+\ell}$ and $G_p$ has class at most 5, with
$$
a^{p^{2m}}, b^{p^{2m}},c^{p^{\ell}}\in Z,
a^{p^{\ell}}, b^{p^{\ell}},c^{p^m}\in Z_2, a^{p^{m}}, b^{p^{m}}\in Z_3,c\in Z_4,\; Z_5=G_p.
$$

\begin{theorem}\label{teo4} If $n=m<\ell<2m$ and $2\ell>3m$, then $e=5m+\ell$, $f=5$, $o(a)=p^{m+\ell}=o(b)$,
and $o(c)=p^{2m}$.
\end{theorem}

\begin{proof} Let us construct an image
of $G_p$ of order $p^{5m+\ell}$. To achieve this, we appeal to Proposition~\ref{abcabelian} and
take (\ref{model1r1}) into account, so we
begin with a group $T=\langle X,Y,Z\rangle$ of order $p^{2m+\ell}$
having defining relations $[X,Y]=[X,Z]=[Y,Z]=1$, $Z^{p^{\ell-m} k}=X^{p^{m} w_\a}Y^{p^{m} w_\b}$,
and $X^{p^{\ell}}=Y^{p^{\ell}}=Z^{p^{m}}=1$. Here $X,Y,Z$ play the roles of
$a^{p^m},b^{p^m},c^{p^m}$, respectively. The proof can be continued as that of Theorem \ref{teo3}.
\end{proof}

\section{The case when $m=n$ and $\ell\geq 2m$}\label{ell2m}

We maintain the hypotheses of Sections \ref{case1} and \ref{ellmn} and assume further that $\ell\geq 2m$ (this includes Case 1 from \cite{MS}).
Then (\ref{prin1}), (\ref{p4}), and $\ell\geq 2m$ yield
$$a^{p^{2m}}, b^{p^{2m}}\in Z.$$
Since $a^{p^{3m}}=1=b^{p^{3m}}$, this implies
\begin{equation}\label{jt}
c^{p^{2m}}=1,
\end{equation}
which is equivalent to $a^{p^{m+\ell}}=1=b^{p^{m+\ell}}$, so all subgroups appearing in (\ref{haim3}) are trivial, 
and the class of $G_p$ is at most 5.

Making use of (\ref{val}), (\ref{model1r1}), (\ref{jt}),  $\ell\geq 2m$, $p\nmid w_\a$, $p\nmid w_\b$, and
the fact that $w_\a$ and $w_\b$ can be used interchangeably in (\ref{model1r1}), we deduce
\begin{equation}\label{inv}
a^{p^{2m}}b^{p^{2m}}=1.
\end{equation}

By  (\ref{jt}), (\ref{inv}), and
our results on the upper central series of $G_p$ from Section \ref{ellmn}, we see that
$$
\langle a^{p^{2m}}\rangle= \langle b^{p^{2m}}\rangle\subseteq Z,\; c^{p^m}\in Z_2,\; a^{p^m}, b^{p^m}\in Z_3,\;
c\in Z_4, G_p=Z_5.
$$

\begin{theorem}\label{teo2} If $m=n$ and $\ell\geq 2m$, then $e=7m$, $f=5$, $o(a)=p^{3m}=o(b)$,
and $o(c)=p^{2m}$.
\end{theorem}

\begin{proof} Let us construct an image
of $G_p$ of order $p^{p^{7m}}$. To achieve this, we appeal to Proposition~\ref{abcabelian}.
As $\ell\geq 2m$, we see that (\ref{model1r1}) is equivalent to (\ref{inv}),
so we begin with a group $T=\langle X,Y,Z\rangle$ of order $p^{4m}$
having defining relations $[X,Y]=[X,Z]=[Y,Z]=1$, $X^{p^m}Y^{p^m}=1$, and $X^{p^{2m}}=Z^{p^m}=1$. Here $X,Y,Z$ play the roles of
$a^{p^m},b^{p^m},c^{p^m}$, respectively. The proof can be continued as that of Theorem \ref{teo3}.
\end{proof}

\section{The case when $p=3$ and $\a,\b\equiv 7\mod 9$}

We assume throughout this section that $p=3$ and $\a,\b\equiv 7\mod 9$. Thus $\a=1+3u$, $\b=1+3v$, where $u,v\equiv -1\mod 3$,
so that $\a\equiv\b\mod 9$. We assume first that $\a\equiv\b\mod 27$. Thus, either $u\equiv -1\equiv v\mod 9$ or
$u\equiv -4\equiv v\mod 9$ or $u\equiv -7\equiv v\mod 9$. The first of these cases needs to be treated separately for
valuation reasons. So begin by assuming that $u\equiv -4\equiv v\mod 9$ or $u\equiv -7\equiv v\mod 9$. Then,
according to (\ref{lomismo}) and \cite[Proposition 2.1]{MS}, we have
$$
v_3(\d_\a)=3=v_3(\d_\b).
$$
From $v_3(\d_\a)=3=v_3(\d_\b)$ and (\ref{adabc4}) we infer
$$
a^{81},b^{81}\in Z,
$$
and therefore
$$
a^{243}=1=b^{243}.
$$
On the other hand, $\a\equiv\b\mod 27$ and $a^{81}\in Z$ imply $a^{\a^{\a-\b}-1}\in Z$. Moreover, a careful calculation shows that
\begin{equation}\label{tazul}
v_3(\l_\b)\geq 4,
\end{equation}
so $b^{81}\in Z$ forces $b^{\l_\b}\in Z$ as well. From $a^{\a^{\a-\b}-1}, b^{\l_\b}\in Z$, $v_3(\d_\b)=3$, and (\ref{prin2}), we deduce
$$c^{27}\in Z,$$ and therefore
$$
a^{81}=1=b^{81}.
$$
Moreover, from $c^{27}\in Z$,  $\a\equiv\b\mod 27$ and (\ref{prin1}), we infer
$$
a^{27},b^{27}\in Z.
$$
Thus, using $a^{81}=1$, we see that
$$
a^{27}=(a^{27})^b=c^{27}a^{\a(\a^{27}-1)/(\a-1)}=c^{27} a^{27},
$$
whence
$$
c^{27}=1.
$$
Going back to (\ref{prin1}) we derive
$$
a^{\d_\a}b^{\d_\b}=1.
$$
In view of (\ref{lomismo}), a suitable conjugation of these central elements by $c$ and $c^{-1}$ yields
that
\begin{equation}\label{tmarron}
a^{\g_\a}b^{\g_\b}=1.
\end{equation}
Now if $u\equiv -4\equiv v\mod 9$, then $u=-1+3u_0$, where $u_0\equiv -1\mod 3$, and $v=-1+3v_0$, where $v_0\equiv -1\mod 3$,
while if $u\equiv -7\equiv v\mod 9$, then $u=-1+3u_0$, where $u_0\equiv 1\mod 3$, and $v=-1+3v_0$, where $v_0\equiv 1\mod 3$.
In both cases $u_0\equiv v_0\mod 3$. On the other hand, the last statement of \cite[Proposition 2.1]{MS} ensures that
(\ref{tmarron}) becomes
$$
a^{-27u_0}b^{-27v_0}=1.
$$
As $u_0\equiv v_0\mod 3$, we deduce
$$
a^{27}b^{27}=1.
$$
Clearly $|G_3|\leq 3^{10}$ and
$$
a^{27},b^{27}\in Z, c^9\in Z_2, a^9,b^9\in Z_3, c^3\in Z_4, a^3,b^3\in Z_4, c\in Z_6, Z_7=G_3.
$$

The case when $u\equiv -1\equiv v\mod 9$ is similar, but requires new relations, as the old valuations do not give the required results.
By Proposition \ref{nuevaca}, we have
$$
a^{(\alpha-1)\m_\a}=b^{(\beta-1)\m_\b},
$$
and therefore
$$
a^{(\alpha-1)^2\m_\a}=1=b^{(\beta-1)^2\m_\b}.
$$
According to \cite[Proposition 2.2]{MS}, we have
$$
v_3((\alpha-1)\m_\a)=4=v_3((\b-1)\m_\b),
$$
because in this case $u\equiv -1\equiv v\mod 9$. We may now deduce from $a^{(\alpha-1)\m_\a}=b^{(\beta-1)\m_\b}$ that
$$
a^{81},b^{81}\in Z,
$$
and therefore
\begin{equation}\label{c0}
a^{243}=1=b^{243}.
\end{equation}
By (\ref{izq}), we have
\begin{equation}\label{c}
b^{\m_\b}=a^{-\m_\a \alpha_0^{\alpha^2+2}}c^{\alpha^2-\beta^2}.
\end{equation}
The operator $[a,-]$ then gives
$$
b^{(\b-1)(\b+2\b^2+\cdots+(\m_\b-1)\b^{\m_\b-1})}c^{\m_\b}=a^{\a^{\a^2-\b^2}-1}.
$$
Here $v_3(\a^{\a^2-\b^2}-1)=4$ and the same calculation that yields (\ref{tazul}) also gives
$$
v_3((\b-1)(\b+2\b^2+\cdots+(\m_\b-1)\b^{\m_\b-1}))\geq 4,
$$
so $a^{81},b^{81}\in Z$ forces
$$
c^{27}\in Z,
$$
and therefore
$$
a^{81}=1=b^{81}.
$$
Going back to (\ref{c}) and using $c^{27}\in Z$ shows that
$$
a^{27},b^{27}\in Z.
$$
As above, this implies $c^{27}=1$. Going back to (\ref{c}) and using $a^{27}\in Z$ and $c^{27}=1$, we find that
$$
a^{\m_\a}b^{\m_\b}=1.
$$
The proof of \cite[Proposition 2.2]{MS} shows that $\m_\a\equiv -27\equiv \m_\b\mod 81$, whence
$$
a^{27}b^{27}=1.
$$
Thus, as above, we have $|G_3|\leq 3^{10}$ and
$$
a^{27},b^{27}\in Z, c^9\in Z_2, a^9,b^9\in Z_3, c^3\in Z_4, a^3,b^3\in Z_4, c\in Z_6, Z_7=G_3.
$$
Careful calculations reveal that $(a^3)^{c^3}=a^{-24}$, $(a^3)^{b^3}=a^{-24} c^9$, $(c^3)^{b^3}=b^{-27} c^3$.

\begin{theorem}\label{teo6} If $\a,\b\equiv 7\mod 9$ and $\a\equiv\b\mod 27$, then $e=10$, $f=7$, $o(a)=81=o(b)$, and $o(c)=27$.
\end{theorem}

\begin{proof} This is an extension of \cite[Theorem 11.5]{MS}, whose proof still applies with minor modifications
indicated below. Missing calculations can be found in the proof of \cite[Theorem 11.5]{MS}.

We begin with a group $\langle x, y, z\rangle$ of order $3^6$ having defining relations
    \[
        x^{27} = 1,\;
        xy = yx,\;
        x^9y^3 = 1,\;
        z^9 = 1,\;
        x^z = x^{-8},\;
        yz = zy.
    \]
    Here $x$, $y$, $z$ play the roles of $a^3$, $b^9$, $c^3$, respectively. Note that $o(y)=9$ and $[x^3,z] = 1 = [x,z^3]$.

    Consider the assignment
    \[
        x\mapsto x^{-8}z^3,\;
        y\mapsto y,\;
        z\mapsto y^{-3}z.
    \]
    All relations are obviously preserved, so the given assignment extends to a surjective endomorphism and hence an automorphism $\Omega_1$
		of $\langle x, y, z\rangle$. As in the proof of \cite[Theorem 11.5]{MS}, we see that $\Omega_1^3$ is conjugation by $y$.

    Then there is a group $\langle x, y_0, z\rangle$ of order $3^7$ such that $y_0^3 = y$, $\Omega_1$ is conjugation by $y_0$, and 
		having defining relations
    \[
        x^{27} = 1,\;
        x^{y_0} = x^{-8}z^3,\;
        x^9y_0^9 = 1,\;
        z^9 = 1,\;
        x^z = x^{-8},\;
        z^{y_0} = y_0^{-9}z.
    \]
    Note that $o(y_0) = 27$ and $[y_0^9,z] = 1$.

    Let $\sigma\in\N$ be such that $\sigma\beta\equiv 1\mod 27$ and consider the assignment
    \[
        x\mapsto x^\alpha,\;
        y_0\mapsto y_0^\sigma,\;
        z\mapsto z.
    \]
    All relations but the second one are obviously preserved. From $\beta\equiv 1\mod 3$, we have
		$\sigma\equiv 1\mod 3$. We need to verify that $(x^\alpha)^{y_0^\sigma} = (x^\alpha)^{-8}z^3$. By induction we see that
    $$
        x^{y_0^t} = x^{1-9t}z^{3t},\quad t\geq 1.
    $$
    In particular
    $x^{y_0^\sigma} = x^{1-9\sigma}z^{3\sigma} = x^{1-9}z^3 = x^{-8}z^3$, which implies
    $(x^\alpha)^{y_0^\sigma} = (x^{y_0^\sigma})^\alpha = (x^{-8}z^3)^\alpha = x^{-8\alpha}z^{3\alpha} = x^{-8\alpha} z^3$, so the second relation is preserved. Thus the given assignment extends to a surjective endomorphism and hence an automorphism $\Omega_2$
		of $\langle x, y_0, z\rangle$. As in the proof of \cite[Theorem 11.5]{MS}, we see that $\Omega_2^3$ is conjugation by $z$.

   Then there is a group $\langle x, y_0, z_0\rangle$ of order $3^8$ such that $z_0^3 = z$, $\Omega_2$ is conjugation by $z_0$, and 
		having defining relations
			\[
        x^{27} = 1,\;
        x^{y_0} = x^{-8}z_0^9,\;
        x^9y_0^9 = 1,\;
        z_0^{27} = 1,\;
        x^{z_0} = x^\alpha,\;
        ^{z_0}y_0 = y_0^\beta.
    \]
    Consider the assignment
    \[
        x\mapsto x,\;
        y_0\mapsto z_0^{-3}y_0^{-2},\;
        z_0\mapsto z_0x^{-u}.
    \]
		As in the proof of \cite[Theorem 11.5]{MS}, we see that the defining relations of $\langle x, y_0, z_0\rangle$
		are preserved, which gives rise to a surjective endomorphism and hence an automorphism $\Omega_3$
		of $\langle x, y_0, z_0\rangle$. Moreover, as in the proof of \cite[Theorem 11.5]{MS}, we see that 
		$\Omega_3^3$ is conjugation by $x$.
		
     Then there is a group $\langle x_0, y_0, z_0\rangle$ of order $3^9$ such that
	$x_0^3 = x$, $\Omega_3$ is conjugation by $x_0$, and having defining relations
    \[
        x_0^{81} = 1,\;
        y_0^{x_0} = z_0^{-3}y_0^{-2},\;
        x_0^{27}y_0^9 = 1,\;
        z_0^{27} = 1,\;
        x_0^{z_0} = x_0^\alpha,\;
        ^{z_0}y_0 = y_0^\beta.
    \]
    Consider the assignment
    \[
        x_0\mapsto x_0z_0,\;
        y_0\mapsto y_0,\;
        z_0\mapsto y_0^vz_0.
    \]
		As in the proof of \cite[Theorem 11.5]{MS}, we see that the first, third, fourth, and sixth relations are preserved.
		Making the replacements $k\to v$, $\beta\to\sigma$, and $\alpha\to\beta$, the proof of \cite[Theorem 11.5]{MS}
		also yields that the second relation is preserved. Let us see that the fifth relation is preserved. As 
		in the proof of \cite[Theorem 11.5]{MS}, we see that
    \[
        (x_0z_0)^\alpha = z_0^\alpha x_0^{\alpha(1+\alpha+\ldots+\alpha^{\alpha-1})},\; 
				(x_0z_0)^{y_0^vz_0} = x_0^\alpha x_0^{27v_0} z_0^\beta,
		\]
		where $v+1=3v_0$ with $v_0\in\N$, so we need to show that 
		$z_0^\alpha x_0^{\alpha(1+\alpha+\ldots+\alpha^{\alpha-1})}=x_0^\alpha x_0^{27v_0} z_0^\beta$. Noting that $[x_0^{27},z_0] = 1$, the last identity becomes
    $x_0^{\alpha(1+\alpha+\ldots+\alpha^{\alpha-1})} = z_0^{-\alpha} x_0^{\alpha} x_0^{27v_0} z_0^{\beta} = x_0^{\alpha\alpha^\alpha} x_0^{27v_0} z_0^{\beta-\alpha}$, where $z_0^{\beta-\alpha} = 1$, since $\alpha\equiv\beta\mod 27$. Thus we have to see that
    $$
        x_0^{-27v_0} = x_0^{\alpha\gamma_\alpha}.
    $$
    Let $u+1 = 3u_0$ with $u_0\in\N$. Since $\alpha\equiv\beta\mod 27$, then $u_0\equiv v_0\mod 3$. 
		Suppose first that $v_3(u+1) = v_3(3u_0)\geq 2$, then $3\mid u_0$ and $3\mid v_0$, so $x_0^{-27v_0} = 1$.
		On the other hand, by [MS, Proposition 2.1], $v_3(\gamma_\alpha) = 2 + v_3(u+1)\geq 4$, so $x_0^{\alpha\gamma_\alpha} = 1$.
Now, if $v_3(u+1) = v_3(3u_0) = 1$, by [MS, Proposition 2.1], $v_3(\gamma_\alpha) = 2 + v_3(u+1) = 3$ and 
$\gamma_\alpha = 27t$ where $t\in\N$, $t\equiv -u_0\mod 3$. Then
    \[
        \alpha\gamma_\alpha\equiv(1+3u)(27t)\equiv 27t\equiv -27u_0\equiv -27v_0\mod 81,
    \]
    since $u_0\equiv v_0\mod 3$. Thus $x_0^{\alpha\gamma_\alpha} = x_0^{-27v_0}$ and the fifth relation is preserved. Thus the given
		assignment extends to a surjective endomorphism and hence an automorphism $\Omega_4$ of $\langle x_0, y_0, z_0\rangle$. 
		Making the replacements $k\to v$, $\alpha\to\beta$, and $u\to v_0$, the proof of \cite[Theorem 11.5]{MS} shows that
		$\Omega_4^3$ is conjugation by $y_0$.

    Then there is a group $\langle x_0, y_1, z_0\rangle$ of order $3^{10}$ such that $y_1^3 = y_0$, $\Omega_4$ is conjugation by
		$y_1$, and 
		$$[x_0,y_1] = z_0,\;
        x_0^{z_0} = x_0^\alpha,\;
        ^{z_0}y_1 = y_1^\beta.
		$$
		Thus $G_3$ has an image of order $3^{10}$, as required.	
\end{proof}

We next move to the case when $v_3(\a-\b)=2$.  Then $u\equiv -1\equiv v\mod 3$ but $u\not\equiv v\mod 9$. Thus, if $u\equiv -1\mod 9$,
then $v\equiv -4,-7\mod 9$; if $u\equiv -4\mod 9$,
then $v\equiv -1,-7\mod 9$; if $u\equiv -7\mod 9$,
then $v\equiv -1,-4\mod 9$.

In any case, since $v_3(\a-1)=1=v_3(\b-1)$ and $v_3(\a-\b)=2$, we may appeal to (\ref{h3}) to deduce
$$
a^{81}=1=b^{81}.
$$
This easily implies
$$
c^{81}=1.
$$
From (\ref{h}), (\ref{h2}), and $v_3(\a-\b)=2$, we deduce
$$
c^{27}\in \langle a\rangle\cap \langle b\rangle,
$$
so
$$
c^{27}\in Z.
$$

Moreover, from (\ref{prin1}), we have
$$
b^{\d_\b}=a^{-\a_0^{\a+1} \d_\a}c^{\a-\b}.
$$
Thus, if $u\equiv -1\mod 9$, then $a^{81}=1$ and \cite[Proposition 2.1]{M} give
$$
b^{27 w}=c^{9 z},
$$
where $3\nmid w$ and $3\nmid z$, which implies
$$
c^{27}=1.
$$
If $u\not\equiv -1\mod 9$, and $v\equiv -1\mod 9$, then $b^{81}=1$ and \cite[Proposition 2.1]{M} give
$$
a^{27 w}=c^{9 z},
$$
where $3\nmid w$ and $3\nmid z$, which implies
$$
c^{27}=1.
$$
If $u\not\equiv -1\mod 9$ and $v\not\equiv -1\mod 9$, then \cite[Proposition 2.1]{M} gives
\begin{equation}\label{falt}
c^{9w_1}=a^{27w_2}b^{27w_3},
\end{equation}
where none of $w_1,w_2,w_3$ are multiples of 3. But
$$
(a^{27})^b=c^{27}a^{27},
$$
with $c^{27}\in Z$ and $c^{81}=1$, so
\begin{equation}\label{falt2}
[a^{27},b^3]=1.
\end{equation}
As $a^{81}=1=b^{81}$, we deduce from (\ref{falt}) and (\ref{falt2}) that
$$
c^{27}=1.
$$
Thus $c^{27}=1$ and $c^9\in \langle a^{27}\rangle \langle b^{27}\rangle$ are valid in every case. From $c^{27}=1$ and $a^{81}=1=b^{81}$, we easily infer
$$
a^{27},b^{27}\in Z.
$$
But $c^9\in \langle a^{27}\rangle \langle b^{27}\rangle$, so
$$
c^9\in Z,
$$
which now implies
$$
a^{27}=1=b^{27}, c^9=1, a^9,b^9\in Z.
$$
Thus
$$
|G_3|\leq 3^8, a^9,b^9\in Z, c^3\in Z_2, a^3,b^3\in Z_3, c\in Z_4, Z_5=G_3.
$$
Careful calculations reveal that $(b^3)^a=c^{-3}b^{-6}$ and $(b^3)^{a^3}=b^{3}$.

 \begin{theorem}\label{teo7} If $\a,\b\equiv 7\mod 9$ and $v_3(\a-\b)=2$, then $e=8$, $f=5$, $o(a)=o(b)=27$, and $o(c)=9$.
\end{theorem}

\begin{proof} Consider the group $T=\langle X,Y,Z\rangle$ of order $3^6$
having defining relations $[X,Y]=1$, $X^Z=X^{\a}$, ${}^Z Y=Y^{\b}$, $X^{9}=Y^{9}=Z^{9}=1$. Here $X,Y,Z$ play the roles of $a^{3},
b^{3},c$, respectively. Consider the assignment $X\mapsto X$, $Y\mapsto Z^{-3} Y^{-2}$, $Z\mapsto Z X^{-u}$. Using $[Z^3,X]=1=[Y,Z^3]$
we see that the given assignment extends to an automorphism $\Omega$
of $T$ (which plays the role of conjugation by~$a$) that fixes $X$ and such that $\Omega^{3}$ is conjugation by $X$. 
Let $E=\langle X_0,Y,Z\rangle$ be the group arising from
Theorem \ref{Z}, so that $E/T\cong C_{3}$, $X_0$ has order $3$ modulo $T$, $X_0^{3}=X$, and $\Omega$ is conjugation by $X_0$.
Then $|E|=3^7$, with defining relations $Z^{X_0}=Z X_0^{1-\alpha}$, $Y^{X_0}= Z^{-3}Y^{-2}$, ${}^Z Y=Y^\b$, $X_0^{27}=Y^{9}=Z^{9}=1$.
Careful calculations show that assignment $X_0\mapsto X_0Z$, $Y\mapsto Y$, $Z\mapsto Y^v Z$ extends to an automorphism $\Psi$
of $E$ (which plays the role of conjugation by $b$) that fixes $Y$ and such that $\Psi^{3}$ is conjugation by $Y$.
Let $F=\langle X_0,Y_0,Z\rangle$ be the group arising from
Theorem~\ref{Z}, so that $F/E\cong C_{3}$, $Y_0$ has order $3$ modulo $E$, $Y_0^{3}=Y$, and $\Psi$ is conjugation by $Y_0$.
Then $|F|=3^{8}$, $Z=[X_0,Y_0]$, $X_0^Z=X_0^{\a}$ and ${}^Z Y_0=Y^\b$.
\end{proof}

\section{The case when $p=3$ exactly one of $\a,\b$ is $\equiv 7\mod 9$}

We suppose in this section that $p=3$ and exactly one of $\a,\b$ is congruent to 7 modulo 9. The isomorphism $G(\a,\b)\to G(\b,\a)$
allows us to assume without loss that $\a\equiv 7\mod 9$. Since we are assuming from the beginning that $p$ is a factor of both $\a-1$
and $\b-1$, we must have $\b\equiv 4\mod 9$ or $\b\equiv 1\mod 9$. We begin by supposing that $\b\equiv 4\mod 9$. In particular,
$v_3(\a-1)=1=v_3(\b-1)$ and $v_3(\a-\b)=1$ so (\ref{h3}) gives
$$
a^{27}=1=b^{27}.
$$
Moreover, by (\ref{lomismo}) and \cite[Proposition 2.1]{M}, we have $v_3(\d_\b)=2$ and $v_3(\d_\a)\geq 3$, so (\ref{prin1}) gives
$$
c^3=a^{27x}b^{9y}=b^{9y},
$$
where $3\nmid y$. In particular,
$$
c^9=1.
$$
This readily implies $b^9\in Z$, so by above $c^3\in Z$, whence
$$
a^9=1=b^9,
$$
and therefore
$$
c^3=1.
$$
Thus
$$
|G_3|\leq 3^5, a^3,b^3\in Z, c\in Z_2, Z_3=G_3.
$$

\begin{theorem}\label{teo8} If $\a\equiv 7\mod 9$ and $\b\equiv 4\mod 9$, then $e=5$, $f=3$, $o(a)=9=o(b)$, and $o(c)=3$.
\end{theorem}

\begin{proof} Consider the group $T=\langle X,Y,Z\rangle$ of order $27$
having defining relations $[X,Y]=[X,Z]=[Y,Z]=1$ and $X^{3}=Y^{3}=Z^{3}=1$. Here $X,Y,Z$ play the roles of $a^{3},
b^{3},c$, respectively. The assignment $X\mapsto X$, $Y\mapsto Y$, $Z\mapsto Z X^{-u}$ extends to an automorphism $\Omega$
of $T$ (which plays the role of conjugation by $a$) that fixes $X$ and such that $\Omega^{3}$ is conjugation by $X$, namely
trivial. Let $E=\langle X_0,Y,Z\rangle$ be the group arising from
Theorem \ref{Z}, so that $E/T\cong C_{3}$, $X_0$ has order $3$ modulo $T$, $X_0^{3}=X$, and $\Omega$ is conjugation by $X_0$.
Then $|E|=81$, with defining relations $Z^{X_0}=Z X_0^{1-\alpha}$, $[X_0,Y]=[Y,Z]=1$, $X^{9}=Y^{3}=Z^{3}=1$.
The assignment $X_0\mapsto X_0Z$, $Y\mapsto Y$, $Z\mapsto Y^v Z$ extends to an automorphism $\Psi$
of $E$ (which plays the role of conjugation by $b$) that fixes $Y$ and such that $\Psi^{3}$ is conjugation by $Y$.
Let $F=\langle X_0,Y_0,Z\rangle$ be the group arising from
Theorem \ref{Z}, so that $F/E\cong C_{3}$, $Y_0$ has order $3$ modulo $E$, $Y_0^{3}=Y$, and $\Psi$ is conjugation by $Y_0$.
Then $|F|=3^{5}$, $Z=[X_0,Y_0]$, $X_0^Z=X_0^{\a}$ and ${}^Z Y_0=Y^\b$.
\end{proof}

We finally assume that $\b\equiv 1\mod 9$. From (\ref{h3}), and following the convention specified by the end of the Introduction,
we find that $b^{\b^{(\a-\b)(\a-1)}-1}=1$ if $\a>\b$,
and $b^{\b_0^{(\b-\a)(\a-1)}-1}=1$ if $\b>\a$, where $\b_0$ is defined in Section \ref{secfr} and satisfies $\b\b_0\equiv 1\mod o(b)$.
Here $v_3(\b-1)=n\geq 2$, $v_3(\a-1)=1=v_3(\a-\b)$, and $v_3(\b_0-1)=n$ by Theorem \ref{finite}. Thus if $\a>\b$ then
$v_3(\b^{(\a-\b)(\a-1)}-1)=n+2$, and if $\b>\a$ then $v_3(\b_0^{(\b-\a)(\a-1)}-1)=n+2$.
Therefore
$$
b^{3^{n+2}}=1.
$$
On the other hand, by (\ref{prin1}) and \cite[Proposition 2.1]{M}, we have
$$
c^3=b^{3^{2n}x}a^{3^{2+s} y},
$$
where $3\nmid x$, $3\nmid y$, and $s=v_3(u+1)$. Here $s\geq 1$ since $u\equiv -1\mod 3$. As $2n\geq n+2$, we infer
\begin{equation}\label{93}
\langle c^3 \rangle=\langle a^{3^{2+s}}\rangle.
\end{equation}
On the other hand,  from $v_3(\b-1)=n\geq 2$, we readily see that $v_3(\l_\b)\geq n+2$, so (\ref{prin2}) gives
\begin{equation}\label{932}
\langle c^{3^{2n}} \rangle=\langle a^{9}\rangle.
\end{equation}
It follows from (\ref{93}) and (\ref{932}) that $a^9=1=c^3$. But then $1=[b,c^{-3}]=b^{\b^3-1}$, so
$$
b^{3^{n+1}}=1.
$$
Therefore $$a^9=1,b^{3^{n+1}}=1,a^3,b^3\in Z,c\in Z_2, Z_3=G_3,|G_3|\leq 3^{n+4}.$$

\begin{theorem}\label{teo9} If $\a\equiv 7\mod 9$ and $\b\equiv 1\mod 9$, then $e=n+4$, $f=3$,  $o(a)=9$, $o(b)=3^{n+1}$, and $o(c)=3$.
\end{theorem}

\begin{proof} Consider the group $T$ of order $3^{n+3}$ generated by elements $X,Y,Z$ subject to the defining relations
$[X,Y]=1$, $X^Z=X$, ${}^Z Y=Y^\b$, $X^{3}=Y^{3^{n+1}}=Z^3=1$. Here $X,Y,Z$ play the roles of $a^3,b,c$. The assignment
$X\mapsto X$, $Y\mapsto YZ^{-1}$, $Z\mapsto ZX^{-u}$ preserves the defining relations of $T$, so it extends
to an automorphism $\Omega$ of $T$ (which plays the role of conjugation by $a$). The only relation whose verification
is not trivial is ${}^{Z X^{-u}} (YZ^{-1})=(YZ^{-1})^\b$. The left hand side equals $Y^\b Z^{-1}$, while the right hand side
becomes $Z^{-1} Y^{\b^2(\b+1)/2}$, so we must verify that $Y^\b=Y^{\b^2(\b+1)/2}$.
Since $\b(\b+1)-2=3(\b-1)+(\b-1)^2$ and $Y^{3^{n+1}}=1$, this holds. It is clear that $\Omega$ fixes $X$ and $\Omega^3$
is conjugation by $X$, namely trivial. Let $E=\langle X_0,Y,Z\rangle$ be the group arising from
Theorem~\ref{Z}, so that $E/T\cong C_3$, $X_0$ has order 3 modulo $T$, $X_0^3=X$, and $\Omega$ is conjugation by $X_0$.
Then $|E|=3^{n+4}$, $Z=[X_0,Y]$, $X_0^Z=X_0^{\a}$ and ${}^Z Y=Y^\b$.
\end{proof}

\section{general facts when $p=2$}\label{b2}

We assume until further notice that $p=2$. By (\ref{lomismo}) and \cite[Proposition 2.1]{MS}, we have $v_2(\d_\a)=2m-1$ and $v_2(\d_\b)=2n-1$, and in fact
\begin{equation}\label{dadb}
\d_\a\equiv 2^{2m-1}u^2\mod 2^{3m},\; \d_\b\equiv 2^{2n-1}v^2\mod 2^{3n}.
\end{equation}
From $v_2(\d_\a)=2m-1$, $v_2(\d_\b)=2n-1$, and (\ref{adabc4}), we infer
$$
a^{2^{3m-1}}, b^{2^{3n-1}}\in Z,
$$
and therefore
$$
a^{2^{4m-1}}=1=b^{2^{4n-1}}.
$$
From $a^{2^{3m-1}}\in Z$, we deduce
$$
a^{2^{3m-1}}=(a^{2^{3m-1}})^b=c^{2^{3m-1}}a^{\a(\a^{2^{3m-1}}-1)/(\a-1)}=c^{2^{3m-1}}a^{2^{3m-1}}a^{2^{4m-2}},
$$
since now
$$
(\a^{2^{3m-1}}-1)/(\a-1)\equiv 2^{3m-1}+2^{4m-2}\mod 2^{4m-1}.
$$
Thus
\begin{equation}\label{ca}
c^{2^{3m-1}}=a^{2^{4m-2}}\in Z,\; c^{2^{3m}}=1.
\end{equation}
Likewise we obtain
\begin{equation}\label{cb3}
c^{2^{3n-1}}=b^{2^{4n-2}}\in Z,\; c^{2^{3n}}=1.
\end{equation}

Recall from (\ref{prin1}) that $b^{\b_0^{\b+1}\d_\b}a^{\d_\a}=c^{\a-\b}=a^{\a_0^{\a+1} \d_\a}b^{\d_\b}$, where $\a_0,\b_0$
are defined in Section \ref{secfr}, and satisfy $\a\a_0\equiv 1\mod o(a)$ and 
$\b\b_0\equiv 1\mod o(b)$. But $\a\equiv 1\mod 2^m$ and $\b\equiv 1\mod 2^n$, where
$2^m\mid o(a)$ and $2^n\mid o(b)$ by Theorem \ref{finite}, so $\a_0\equiv 1\mod 2^m$ and $\b_0\equiv 1\mod 2^n$. Since
$a^{2^{3m-1}},b^{2^{3n-1}}\in Z$, we deduce from (\ref{prin1}) that
\begin{equation}\label{qe}
a^{\d_\a}b^{\d_\b}z_1=c^{\a-\b}=b^{\d_\b}a^{\d_\a}z_2,
\end{equation}
where $z_1\in \langle a^{2^{3m-1}}\rangle\subseteq Z$ and $z_2\in \langle b^{2^{3n-1}}\rangle\subseteq Z$. In particular,
$a^{\d_\a}$ and $b^{\d_\b}$ commute modulo the central subgroup $Z_{0}=\langle a^{2^{3m-1}},b^{2^{3n-1}}\rangle$.

\section{The case when $m=1$  or $n=1$}

We assume here that $m=1$. By Section \ref{b2}, we have $a^4=c^4\in Z$ and $a^8=1=c^8$.
Here $[a,c^2]=a^{\a^2-1}=1$, since $v_2(\a^2-1)\geq 3$.

Suppose first that $n=1$ as well. Then $b^4=c^4\in Z$ and $b^8=1$, also by Section \ref{b2}.
Clearly $\a\equiv\b\mod 4$, so (\ref{prin1}) implies that
$a^2,b^2\in Z$, which forces $c^2\in Z$. From $1=[a^2,c]=a^{2(\a-1)}$ we deduce $a^4=1$,
so $c^4=b^4=1$. Hence $a^2=(a^2)^b$ gives $a^2=c^2$ via a careful calculation and $b^2=(b^2)^a$ yields $b^2=c^2$.
Moreover, we also have $a^c=a^3$ and $b^c=b^3$. As in \cite[Proposition 8.2]{MS}, we may now deduce that
$G_2\cong Q_{16}$, the generalized quaternion group of order 16.

\begin{theorem}\label{teo10} If $m=1$ and $n=1$, then $e=4$, $f=3$, $o(a)=o(b)=o(c)=4$, and $G_2\cong Q_{16}$.
\end{theorem}

Suppose next that $n>1$. Then $v_2(\a-\b)=1$. It follows from (\ref{h2})
that $c^4\in\langle b\rangle$, so $b^{2^{n+2}}=1$. On the other hand, by (\ref{adabc4}), $\langle a^4\rangle=\langle b^{2^{3n-1}}\rangle$,
whence $a^4=c^4=1$.

Assume first that $n\geq 3$. Then $2n-1\geq n+2$, so $a^{\d_\a}=b^{-\b_0^{\b+1} \d_\b}c^{\a-\b}$ yields
$a^{\d_\a}=c^{\a-\b}$, whence $a^2=c^2$, and therefore
$
(a^2)^b=c^2 a^{\a(1+\a)}=a^2,
$
so $b^{c^2}=b$ and hence $b^{2^{n+1}}=1$. This implies $(b^4)^a=b^4$, that is, $b^{2^4}\in Z$.
It follows that $c\in Z_2$ and $Z_3=G_2$. Clearly $|G_2|\leq 2^{n+4}$, because $a^4=1$, $c^2=a^2$,
and $b^{2^{n+1}}=1$.

\begin{theorem}\label{teo11} If $m=1$ and $n>2$, then $e=n+4$, $f=3$, $o(a)=4=o(c)$, and $o(b)=2^{n+1}$.
\end{theorem}

\begin{proof} The group $T=C_{2^{n-1}}\times Q_8$ is generated by elements $X,Y,Z$ subject to the defining relations
$[X,Y]=1=[X,Z]$, $X^{2^{n-1}}=1$, $Y^2=Z^2$, $Z^Y=Z^{-1}$. Here $X,Y,Z$ play the roles of $b^4,a,c$. The assignment
$X\mapsto X$, $Y\mapsto YZ$, $Z\mapsto X^{2^{n-2}v} Z$ preserves the defining relations of $T$, so it extends
to an automorphism $\Omega$ of $T$ (which plays the role of conjugation by $b$). Here $\Omega$ fixes $X$ and $\Omega^4$
is conjugation by $X$, namely trivial. Let $E=\langle X_0,Y,Z\rangle$ be the group arising from
Theorem \ref{Z}, so that $E/T\cong C_4$, $X_0$ has order 4 modulo $T$, $X_0^4=X$, and $\Omega$ is conjugation by $X_0$.
Then $|E|=2^{n+4}$, $Z=[Y,X_0]$, ${}^Z X_0=X_0^{\b}$ and $Y^Z=Y^\alpha$.
\end{proof}

Assume finally that $n=2$. Then $b^{16}=1$. From $b^{\b_0^{\b+1}\d_\b}a^{\d_\a}=c^{\a-\b}$, we deduce
$b^8=a^2c^2$, a central element of $G_2$. Then $|G_2|\leq 2^{7}$, because $a^4=1=c^4$
and $b^{8}=a^2c^2$. Note that
$$b^8\in Z, b^4, c^2, a^2\in Z_2, c\in Z_3, Z_4=G_2,
$$
so is the nilpotency class is at most 4 in this case.

\begin{theorem}\label{teo12} If $m=1$ and $n=2$, then $e=7$, $f=4$, $o(a)=4=o(c)$, and $o(b)=8$.
\end{theorem}

\begin{proof} Consider the group $T=\langle X,Z\,|\, X^{4}=1=Z^4, X^Z=X^{-1}\rangle$ of order 16. Here $X,Z$ play the roles of $a,c$. The assignment
$X\mapsto X^{-1}Z^2$, $Z\mapsto Z$ preserves the defining relations of $T$, so it extends
to an automorphism $\Omega$ of $T$ (which plays the role of conjugation by $b^4$). Here $\Omega$ fixes $X^2Z^2$ and $\Omega^2$
is conjugation by $X^2Z^2$, namely trivial. Let $E=\langle X,Z,Y\rangle$ be the group arising from
Theorem \ref{Z}, so that $E/T\cong C_2$, $Y$ has order 2 modulo $T$, $Y^2=X^2Z^2$, and $\Omega$ is conjugation by~$Y$.
Then $|E|=32$, with defining relations $X^{4}=1=Z^4$, $X^Z=X^{-1}$, $Y^2=X^2Z^2$, $X^Y=X^{-1}Z^2$, $Z^Y=Z$.
The assignment
$X\mapsto XZ$, $Y\mapsto Y$, $Z\mapsto Y^v Z$ preserves the defining relations of $E$, so it extends
to an automorphism $\Psi$ of $E$ (which plays the role of conjugation by $b$). Here $\Psi$ fixes $Y$ and $\Psi^4$
is conjugation by $Y$. Let $F=\langle X,Z,Y_1\rangle$ be the group arising from
Theorem \ref{Z}, so that $F/E\cong C_4$, $Y_1$ has order 4 modulo~$E$, $Y_1^4=Y$, and $\Psi$ is conjugation by $Y_1$.
Then $|F|=128$, $Z=[X,Y_1]$, $X^Z=X^{\a}$ and ${}^Z Y_1=Y^\b$.
\end{proof}

\section{Generalities of the case $m,n>1$}\label{mnmas1}

We assume until Section \ref{larga} inclusive that $m,n>1$. We have
$
v_2(\l_\b)=3n-2,
$
so
$$
\l_\b=2^{3n-2}x=y,
$$
where $x$ is odd. The operator $[a,-]$ applied to (\ref{prin2}) gives
$$
1=[a,b^{\l_\b}c^{\d_\b}]=[a,b^{y}c^{\d_\b}].
$$
The identity $[X,YZ]=[X,Z][X,Y]^Z$, valid in any group, gives
$$
1=a^{\a^{\d_\b}-1}(b^{(\b-1)(\b+2\b^2+\cdots+(y-1)\b^{y-1})}c^y)^{c^{\d_\b}}.
$$
Here
$$
v_2((\b-1)(\b+2\b^2+\cdots+(y-1)\b^{y-1}))=4n-3\geq 3n-1,
$$
so
$$
b^{(\b-1)(\b+2\b^2+\cdots+(y-1)\b^{y-1})}\in Z.
$$
Therefore
\begin{equation}\label{P2}
a^{2^{m+2n-1}w_1}=b^{2^{4n-3}w_2}c^{2^{3n-2}w_3},
\end{equation}
where $b^{2^{4n-3}w_2}\in Z$ and $w_1$ and $w_3$ are odd. The usual transformation yields
\begin{equation}\label{P3}
b^{2^{n+2m-1}z_1}=a^{2^{4m-3}z_2}c^{2^{3m-2}z_3},
\end{equation}
where $a^{2^{4m-3}z_2}\in Z$ and $z_1$ and $z_3$ are odd.

If $m\geq n$ then from $b^{2^{3n-1}}\in Z$ we deduce $b^{2^{n+2m-1}}\in Z$, so (\ref{P3}) forces
$c^{2^{3m-2}}\in Z$, and therefore
$a^{2^{4m-2}}=1$, whence $c^{2^{3m-1}}=1$ by (\ref{ca}). Likewise, if $n\geq m$ we obtain  
$b^{2^{4n-2}}=1$ and $c^{2^{3n-1}}=1$.

We may assume without loss that $m\geq n$ and we do so for the remainder of this section.

Squaring (\ref{P2}) and making use of $b^{2^{4n-3}}\in Z$ and (\ref{cb3}) yields
\begin{equation}\label{P4}
a^{2^{m+2n}}=b^{2^{4n-2}w_4}.
\end{equation}
If $m=n$ then $b^{2^{4n-2}}=1$ by above, so (\ref{P4}) gives
$$
a^{2^{m+2n}}=1=[c^{2^{2n}},a].
$$
If $m-n=f>0$, raising $a^{2^{m+2n}}=b^{2^{4n-2}w_4}$ to the $2^f$th power, we obtain
$$
a^{2^{2m+n}}=b^{2^{4n+f-2}w_4}=1,
$$
since $b^{2^{4n-1}}=1$, as seen in Section \ref{b2}, and therefore
$$
a^{2^{2m+n}}=1=[c^{2^{m+n}},a].
$$
Squaring (\ref{P3}) and appealing to $a^{2^{4m-3}}\in Z$ and (\ref{ca}) produces
$$
b^{2^{n+2m}}=a^{2^{4m-2}w_5}=1,
$$
because $4m-2\geq 3m\geq 2m+n$. Therefore
$$
b^{2^{n+2m}}=1=[c^{2^{2m}},b].
$$
All in all, we infer
$$
c^{2^{2m}}\in Z.
$$
As above, we have $v_2(\l_\a)=3m-2$. Since $a^{2^{3m}}=1$ and $4m-2\geq 3m$, we infer
$[a^{\l_\a},c]=1$. Also,
$a^{2^{3m-1}}\in Z$, $v_2(\d_\a)=2m-1$, and $c^{2^{2m}}\in Z$,
so squaring (\ref{prin2}) yields
$$
b^{2^{n+\ell+1}}\in Z.
$$

The proof of the following result is similar to that of Proposition \ref{abcabelian} and will be omitted.

\begin{prop}\label{abc2abe} Set $t=1+2^{m-1}v$, and let $H$ be a group with elements $x_1,x_2,x_3$ and an automorphism $\Psi$ such that
for some integer $0\leq g\leq m-1$, we have
$$
x_1^{x_3}=x_1^\a, {}^{x_3} x_2=x_2^\b, x_1^{2^{2m-1}}=x_2^{2^{2m-g-1}}=1,
$$
$$
x_1^{\Psi}=x_1, x_2^\Psi=x_3^{-2^{m+g}}x_2^{t}, x_3^{\Psi}=x_3x_1^{-u}.
$$
Then $[x_1,x_3^{2^{m-1}}]=1=[x_2,x_3^{2^{m-1}}]$, $(x_3^{-2^m})^\Psi=x_1^{2^m u} x_3^{-2^m}$, $x_3^{\Psi^{2^m}}=x_3^{x_1}$, and
\begin{equation}\label{x_123}
x_2^{\Psi^i}=x_1^{2^{m+g}u(i-1)i/2} x_3^{-2^{m+g}(1+t+\cdots+t^{i-1})} x_2^{t^i},\quad i\geq 1.
\end{equation}
In particular, if $x_3^{2^{2m+g}}=1$, then $x_2^{\Psi^{2^m}}=x_2$, so if $a^{2^{3m-1}}=b^{2^{3m-1}}=c^{2^{2m+g}}=1$,
then $\langle a^{2^{m}}, b^{2^{m+g}}, c^{2^{m-1}}\rangle$ is a normal abelian subgroup of $G_2$.
\end{prop}

\section{The case when $\ell=n$}\label{elenen}

We assume here that $\ell=n$, so that $m>n$, for if $m=n$ then $\ell>m=n$.
It follows from Section \ref{mnmas1} that $b^{2^{2n+1}}\in Z$, so
$b^{2^{3n+1}}=1$. Now $b^{2^{2n+1}}\in Z$ and
$b^{2^{3n+1}}=1$ imply that $c^{2^{2n+1}}=b^{2^{3n}}$ and $c^{2^{2n+2}}=1$.
As $m>n$, from $c^{2^{2n+2}}=1$ we infer $c^{2^{2m+1}}=1$.
From $c^{2^{2m+1}}=1$ and $a^{2^{3m}}=1$, we deduce $(a^{2^{2m+1}})^b=a^{2^{2m+1}}$, so $a^{2^{2m+1}}\in Z$.
We know that $a^{\d_\a}$ and $b^{\d_\b}$ commute modulo the central subgroup $Z_0$ defined in Section \ref{b2}.
Thus, raising $a^{\d_\a}=b^{-\b_0^{\b+1} \d_\b}c^{\a-\b}$ to
the 4th power and using $a^{2^{2m+1}}\in Z$ and $b^{2^{2n+1}}\in Z$, we find that $c^{2^{n+2}}\in Z$,
which implies $a^{2^{m+n+2}}=1=b^{2^{2n+2}}$, so raising $a^{\d_\a}=b^{-\b_0^{\b+1} \d_\b}c^{\a-\b}$ to the 8th power gives
$c^{2^{n+3}}=1$. This and $a^{2^{m+n+2}}=1=b^{2^{2n+2}}$ imply $a^{2^{n+3}},b^{2^{n+3}}\in Z$. As $m+n\geq n+3$, $a^{\a^{\a-\b}-1}\in Z$,
so squaring (\ref{prin2}) and using $3n-1\geq n+3$ yields $c^{2^n}\in Z$,
and therefore $a^{2^{m+n}}=1=b^{2^{2n}}$. As $2m-1\geq m+n$, (\ref{prin1})
gives $c^{2^n}=b^{2^{2n-1}x}$, with $x$ odd. As $b^{2^{2n}}=1$,
we see that $c^{2^n}=b^{2^{2n-1}}$, $c^{2^{n+1}}=1$, $a^{2^{n+1}}\in Z$, and $|G_2|\leq 2^{m+4n}$. Also, $b^{2^n}\in Z$, since
$(b^{2^n})^a=c^{-2^n}b^{2^n}b^{2^{2n-1}}=b^{2^n}$. Thus $c\in Z_2$ and $Z_3=G$.

\begin{theorem}\label{teo13} If $m>n=\ell$, then $e=m+4n$, $f=3$, $o(a)=2^{m+n}$, $o(b)=2^{2n}$, and $o(c)=2^{n+1}$.
\end{theorem}

\begin{proof} Consider the group $T=\langle X,Y,Z\rangle$ of order $2^{3n}$
having defining relations $[X,Y]=[X,Z]=[Y,Z]=1$ and $X^{2^n}=Y^{2^n}=1$, $Z^{2^n}=Y^{2^{n-1}}$. Here $X,Y,Z$ play the roles of $a^{2^m},
b^{2^n},c$, respectively. The assignment $X\mapsto X$, $Y\mapsto Y$, $Z\mapsto Z X^{-u}$ extends to an automorphism $\Omega$
of $T$ (which plays the role of conjugation by $a$) that fixes $X$ and such that $\Omega^{2^m}$ is conjugation by $X$, namely
trivial. Let $E=\langle X_0,Y,Z\rangle$ be the group arising from
Theorem \ref{Z}, so that $E/T\cong C_{2^m}$, $X_0$ has order $2^m$ modulo $T$, $X_0^{2^m}=X$, and $\Omega$ is conjugation by $X_0$.
Then $|E|=2^{m+3n}$, with defining relations $Z^{X_0}=Z X_0^{1-\alpha}$, $[X_0,Y]=[Y,Z]=1$, $X^{2^{m+n}}=Y^{2^n}=1$, $Z^{2^n}=Y^{2^{n-1}}$.
The assignment $X_0\mapsto X_0Z$, $Y\mapsto Y$, $Z\mapsto Y^v Z$ extends to an automorphism $\Psi$
of $E$ (which plays the role of conjugation by $b$) that fixes $Y$ and such that $\Psi^{2^n}$ is conjugation by $Y$.
Let $F=\langle X_0,Y_0,Z\rangle$ be the group arising from
Theorem \ref{Z}, so that $F/E\cong C_{2^n}$, $Y_0$ has order $2^n$ modulo $E$, $Y_0^{2^n}=Y$, and $\Psi$ is conjugation by $Y_0$.
Then $|F|=2^{m+4n}$, $Z=[X_0,Y_0]$, $X_0^Z=X_0^{\a}$ and ${}^Z Y_0=Y^\b$.
\end{proof}

\section{The case when $\ell\geq 2m$}

We assume in this section that $m=n$ and $\ell\geq 2m$. Since  $c^{2^{2m}}\in Z$, we infer  $c^{2^{\ell}}\in Z$.
It follows from $a^{\d_\a}=b^{-\b_0^{\b+1} \d_\b}c^{\a-\b}$ that $a^{2^{2m-1}},b^{2^{2m-1}}\in Z$. This forces
$c^{2^{2m-1}}\in Z$, so $a^{2^{3m-1}}=1=b^{2^{3m-1}}$,
\begin{equation}\label{puma}
c^{2^{2m-1}}=a^{2^{3m-2}}=b^{2^{3m-2}}, c^{2^{2m}}=1.
\end{equation}
From $c^{2^{2m}}=1$, $a^{\d_\a}=b^{-\b_0^{\b+1} \d_\b}c^{\a-\b}$,
and $b^{2^{2m-1}}\in Z$, we see that
$$
a^{\d_\a}b^{\d_\b}=1.
$$
Here $\d_\a\equiv 2^{2m-1}u^2\mod 2^{3m-1}$, $\d_\b\equiv 2^{2m-1}v^2\mod 2^{3m-1}$, and $a^{2^{3m-1}}=1=b^{2^{3m-1}}$, so
$
a^{2^{2m-1}u^2}b^{2^{2m-1}v^2}=1.
$
But $u\equiv v\mod 2^{\ell-m}$ and $m+\ell-1\geq 3m-1$, so
\begin{equation}\label{puma2}
a^{2^{2m-1}}b^{2^{2m-1}}=1.
\end{equation}
It is clear from the above relations that $G_2=\langle a\rangle \langle b\rangle \langle c \rangle$ has order at most $2^{7m-3}$ and we have
$$
\langle a^{2^{2m-1}}\rangle=\langle b^{2^{2m-1}}\rangle\subseteq Z, a^{2^{2m-1}}, c^{2^{m-1}}\in Z_2,
a^{2^{m}}, b^{2^{m}}, c^{2^{m-1}}\in Z_3,  a^{2^{m-1}}, b^{2^{m-1}}, c\in Z_4, Z_5=G_2,
$$
as in Case 2 from \cite{MS}.

\begin{theorem}\label{teo14} If $m=n$ and $\ell\geq 2m$, then $e=7m-3$, $f=5$, $o(a)=2^{3m-1}=o(b)$, and $o(c)=2^{2m}$.
\end{theorem}

\begin{proof} This is an extension of \cite[Theorem 11.4]{MS}, whose proof still applies with minor modifications
indicated below. Missing calculations can be found in the proof of \cite[Theorem 11.4]{MS}.

We wish to construct an image
of $G_2$ of order $2^{7m-3}$. Taking into account Proposition \ref{abc2abe}, (\ref{puma}), and (\ref{puma2}), we start with an abelian group
$\langle x,y,z\rangle$ of order $2^{4m-2}$ generated by elements $x,y,z$ subject to the
defining relations:
    \[
    xy=yx,\;
    xz=zx,\;
    yz=zx,\;
    z^{2^m}=x^{2^{2m-2}},\;
    x^{2^{m-1}}y^{2^{m-1}}=1,\;
    x^{2^{2m-1}}=1.
    \]
	Here $x,y,z$ play the roles of
$a^{2^m},b^{2^m},c^{2^{m-1}}$, respectively. 	

 We first construct a cyclic extension $\langle x,y,z_0\rangle$
of $\langle x,y,z\rangle$ of order $2^{5m-3}$, where
$z_0^{2^{m-1}}=z$, by means of an automorphism $\Omega$
of $\langle x,y,z\rangle$ that is conjugation by $z_0$.
This is achieved by 
\[
    x\mapsto x^\alpha,\;
    y\mapsto y^\gamma,\;
    z\mapsto z,
    \]
		where $\gamma=1-2^m v$ is the inverse of $\beta$ modulo $2^{2m}$. We see that $\langle x,y,z_0\rangle$ has defining relations:
    \[
		xy = yx,\;
    x^{z_0} = x^\alpha,\;
    {}^{z_0}y=y^\beta,\;
    z_0^{2^{2m-1}} = x^{2^{2m-2}},\;
		x^{2^{m-1}} y^{2^{m-1}} = 1,\;
    x^{2^{2m-1}} = 1.
    \]		

We next construct a cyclic extension $\langle x_0,y,z_0\rangle$
of $\langle x,y,z_0\rangle$ of order $2^{6m-3}$, where
$x_0^{2^{m}}=x$, by means of an automorphism $\Psi$
of $\langle x,y,z_0\rangle$ that is conjugation by $x_0$. This
is achieved by 
\[
    x\mapsto x,\;
    y\mapsto z_0^{-2^m}y^{1+2^{m-1}v}=z^{-2} y^{1+2^{m-1}v},\;
    z_0\mapsto z_0x^{-u}.
    \]
		By Proposition \ref{abc2abe}, $\Psi^{2^m}$ is conjugation by $x$. We see that $\langle x_0,y,z_0\rangle$ has defining relations:
    \[
    y^{x_0} = z_0^{-2^m}y^{1+2^{m-1}v},\;
    x_0^{z_0} = x_0^\alpha,\;
    {}^{z_0}y = y^\beta,\;
    z_0^{2^{2m-1}} = x_0^{2^{3m-2}},\;
		x_0^{2^{2m-1}}y^{2^{m-1}}=1,\;
		x_0^{2^{3m-1}} = 1.
    \]
We finally construct a cyclic extension $\langle x_0,y_0,z_0\rangle$
of $\langle x_0,y,z_0\rangle$ of order $2^{7m-3}$, where
$y_0^{2^{m}}=y$, by means of an automorphism $\Pi$
of $\langle x_0,y,z_0\rangle$ that is conjugation by $y_0$. This
is achieved by 
\[
    x_0\mapsto x_0z_0,\;
    y\mapsto y,\;
    z_0\mapsto y^vz_0.
    \]
We include the verification that the first and second defining relations of $\langle x_0,y,z_0\rangle$ are preserved,
as these require modifications, especially the relation $x_0^{z_0} = x_0^\alpha$. Replacing $k$ by $v$, $\a$ by $\b$,
and $\b$ by $\gamma$, the  argument given in the proof of \cite[Theorem 11.4]{MS} shows that
$$
y^{x_0z_0}=z_0^{-2^m} y^{\gamma + 2^{m-1}v}=(y^v z_0)^{-2^m} y^{1+2^{m-1}v}.
$$
Thus, the relation $y^{x_0} = z_0^{-2^m}y^{1+2^{m-1}v}$ is preserved. Regarding the relation $x_0^{z_0} = x_0^\alpha$,
the argument given in the proof of \cite[Theorem 11.4]{MS} shows that
$$
(x_0z_0)^\a=z_0^\a x_0^{1+2^{m+1} u+3\times 2^{2m-1} u^2}=z_0^\a x_0^{1+2^{m+1}u+2^{2m-1}u^2+2^{2m}u^2}.
$$
The calculation of $(x_0z_0)^{y^vz_0}$ requires more work. From $z_0y^v z_0^{-1}=y^{v\b}$, we infer
$z_0^{y^v}=y^{2^m v^2} z_0$. Moreover, from $x_0^{-1}y^v x_0=z^{-2v}y^{v(1+2^{m-1}v)}$, we deduce
$x_0^{y^v}=x_0y^{-2^{m-1} v^2} z^{2v}$. Thus
$$
(x_0 z_0)^{y^v}=x_0 y^{-2^{m-1}v^2}z_0^{2^m v} y^{2^m v^2} z_0=x_0 y^{2^{m-1}v^2} z_0^{\b},
$$
$$
(x_0 z_0)^{y^v z_0}=(x_0 y^{2^{m-1}v^2} z_0^{\b})^{z_0}=x_0^\a y^{2^{m-1}v^2} z_0^{\b}=
x_0^\a x_0^{-2^{2m-1}u^2} z_0^{2^\ell k} z_0^{\b}.
$$
Since $\a=\b+2^\ell k$, we see that
$$
(x_0 z_0)^{y^v z_0}=z_0^\a x_0^{\a\a^\b\a^{2^\ell k}}x_0^{-2^{2m-1}u^2}=z_0^\a x_0^{1+2^{m+1}u+2^{2m-1}u^2+2^{2m}uv+2^{m+\ell}uk}.
$$
As 
$$
2^{2m}uv+2^{m+\ell}uk-2^{2m}u^2=2^{2m}u(v-u)+2^{m+\ell}uk=(-2^{\ell-m}k)2^{2m}u+2^{m+\ell}uk=0,
$$
$x_0^{z_0} = x_0^\alpha$ is preserved. Thus $\Pi$ is a surjective endomorphism and hence an automorphism of $\langle x_0,y,z_0\rangle$.
The verification that $\Pi^{2^{m}}$ is conjugation by~$y$ can be achieved as in the proof of \cite[Theorem 11.4]{MS} by merely replacing
$k$ by $v$ and $\alpha$ by $\beta$.

This produces the required extension $\langle x_0,y_0,z_0\rangle$.
We already had $x_0^{z_0}=x_0^\alpha$.
		Moreover, the new relation $z_0^{y_0}=y^v z_0=y_0^{2^m v}z_0=y_0^{\beta-1}z_0$ is equivalent to ${}^{z_0} y_0=y_0^{\beta}$.
		Furthermore, from $x_0^{y_0}=x_0 z_0$ we infer $[x_0,y_0]=z_0$, so $\langle x_0,y_0,z_0\rangle=\langle x_0,y_0\rangle$
		is an image of $G_2$ of order $2^{7m-3}$.
\end{proof}

\section{The case when $m=n<\ell<2m$}\label{larga}

We assume in this section that $m=n<\ell<2m$. Our work from Sections \ref{b2} and \ref{mnmas1}
gives $a^{2^{3m-1}}, b^{2^{3m-1}}\in Z$, $a^{2^{m+\ell+1}}, b^{2^{m+\ell+1}}\in Z$, $c^{2^{2m}}\in Z$, and $c^{2^{3m-1}}=1=a^{2^{3m}}=
b^{2^{3m}}$.
Note that $a^{2^{m+\ell+1}}\in Z$ and $a^{2^{3m}}=1$ imply that
$$c^{2^{m+\ell+1}}=1.$$

We clearly have $a^{2^{2m}}, b^{2^{2m}}, c^{2^{\ell+1}}\in Z_2$, $a^{2^{\ell+1}}, b^{2^{\ell+1}}, c^{2^{m}}\in Z_3$.
If $\ell<2m-1$, it follows that $a^{2^{m}}, b^{2^{m}}\in Z_4$, $c\in Z_5$, $Z_6=G_2$.

When $\ell=2m-1$ we have $Z_5=G_2$, as shown below. This completes
the proof that the class of $G_p$, $p\neq 3$, is always at most 6.

Raising (\ref{prin1}) to
the $(2^{2m-\ell})$th power and using $c^{2^{2m}}\in Z$, yields $$a^{2^{4m-(\ell+1)}},b^{2^{4m-(\ell+1)}}\in Z.$$

Three cases arise: $2\ell+2=3m+1$; $2\ell+2>3m+1$; and $2\ell+2\leq 3m$.

Suppose first that $2\ell+2\geq 3m+1$. This is equivalent to $m+\ell\geq 4m-(\ell+1)$. Note also
that $3m-2\geq 4m-(\ell+1)$, which is equivalent to $\ell\geq m+1$.
We have $v_2(\l_\a)=3m-2\geq 4m-(\ell+1)$, $v_2(\d_\a)=2m-1$, and
$v_2(\b^{\b-\a}-1)=m+\ell\geq 4m-(\ell+1)$. As $a^{2^{4m-(\ell+1)}},b^{2^{4m-(\ell+1)}}\in Z$, it follows from (\ref{prin2}) that
$$c^{2^{2m-1}}\in Z.$$ This implies
$$a^{2^{3m-1}}=1=b^{2^{3m-1}},\; a^{\a_0^{\a+1}\d_\a}=a^{\d_\a},\; b^{\b_0^{\b+1} \d_\b}=b^{\d_\b}.$$
We deduce from (\ref{prin1}) that
\begin{equation}
\label{fuj}
b^{\d_\b}a^{\d_\a}=c^{\a-\b}=a^{\d_\a}b^{\d_\b},
\end{equation}
whence
$$
[a^{\d_\a},b^{\d_\b}]=1, \text{ that is, }[a^{2^{2m-1}},b^{2^{2m-1}}]=1.
$$
Raising (\ref{fuj}) to the power $2^{2m-(\ell+1)}$ and using $c^{2^{2m-1}}\in Z$ and $[a^{\d_\a},b^{\d_\b}]=1$ gives
\begin{equation}
\label{fuj2}
a^{2^{4m-(\ell+2)}},b^{2^{4m-(\ell+2)}}\in Z.
\end{equation}

It is convenient at this point to deal with the special case $\ell=2m-1$. Then (\ref{prin2}) yields $c^{2^{2m}}=1$ and
$a^{2^{3m-2}}=c^{2^{2m-1}}=b^{2^{3m-2}}$. On the other hand,  by (\ref{fuj}), we have $a^{2^{2m-1}u^2}b^{2^{2m-1}v^2}=c^{2^{2m-1}}$,
where $u\equiv v\mod 2^{m-1}$, so $u^2\equiv v^2\mod 2^{m}$, whence  $(a^{2^{2m-1}}b^{2^{2m-1}})^{u^2}=c^{2^{2m-1}}$
and therefore $a^{2^{2m-1}}b^{2^{2m-1}}=c^{2^{2m-1}}$. In particular, $a^{2^{2m-1}},b^{2^{2m-1}}\in Z$, which also follows from
(\ref{fuj2}). We deduce that
$|G_2|\leq 2^{3m-1}\times 2^{2m-1} \times 2^{2m-1}=2^{7m-3}$.
It is clear that
$$a^{2^{2m-1}}, b^{2^{2m-1}},c^{2^{2m-1}}\in Z, c^{2^{m-1}}\in Z_2, a^{2^{m}}, b^{2^{m}}\in Z_3, c\in Z_4, Z_5=G_2.
$$

\begin{theorem}\label{teo15} If $m=n$ and $\ell=2m-1$, then $e=7m-3$, $f=5$, $o(a)=2^{3m-1}=o(b)$, and $o(c)=2^{2m}$.
\end{theorem}

\begin{proof} Taking into account Proposition \ref{abc2abe}, we start with an abelian group
$\langle x,y,z\rangle$ of order $2^{4m-2}$ generated by elements $x,y,z$ subject to the
defining relations:
    \[
    xy=yx,\;
    xz=zx,\;
    yz=zx,\;
    x^{2^{m-1}}y^{2^{m-1}}=x^{2^{2m-2}}=y^{2^{2m-2}}=z^{2^{m}},\;
    x^{2^{2m-1}}=1.
    \]
	Here $x,y,z$ play the roles of
$a^{2^m},b^{2^m},c^{2^{m-1}}$, respectively. The proof of Theorem \ref{teo14} goes through essentially unchanged.
\end{proof}

Note that if $m=2$ then the condition $m<\ell<2m$ forces $\ell=3=2m-1$. Thus,
we may assume from now on that $\ell<2m-1$ and $m\geq 3$.

We continue to suppose that $2\ell+2\geq 3m+1$, $\ell\leq 2m-2$ and $m\geq 3$. Then by~(\ref{fuj2}) 
\begin{equation}
\label{fuj3}
c^{2^{4m-(\ell+2)}}=1.
\end{equation}
We note that $\ell\leq 2m-2$ is required in this calculation. Raising $c^{\a-\b}=a^{\d_\a}b^{\d_\b}$
to the power $2^{m-1}$, we see that
$c^{2^{m+\ell-1}w}=a^{2^{3m-2}}b^{2^{3m-2}}$, where $m+\ell-1\geq 4m-(\ell+2)$, so $b^{2^{3m-2}}=a^{2^{3m-2}}$.
Then (\ref{prin2}) shows that
$c^{2^{2m-1}}\in\langle a\rangle$, so we can write every element of $\langle a\rangle \langle c\rangle$
in the form $a^i c^j$, where $0\leq i<2^{3m-1}$ and $0\leq j<2^{2m-1}$. From $c^{\a-\b}=a^{\d_\a}b^{\d_\b}$, we see that
we can write every element of $G_2$ in the form $a^i c^j b^s$, where $0\leq i<2^{3m-1}$, $0\leq j<2^{2m-1}$, and
$0\leq s<2m-1$. Thus $|G_2|\leq 2^{3m-1}\times 2^{2m-1} \times 2^{2m-1}=2^{7m-3}$.

It is convenient at this point to deal with the special case $\ell=2m-2$. In this case, (\ref{prin2}) yields
$b^{2^{3m-2}}=c^{2^{2m-1}}a^{2^{3m-2}}$, whence $c^{2^{2m-1}}=1$. Moreover,
$$a^{2^{2m}}, b^{2^{2m}}\in Z, c^{2^{m}}, a^{2^{2m-1}}, b^{2^{2m-1}}\in Z_2, a^{2^{m}}, b^{2^{m}}\in Z_3, c\in Z_4, Z_5=G_2.
$$

\begin{theorem}\label{teo16} If $m=n\geq 3$ and $\ell=2m-2$, then $e=7m-3$, $f=5$, $o(a)=2^{3m-1}=o(b)$, and $o(c)=2^{2m-1}$.
\end{theorem}

\begin{proof} Consider the abelian group of order $2^{4m-2}$ generated by $x,y,z$ subject to the defining relations
$[x,y]=[x,z]=[y,z]=1$ as well as
\[
    x^{2^{2m-2}}=y^{2^{2m-2}},\;
    x^{2^{2m-1}}=1=z^{2^{m}},\; z^{2^{m-1}}=x^{2^{m-1}u^2}y^{2^{m-1}v^2},
\]		
where $x,y,z$ play the roles of $a^{2^m}, b^{2^m}, c^{2^{m-1}}$, respectively. By means of the automorphism $\Omega$
used in the proof of Theorem \ref{teo14} we construct an extension $\langle x,y,z_0\rangle$ of $\langle x,y,z\rangle$
of order $2^{5m-3}$, where $z_0^{2^{m-1}}=z$,
and having defining relations $xy = yx$, $x^{z_0} = x^\alpha$, ${}^{z_0}y=y^\beta$ we well as
\[
		x^{2^{2m-2}}=y^{2^{2m-2}},\;
    z_0^{2^{2m-1}}=1=x^{2^{2m-1}},\;
		z_0^{2^{2m-2}}=x^{2^{m-1}u^2} y^{2^{m-1}v^2}.
		\]		

We next construct a cyclic extension $\langle x_0,y,z_0\rangle$
of $\langle x,y,z_0\rangle$ of order $2^{6m-3}$, where
$x_0^{2^{m}}=x$, by means of an automorphism $\Psi$
of $\langle x,y,z_0\rangle$ that is conjugation by $x_0$. This
is achieved by the same automorphism $\Psi$ used in the proof of Theorem~\ref{teo14}.
The verification that the defining relations of $\langle x_0,y,z_0\rangle$ are preserved
goes through as in the proof of Theorem \ref{teo14}, except for the last one, which requires changes.
Note that
\[
    (z_0x^{-u})^{2^{2m-2}}
    = z_0^{2^{2m-2}} x^{-u(1 + \alpha + \cdots + \alpha^{2^{2m-2}-1})}= z_0^{2^{2m-2}} x^{2^{2m-2}},
    \]
	since $(\alpha^{2^{2m-2}}-1)/(\alpha-1)\equiv  2^{2m-2}\mod 2^{2m-1}$. Observe also that
	$$
	(z^{-2} y^{1+2^{m-1}v})^{2^{m-1}}=z^{-2^m} y^{2^{m-1}} y^{2^{2m-2}}= y^{2^{m-1}}x^{2^{2m-2}}.
	$$
Thus the last defining relation of $\langle x_0,y,z_0\rangle$ is preserved.  It follows from Proposition \ref{abc2abe}
that $\Psi^{2^m}$ is conjugation by $x$. This produces the
required extension  $\langle x_0,y,z_0\rangle$, which has defining relations $y^{x_0} = z_0^{-2^m}y^{1+2^{m-1}v}$,
    $x_0^{z_0} = x_0^\alpha$, ${}^{z_0}y = y^\beta$, and
\[
   x_0^{2^{3m-2}}=y^{2^{2m-2}},\;
    z_0^{2^{2m-1}} = 1= x_0^{2^{3m-1}},\;
		z_0^{2^{2m-2}}=x_0^{2^{2m-1}u^2} y^{2^{m-1}v^2}.
		\]

We finally construct a cyclic extension $\langle x_0,y_0,z_0\rangle$
of $\langle x_0,y,z_0\rangle$ of order $2^{7m-3}$, where
$y_0^{2^{m}}=y$, by means of an automorphism $\Pi$
of $\langle x_0,y,z_0\rangle$ that is conjugation by $y_0$. This
is achieved by the same automorphism $\Pi$ used in the proof of Theorem~\ref{teo14}.
The preservation of the first, third, and fifth relations is easily verified. Regarding the fourth relation,
we have
\[
    (x_0z_0)^{2^{3m-2}}
    = z_0^{2^{3m-2}} x_0^{\alpha (1 + \alpha + \cdots + \alpha^{2^{3m-2}-1})}
    = z_0^{2^{3m-2}} x_0^{2^{3m-2}}
    = x_0^{2^{3m-2}},
    \]
    since $(\alpha^{2^{3m-2}}-1)/(\alpha-1)\equiv 2^{3m-2}\mod 2^{3m-1}$ and $z_0^{2^{2m-1}} = 1= x_0^{2^{3m-1}}$.
In regards to the sixth relation, observe that
\[
    (x_0z_0)^{2^{2m-1}}
    = z_0^{2^{2m-1}} x_0^{\alpha (1 + \alpha + \cdots + \alpha^{2^{2m-1}-1})}
    = x_0^{2^{2m-1} - 2^{3m-2}},
    \]
    since $\alpha(\alpha^{2^{2m-1}}-1)/(\alpha-1)\equiv 2^{2m-1} - 2^{3m-2}\mod 2^{3m-1}$. On the other hand,
		\[
    (y^v z_0)^{2^{2m-2}}
    = y^{v(1+\alpha +\cdots +\alpha^{2^{2m-2}-1})}z_0^{2^{2m-2}}
    = y^{2^{2m-2}} z_0^{2^{2m-2}}
    = x_0^{2^{3m-2}} z_0^{2^{2m-2}},
    \]
		since $(\alpha^{2^{2m-2}}-1)/(\alpha-1)\equiv  2^{2m-2} \mod 2^{2m-1}$ and $y^{2^{2m-1}}=1$. Thus the sixth 
		relation is preserved. The preservation of $x_0^{z_0} = x_0^\alpha$ can be achieved
		as the proof of Theorem \ref{teo14}. The rest of the proof can be continued as in the proof of Theorem \ref{teo14}.
\end{proof}

We continue to assume  $2\ell+2\geq 3m+1$ and also $\ell\leq 2m-3$, which forces $m\geq 5$.
From (\ref{prin2}), we deduce
$b^{2^{3m-2}}c^{2^{2m-1}v^2}=a^{2^{\ell+m}ku}$. Since $b^{2^{3m-2}}=a^{2^{3m-2}}$, we infer
\begin{equation}\label{swim}
c^{2^{2m-1}}=a^{2^{\ell+m}ku/v^2}a^{-2^{3m-2}}=a^{2^{\ell+m}(1-2^{2m-2-\ell})k/u},
\end{equation}
where we have used that $a^{2^{3m-1}}=1$, $u^2\equiv v^2\mod 2^{\ell-m+1}$, and $2\ell+1\geq 3m$.
We note that $1-2^{2m-2-\ell}$ is odd as $\ell\leq 2m-3$.

Likewise, (\ref{prin2}) yields
$a^{2^{3m-2}}c^{2^{2m-1}u^2}=b^{2^{\ell+m}kv}=b^{2^{\ell+m}ku}$, since $2\ell\geq 3m-1$, so
\begin{equation}\label{swim2}
c^{2^{2m-1}}=b^{2^{\ell+m}(1-2^{2m-2-\ell})k/u}.
\end{equation}
As $1-2^{2m-2-\ell}$ is odd, we infer
\begin{equation}\label{eswim}
a^{2^{\ell+m}}=b^{2^{\ell+m}}.
\end{equation}

On the other hand, from $c^{\a-\b}=a^{\d_\a}b^{\d_\b}$, we have
\begin{equation}\label{upq}
c^{2^\ell k}= a^{2^{2m-1} u^2}b^{2^{2m-1} v^2},
\end{equation}
where the factors on the right hand side commute. Raising (\ref{upq}) to the power $2^{\ell-m}$, we obtain
\begin{equation}\label{swim3}
c^{2^{2\ell-m} k}= a^{2^{\ell+m-1} u^2}b^{2^{\ell+m-1} v^2}=a^{2^{\ell+m-1} u^2}b^{2^{\ell+m-1} u^2},
\end{equation}
since $u^2\equiv v^2\mod 2^{\ell-m+1}$ and $2\ell\geq 3m-1$.

We assume next that $2\ell+2=3m+1$. Then (\ref{swim3}) translates into
\begin{equation}\label{swim4}
c^{2^{2m-1} k}=a^{2^{(5m-3)/2} u^2}b^{2^{(5m-3)/2} u^2},
\end{equation}
while (\ref{swim}) and (\ref{swim2}) become
\begin{equation}\label{swim5}
a^{2^{(5m-1)/2}(1-2^{(m-3)/2})k^2/u}=c^{2^{2m-1}k}=b^{2^{(5m-1)/2}(1-2^{(m-3)/2})k^2/u},
\end{equation}
and (\ref{eswim}) becomes
\begin{equation}\label{swim6}
a^{2^{(5m-1)/2}}=b^{2^{(5m-1)/2}}.
\end{equation}
From (\ref{swim4}) and (\ref{swim5}), we deduce
$$
a^{2^{(5m-3)/2} u^2}b^{2^{(5m-3)/2} u^2}=a^{2^{(5m-1)/2}(1-2^{(m-3)/2})k^2/u}=b^{2^{(5m-1)/2}(1-2^{(m-3)/2})k^2/u}.
$$
Therefore
$$
a^{2^{(5m-3)/2} u^2}=b^{2^{(5m-3)/2} [2k^2(1-2^{(m-3)/2})/u-u^2]},\; b^{2^{(5m-3)/2} u^2}=a^{2^{(5m-3)/2}[2k^2(1-2^{(m-3)/2})/u-u^2]},
$$
and hence
$$
a^{2^{(5m-3)/2} u^4}=b^{2^{(5m-3)/2} u^2 [2k^2(1-2^{(m-3)/2})/u-u^2]}=a^{2^{(5m-3)/2} [2k^2(1-2^{(m-3)/2})/u-u^2]^2},
$$
which implies
$$
a^{2^{(5m+1)/2}(u^3-k^2(1-2^{(m-3)/2}))}=1=b^{2^{(5m+1)/2}(u^3-k^2(1-2^{(m-3)/2}))}.
$$

On the other hand, from the beginning of the section, we know that
\begin{equation}\label{serie}
a^{2^{m+\ell+1}}, b^{2^{m+\ell+1}}, c^{2^{2m}}\in Z, a^{2^{2m}}, b^{2^{2m}}, c^{2^{\ell+1}}\in Z_2, a^{2^{\ell+1}},
b^{2^{\ell+1}}, c^{2^{m}}\in Z_3, a^{2^{m}}, b^{2^{m}}\in Z_4, c\in Z_5, Z_6=G_2.
\end{equation}

Set $s_0=v_2(u^3-k^2(1-2^{(m-3)/2}))\geq 1$, so that
\begin{equation}\label{swim11}
a^{2^{(5m+2s_0+1)/2}}=1=b^{2^{(5m+2s_0+1)/2}}.
\end{equation}
Raising (\ref{swim5}) to the power $2^{s_0+1}$, we obtain
\begin{equation}\label{swim12}
c^{2^{2m+s_0}}=1.
\end{equation}

We already know from the beginning of the section that $a^{2^{3m-1}}=1$ as well as $c^{2^{(5m-3)/2}}=1$ from (\ref{fuj3}).
For (\ref{swim11}) to yield additional
information, we need
$$
(5m+2s_0+1)/2<3m-1,\text{ that is, }s_0<(m-3)/2,
$$
which forces $m\geq 7$. Regarding (\ref{swim12}), note that $2m+s_0<(5m-3)/2$ is also equivalent to $s_0<(m-3)/2$.


Suppose first $s_0\geq (m-3)/2$. We know from the beginning of the section that $a^{2^{3m-1}}=1=b^{2^{3m-1}}=c^{2^{(5m-3)/2}}$,
$c^{2^{2m-1}}\in\langle a\rangle$
and $b^{2^{2m-1}}\in\langle a,c\rangle$, so $|G_2|\leq 2^{7m-3}$. 

Suppose next $s_0<(m-3)/2$. Then $m\geq 7$. Since $c^{2^{2m-1}}\in\langle a\rangle$
and $b^{2^{2m-1}}\in\langle a,c\rangle$, we see that
$$
|G_2|\leq 2^{(5m+2s_0+1)/2}\times 2^{2m-1}\times 2^{2m-1}=2^{(13m+2s_0-3)/2}.
$$
Note that  $2^{(13m+2s_0-3)/2}<7m-3$ (the above bound) is also equivalent to $s_0<(m-3)/2$.

Set $s=v_2(u^3-k^2)$. Since $s_0=v_2(u^3-k^2+k^2 2^{(m-3)/2})$, we have $s\geq (m-3)/2\Leftrightarrow s_0\geq (m-3)/2$,
and $s=s_0$ if $s_0<(m-3)/2$, that is, $s<(m-3)/2$.

\begin{theorem}\label{teo17} Suppose that $m=n\geq 5$ and $2\ell+2=3m+1$. Then
$f=6$. If $s<(m-3)/2$, then $m\geq 7$, $e=(13m+2s-3)/2$,
$o(a)=2^{(5m+2s+1)/2}=o(b)$,
and $o(c)=2^{2m+s}$. If $s\geq (m-3)/2$,
then $e=7m-3$, $o(a)=2^{3m-1}=o(b)$, and $o(c)=2^{(5m-3)/2}$.
\end{theorem}

\begin{proof} Set $r=(m-3)/2$, $t=1-2^{r}$, $q=\min\{s, r\}$, and consider the abelian group generated by $x,y,z$ subject to the defining relations
$
[x,y]=[x,z]=[y,z]=1=x^{2^{\ell+q+1}},
$
as well as
$$
z^{2^{\ell-(m-1)}k}=x^{2^{m-1}u^2}y^{2^{m-q-1}v^2},
x^{2^{\ell}tk}=z^{2^{m}u},\; x^{2^{\ell}}=y^{2^{\ell-q}},
$$
where $x,y,z$ play the roles of $a^{2^m}, b^{2^{m+q}}, c^{2^{m-1}}$, respectively. Note that $y^{2^{\ell+1}}=1=z^{2^{m+\ell+1}}$.

We claim that $\langle x,y,z\rangle$ has order $2^{(7m-1)/2}$. Indeed, passing to an additive notation, we can view
$\langle x,y,z\rangle$ as the quotient of a free abelian
group with basis $\{X,Y,Z\}$ by the subgroup generated by $2^{\ell+q+1}X, 2^{m-1}u^2 X+2^{m-q-1}v^2 Y-2^{\ell-(m-1)}k Z,
2^{\ell} X-2^{\ell-q}Y, 2^{\ell}tk X-2^{m}u Z$. Thus, the matrix whose columns are the coordinates of these generators relative
to the basis $\{X,Y,Z\}$ is
$$
M=\begin{pmatrix} 2^{\ell+q+1} & 2^{m-1}u^2 & 2^{\ell} & 2^{\ell}tk\\
0 & 2^{m-q-1}v^2 & -2^{\ell-q} & 0\\
0 & -2^{\ell-m+1}k & 0 & -2^{m}u\end{pmatrix}.
$$
The order of $\langle x,y,z\rangle$ is the absolute value of the product of the invariant factors of $M$. 
As $\langle x,y,z\rangle$ is clearly a finite 2-group, to compute
these invariant factors, we may view $M$ as a matrix over the localization $\Z_{(2)}$ of $\Z$ at (2). The $\gcd$
of the entries of $M$ is easily seen to be $2^{\ell-m+1}$, which is thus the first invariant factor of~$M$.

Let us perform the following columns and row operations on $M$ in the given order: $C_4\to k C_4$, $C_4\to C_4-u 2^{2m-(\ell+1)}C_2$, $R_2\leftrightarrow R_3$, 
$R_1\leftrightarrow R_2$, $C_1\leftrightarrow C_2$. The resulting matrix is
$$
\begin{pmatrix} -2^{\ell-m+1}k & 0 & 0 & 0\\
0 & 2^{\ell+q+1} & 2^{\ell} & 2^{\ell-1}(2tk^2-u^3)\\
0 & 0 & -2^{\ell-q} & -2^{\ell-q-1}uv^2\end{pmatrix}.
$$
Let $N$ the $2\times 3$ matrix obtained by deleting the first row and the first column of $M$.
The $\gcd$ of the entries of $N$ is easily seen to be $2^{\ell-q+1}$, which is thus the second invariant factor of $M$.
Let $d_1,d_2,d_3$ the determinants of the 3 submatrices $N_1,N_2,N_3$ of $N$ of size $2\times 2$, obtained by deleting
columns 3, 2, and 1, respectively, and let $d=\gcd(d_1,d_2,d_3)$.
Here $d_1=2^{2\ell+1}$, $d_2=2^{2\ell}$ and $d_3=2^{2\ell-(q+1)}v_2(2tk^2-u(u^2+v^2))$.
Now $u^2\equiv v^2\mod 2^{\ell-m+1}$, so $u(u^2+v^2)\equiv 2u^3\mod 2^{\ell-m+1}$, and therefore 
$2tk^2-u(u^2+v^2)\equiv 2tk^2-2u^3\equiv 2(k^2-u^3)-2^{r+1}k^2\mod 2^{\ell-m+1}$, where $r+1=\ell-m$.
If $s<r$, then $q+1=s+1<r+1$, so $d_3=2^{2\ell}$ and $d=2^{2\ell}$. 
If $s\geq r$, then $q+1=r+1\geq s+1$, so $d_3=2^{2\ell+j}$ with $j\geq 0$, and  $d=2^{2\ell}$.
Thus $d=2^{2\ell}$ in both cases. Therefore, the third invariant factor of $M$ is $2^{2\ell}/2^{\ell-q+1}$.
It follows that the order of $\langle x,y,z\rangle$ is $2^{\ell-m+1} 2^{2\ell}=2^{(7m-1)/2}$.

We next construct a cyclic extension $\langle x,y,z_0\rangle$
of $\langle x,y,z\rangle$ of order $2^{(9m-3)/2}$, where
$z_0^{2^{m-1}}=z$, by means of an automorphism $\Omega$
of $\langle x,y,z\rangle$ that fixes $z$ and such that
$\Omega^{2^{m-1}}$ is conjugation by~$z$, that is, the trivial automorphism. In order to achieve this goal,
we consider the assignment
		\[
    x\mapsto x^\alpha,\;
    y\mapsto y^\gamma,\;
    z\mapsto z,
    \]
		where $\gamma=1-2^m v$ is the inverse of $\beta$ modulo $2^{2m}$. 
		The defining relations  of  $\langle x,y,z\rangle$ are easily seen to be preserved. 
		Thus the above assignment extends to an endomorphism $\Omega$ of $\langle x,y,z\rangle$ which is clearly surjective
		and hence an automorphism of $\langle x,y,z\rangle$. Since $\alpha^{2^{m-1}}\equiv 1 \mod 2^{2m-1}$ and 
		$\gamma^{2^{m-1}}\equiv 1 \mod 2^{2m-1}$, we see that $\Omega^{2^{m-1}}$ is the trivial automorphism.
		This produces the required extension, where $\Omega$ is conjugation by $z_0$. We see that $\langle x,y,z_0\rangle$ has defining relations:
		\[
		xy = yx,\;
    x^{z_0} = x^\alpha,\;
    {}^{z_0}y=y^\beta,\;
		x^{2^{\ell+q+1}}=1,\]
		\[
   z_0^{2^{\ell}k}=x^{2^{m-1}u^2}y^{2^{m-q-1}v^2},
   x^{2^{\ell}tk}=z_0^{2^{2m-1}u},\; x^{2^{\ell}}=y^{2^{\ell-q}}.
	\]
		
We next construct a cyclic extension $\langle x_0,y,z_0\rangle$
of $\langle x,y,z_0\rangle$ of order $2^{(11m-3)/2}$, where
$x_0^{2^{m}}=x$, by means of an automorphism $\Psi$
of $\langle x,y,z_0\rangle$ that fixes $x$ and such that
$\Psi^{2^{m}}$ is conjugation by~$x$. For this purpose,
we consider the assignment
	 \[
    x\mapsto x,\;
    y\mapsto z_0^{-2^{m+q}}y^{1+2^{m-1}v}=z^{-2^{1+q}} y^{1+2^{m-1}v},\;
    z_0\mapsto z_0x^{-u}.
    \]
		Let us verify that the defining relations of $\langle x,y,z_0\rangle$ are preserved.
		This is easily seen to be true for $xy = yx$, $x^{z_0} = x^\alpha$, $x^{2^{\ell+q+1}}=1$,
	$x^{2^{\ell}}=y^{2^{\ell-q}}$, and $x^{2^{\ell}tk}=z_0^{2^{2m-1}u}$. Regarding ${}^{z_0}y=y^\beta$, we have
	\[
    {}^{(z_0x^{-u})}(z^{-2^{1+q}}y^{1+2^{m-1}v})= z^{-2^{1+q}} y^{\beta (1+2^{m-1}v)}=z^{-2^{1+q}\beta} y^{\beta (1+2^{m-1}v)}=
		(z^{-2^{1+q}}y^{1+2^{m-1}v})^\beta,
    \]
		as $2^{1+q}\beta\equiv 2^{1+q}\mod 2^{m+q+1}$. The preservation of $z_0^{2^{\ell}k}=x^{2^{m-1}u^2}y^{2^{m-q-1}v^2}$
		is delicate. On the one hand, we have
		$$
		(z_0x^{-u})^{2^{\ell}k}=z_0^{2^{\ell}k} x^{-u(\a^{2^{\ell} k}-1)/(\a-1)}=z_0^{2^{\ell}k}x^{-2^\ell u k},
		$$
		since $x^{2^{\ell+q+1}}=1$ and $(\ell-1)+m\geq \ell+q+1$. Other other hand,
		$$
		(z^{-2^{1+q}} y^{1+2^{m-1}v})^{2^{m-q-1}v^2}=z^{-2^{m}v^2}y^{2^{m-q-1}v^2}x^{2^{\ell+r}},
		$$
		using $2m-r-2=\ell$, $x^{2^{\ell}}=y^{2^{\ell-q}}$, and $x^{2^{\ell+q+1}}=1$.
		Thus $z_0^{2^{\ell}k}=x^{2^{m-1}u^2}y^{2^{m-q-1}v^2}$ is preserved if and only if
		$$
		x^{2^\ell u k}=z^{2^{m}v^2}x^{2^{\ell+r}}.
		$$
		Now $u\equiv v\mod 2^{\ell-m}$, so $u^2\equiv v^2\mod 2^{\ell-m+1}$. Since $\ell+1\geq m+r+1$, we infer that
		$$
		z^{2^{m}v^2}=z^{2^{m}u^2}=x^{2^{\ell} t k u}=x^{2^{\ell}k u(1-2^{r})}=x^{2^{\ell}k u}x^{2^{\ell+r}},
		$$
		as needed. Thus the above assignment extends to an endomorphism $\Psi$ of $\langle x,y,z_0\rangle$.
		As $\mathrm{im}(\Psi)$ contains $x,y^{1+2^{m-1}v},z_0$, with $m>1$, $\Psi$ is
		surjective and hence an automorphism of $\langle x,y,z_0\rangle$. 
		
		Since $z_0^{2m+s}=1$, Proposition~\ref{abc2abe} ensures that $\Psi^{2^{m}}$ is conjugation by $x$. This produces the required extension, where $\Psi$ is conjugation by $x_0$. We see that $\langle x_0,y,z_0\rangle$ has defining relations:
    \[
    y^{x_0} = z_0^{-2^{m+q}}y^{1+2^{m-1}v},\;
    x_0^{z_0} = x_0^\alpha,\;
    {}^{z_0}y = y^\beta,\;
    x_0^{2^{\ell+m+q+1}}=1,
    \]
		\[
		z_0^{2^{\ell k}}=x_0^{2^{2m-1}u^2}y^{2^{m-q-1}v^2},
x_0^{2^{\ell+m}tk}=z_0^{2^{2m-1}u},\; x_0^{2^{\ell+m}}=y^{2^{\ell-q}}.
    \]
		
We next construct a cyclic extension $\langle x_0,y_0,z_0\rangle$
of $\langle x_0,y,z_0\rangle$ of order $2^{(11m+2q-3)/2}$, where
$y_0^{2^{q}}=y$, by means of an automorphism $\Pi$
of $\langle x_0,y,z_0\rangle$ that fixes $y$ and such that
$\Pi^{2^{q}}$ is conjugation by~$y$. For this purpose, we consider the assignment
$$
x_0\mapsto x_0 y^{-2^{m-q-1}v} z_0^{2^m}=x_0 y^{-2^{m-q-1}v} z^2,\; y\mapsto y,\; z_0\mapsto y^{2^{m-q}v}z_0.
$$
We claim that all defining relations of $\langle x_0,y,z_0\rangle$ are preserved, in which case the given
assignment extends to an endomorphism of $\langle x_0,y,z_0\rangle$, which is then clearly an automorphism.

$\bullet$ ${}^{z_0}y = y^\beta$. This is clearly preserved.

$\bullet$ $x_0^{z_0} = x_0^\alpha$. We need to show that
\begin{equation}
\label{redif}
(x_0 y^{-2^{m-q-1}v} z^2)^{y^{2^{m-q}v}z_0}=(x_0 y^{-2^{m-q-1}v} z^2)^\a.
\end{equation}
We first compute the right hand side of (\ref{redif}). Set $h=-2^{m-q-1}v$. Then
\begin{equation}
\label{redif2}
(x_0 y^{h} z^2)^\a=x_0^\a (y^{h} z^2)^{x_0^{\a-1}}(y^{h} z^2)^{x_0^{\a-2}}\cdots (y^{h} z^2)^{x_0} (y^{h} z^2).
\end{equation}
The calculation of (\ref{redif2}) requires that we know how to conjugate $y^h$ and $z^2$ by $x_0^i$, $i\geq 1$. From,
$
y^{x_0}=z^{-2^{1+q}}y^{1+2^{m-1}v}
$
we infer
$$
(y^{h})^{x_0}=z^{-2^{1+q}h}y^{(1+2^{m-1}v)h}=z^{2^m}x^{2^{\ell+r}}y^h ,
$$
using $(m-1)+(m-r-1)=2m-r-2=\ell$, $x^{2^{\ell}}=y^{2^{\ell-q}}$, and $x^{2^{\ell+q+1}}=1$. 
Noting that $[x_0,z^{2^m}]=[x_0, z_0^{2^{2m-1}}]=1$, since $x_0^{2^{3m-1}}=1$, we deduce
\begin{equation}
\label{redif3}
(y^{h})^{x_0^i}=z^{2^m i}x^{2^{\ell+r}i}\, y^h ,\quad i\geq 1.
\end{equation}
On the other hand, from $\a^{2^{m}}\equiv 1+2^{2m} u\mod 2^{3m-1}$, we successively derive
$$x_0^{z^2}=x_0^{z_0^{2^m}}=x_0^{\a^{2^{m}}}=x_0^{1+2^{2m} u},
$$
$$
(z^2)^{x_0}=z^2 x_0^{2^{2m}u}=z^2 x^{2^{m}u},
$$
\begin{equation}
\label{redif4}
(z^2)^{x_0^i}=z^2 x^{2^{m}ui},\quad i\geq 1.
\end{equation}
Combining (\ref{redif2})-(\ref{redif4}) we obtain
\begin{equation}
\label{redif5}
(x_0 y^{h} z^2)^\a=x_0^\a y^h z^{2\a}.
\end{equation}
 Regarding the left hand side of (\ref{redif}), by $y^{x_0}=z^{-2^{1+q}}y^{1+2^{m-1}v}$ and $2m-r-1=\ell+1$, 
$$
(y^{2^{m-q}v})^{x_0}=z^{-2^{m+1}v}y^{2^{m-q}v},
$$
$$
x_0^{y^{2^{m-r}v}}=x_0 z^{2^{m+1}v},
$$
$$
(x_0 y^{h} z^2)^{y^{2^{m-r}v}z_0}=(x_0 z^{2^{m+1}v} y^{h} z^2)^{z_0}.
$$
Here $[y^h,z_0]=1$, since $2m-r-1=\ell+1$. Hence
\begin{equation}
\label{redif6}
(x_0 y^{h} z^2)^{y^{2^{m-r}v}z_0}=x_0^\a z^{2^{m+1}v} y^{h} z^2.
\end{equation}
By (\ref{redif5}) and (\ref{redif6}), we see that  (\ref{redif}) holds if and only if
$$
z^2 z^{{2^{m+1}v}}=z^2 z^{2^{m+1} u},
$$
which is true because $(m+1)+(\ell-m)=\ell+1>m+r+1$.

$\bullet$ $x_0^{2^{\ell+m+q+1}}=1$ and $x_0^{2^{\ell+m}}=y^{2^{\ell-q}}$. These follows easily from  (\ref{redif2})-(\ref{redif4}).


$\bullet$ $x_0^{2^{\ell+m}tk}=z_0^{2^{2m-1}u}$. From (\ref{redif2})-(\ref{redif4}), we see that
$$
(x_0 y^{h} z^2)^{2^{\ell+m}tk}=x_0^{2^{\ell+m}tk}.
$$
On the other hand, since $(m-r)+(2m-1)=3m-1-q\geq \ell+1$,
$$
(y^{2^{m-q}v}z_0)^{2^{2m-1}u}=(y^{2^{m-q}v})^{1+\b+\cdots+\b^{2^{2m-1}u}}z_0^{2^{2m-1}u}=z_0^{2^{2m-1}u}.
$$

$\bullet$ $z_0^{2^{\ell}k}=x_0^{2^{2m-1}u^2}y^{2^{m-q-1}v^2}$. This follows as above and by appealing to (\ref{redif2})-(\ref{redif4}).

$\bullet$ $y^{x_0} = z_0^{-2^{m+q}}y^{1+2^{m-1}v}$. Arguing as above, we see that 
$$
(y^{2^{m-q}v}z_0)^{-2^{m+q}}y^{1+2^{m-1}v}=z_0^{-2^{m+q}}y^{1+2^{m-1}v},
$$
while on the other hand
$
y^{x_0 y^{-2^{m-q-1}v} z^2}=y^{x_0}.
$

We next claim that $\Pi^{2^q}$ is conjugation by $y$. This is clear for $y$. As for $z_0$, from $z_0 y z_0^{-1}=y^{1+2^m v}$,
we deduce that $z_0^{y}=y^{2^m v}z_0= z_0 \Pi^{2^q}$. Regarding $x_0$, note that $\Pi$ fixes $z$, using that $2m-r-1=\ell+1$,
so
$$
x_0 \Pi^{2^q}= x_0 y^{-2^{m-1}v}z^{2^{q+1}}.
$$
On the other hand,  from $y^{x_0}=z^{-2^{1+q}}y^{1+2^{m-1}v}$, we deduce
$$
x_0^{y}=x_0 y^{-2^{m-1}v}z^{2^{q+1}}.
$$

This produces the required extension, where $\Pi$ is conjugation by $y_0$. We readily verify that $\langle x_0,y_0,z_0\rangle$ has defining relations:
    \[
    x_0^{y_0} =x_0 y_0^{-2^{m-1}v} z_0^{2^m},\;
    x_0^{z_0} = x_0^\alpha,\;
    ^{z_0} y_0=y_0^\b,\;
    x_0^{2^{\ell+m+q+1}}=1,
    \]
		\[
		z_0^{2^{\ell}k}=x_0^{2^{2m-1}u^2}y_0^{2^{m-1}v^2},
x_0^{2^{\ell+m}tk}=z_0^{2^{2m-1}u},\; x_0^{2^{\ell+m}}=y_0^{2^{\ell}}.
    \]

We finally construct a cyclic extension $\langle x_0,y_1,z_0\rangle$
of $\langle x_0,y,z_0\rangle$ of order $2^{(13m+2q-3)/2}$, where
$y_1^{2^{m}}=y_0$, by means of an automorphism $\Lambda$
of $\langle x_0,y_0,z_0\rangle$ that fixes $y_0$ and such that
$\Lambda^{2^{m}}$ is conjugation by~$y_0$. For this purpose, we consider the assignment
$$
x_0\mapsto x_0 z_0,\; y_0\mapsto y_0,\; z_0\mapsto y_0^v z_0.
$$
We claim that all defining relations of $\langle x_0,y_0,z_0\rangle$ are preserved, in which case the given
assignment extends to an endomorphism of $\langle x_0,y,z_0\rangle$, which is then clearly an automorphism.

$\bullet$ $^{z_0} y_0=y_0^\b$. This is clear.

$\bullet$ $x_0^{2^{\ell+m}}=y_0^{2^{\ell}}$. From $2m+\ell-1>3m-1$ and $\ell+m>2m+r$, we obtain
$$
(x_0 z_0)^{2^{\ell+m}}=z_0^{2^{\ell+m}}x_0^{\a(\a^{2^{\ell+m}}-1)/(\a-1)}=x_0^{2^{\ell+m}}.
$$

$\bullet$ $x_0^{2^{\ell+m+q+1}}=1$. This follows from
$$
(x_0 z_0)^{2^{\ell+m+q+1}}=z_0^{2^{\ell+m+q+1}}x_0^{\a(\a^{2^{\ell+m+q+1}}-1)/(\a-1)}=1.
$$

$\bullet$ $x_0^{2^{\ell+m}tk}=z_0^{2^{2m-1}u}$. By above, $(x_0 z_0)^{2^{\ell+m}tk}=x_0^{2^{\ell+m}tk}$.
Using $2m-1\geq \ell+r+1$,  we derive
$$
(y_0^v z_0)^{2^{2m-1}u}=y_0^{v (\b^{2^{2m-1}u}-1)/(\b-1)}z_0^{2^{2m-1}u}=z_0^{2^{2m-1}u}.
$$

$\bullet$ $z_0^{2^{\ell}k}=x_0^{2^{2m-1}u^2}y_0^{2^{m-1}v^2}$. On the one hand, we have
$$
(y_0^v z_0)^{2^{\ell}k}=y_0^{v (\b^{2^{\ell}k}-1)/(\b-1)}z_0^{2^\ell k}=y_0^{2^{\ell} v k} z_0^{2^\ell k},
$$
using $\ell-1+m\geq \ell+r+1$. On the other hand,
$$
(x_0 z_0)^{2^{2m-1}u^2}=z_0^{2^{2m-1}u^2}x_0^{\a(\a^{2^{2m-1}u^2}-1)/(\a-1)}=z_0^{2^{2m-1}u^2}x_0^{{2^{2m-1}u^2}}x_0^{2^{3m-2}},
$$
where $2^{3m-2}=\ell+m+r$, using $\a(\a^{2^{2m-1}u^2}-1)/(\a-1)\equiv 2^{2m-1}u^2+2^{3m-2}\mod 2^{3m-1}$.
We are thus reduced to show that $y_0^{2^{\ell} v k}=x_0^{2^{3m-2}} z_0^{2^{2m-1}u^2}$. This is true,
since 
$y_0^{2^{\ell} v k}=x_0^{2^{\ell+m}vk}=x_0^{2^{\ell+m}uk},
$
using that $u\equiv v\mod 2^{\ell-m}$ and $x_0^{2^{2\ell}}=1$, and
$
x_0^{2^{3m-2}}z_0^{2^{2m-1}u^2}=x_0^{2^{3m-2}}x_0^{2^{\ell+m}utk},
$
where 
$
x_0^{2^{\ell+m}uk(1-t)}=x_0^{2^{\ell+m+r}uk}=x_0^{2^{3m-2}}.
$

$\bullet$ $x_0^{y_0} =x_0 y_0^{-2^{m-1}v} z_0^{2^m}$. On the one hand, we have
$$
(x_0z_0)^{y_0}=x_0 y_0^{-2^{m-1}v} z_0^{2^m} y_0^{2^m v}z_0=x_0 y_0^{-2^{m-1}v} y_0^{2^m v} z_0^{2^{m+1}},
$$
and on the other hand, since $y_0^{2^{2m-1}}=1$, 
$$
x_0 z_0 y_0^{-2^{m-1}v} (y_0^v z_0)^{2^m}=x_0 y_0^{-2^{m-1}v} z_0 y_0^{v(\b^{2^m}-1)/(\b-1)} z_0^{2^m}=x_0 y_0^{-2^{m-1}v} 
y_0^{2^m v} z_0^{2^{m+1}}.
$$

$\bullet$ $x_0^{z_0} = x_0^\alpha$. The proof of Theorem \ref{teo14}, with $y_0$ instead of $y$, applies.

We finally show that $\Lambda^{2^m}$ is conjugation by $y_0$. This is obvious for $y_0$ and $z_0$. The result for $x_0$ follows as
in the proof of \cite[Theorem 11.4]{MS}.

This produces the required extension, where $\Lambda$ is conjugation by $y_1$. It is now clear that $\langle x_0,y_1,z_0\rangle=
\langle x_0,y_1\rangle$ is an image of $G_2$ of the required order.
\end{proof}

Suppose next that $2\ell+2>3m+1$ and $\ell\leq 2m-3$. Squaring (\ref{swim3}) and using (\ref{eswim}) yields
\begin{equation}\label{nadar}
c^{2^{2\ell-m+1} k}= a^{2^{\ell+m+1} u^2}.
\end{equation}
On the other hand, squaring (\ref{swim}) gives
\begin{equation}\label{nadar2}
c^{2^{2m}}=a^{2^{\ell+m+1}(1-2^{2m-2-\ell})k/u}.
\end{equation}
Since the right hand sides of (\ref{nadar}) and (\ref{nadar2}) generate the same subgroup, so do the left hand sides.
But $2m<2\ell-m+1$, so
$$
c^{2^{2m}}=1=a^{2^{\ell+m+1}}=b^{2^{\ell+m+1}},
$$
and therefore
$$
c^{2^{\ell+1}}, a^{2^{2m}}, b^{2^{2m}}\in Z, a^{2^{\ell+m}}=b^{2^{\ell+m}}=c^{2^{2m-1}}\in Z.
$$
The comments at the beginning of this section imply that $G_2$ has class at most 5, with
$$
c^{2^{\ell+1}},a^{2^{2m}},b^{2^{2m}}\in Z, a^{2^{\ell+1}}, b^{2^{\ell+1}}, c^{2^{m}}\in Z_2,
a^{2^{m}},b^{2^{m}}\in Z_3, c\in Z_4, Z_5=G_2.
$$
From $a^{2^{\ell+m+1}}=1$, $c^{2^{2m-1}}=a^{2^{\ell+m}}$, and $b^{2^{2m-1}}\in\langle a,c\rangle$ we see that
$|G_2|\leq 2^{5m+\ell-1}$. Raising (\ref{upq}) to the power $2m-\ell$ yields
$
1=a^{2^{4m-\ell-1}}b^{2^{4m-\ell-1}}.
$
Here $4m-\ell-1<m+\ell$, so
$$
1=a^{2^{m+\ell-1}}b^{2^{m+\ell-1}}.
$$

\begin{theorem}\label{teo18} Suppose that $m=n\geq 3$, $2\ell+2>3m+1$, and $\ell\leq 2m-3$. Then
$e=5m+\ell-1$, $f=5$,
$o(a)=2^{\ell+m+1}=o(b)$, and $o(c)=2^{2m}$.
\end{theorem}

\begin{proof} Consider the abelian group of order $2^{2m+\ell}$ generated by $x,y,z$ subject to the defining relations
$[x,y]=[x,z]=[y,z]=1$ as well as
\[
    x^{2^{\ell}}=z^{2^{m}},\;
    x^{2^{\ell+1}}=1,\; x^{2^{\ell-1}}y^{2^{\ell-1}}=1,\;
		z^{2^{\ell-(m-1)} k}= x^{2^{m-1}u^2}y^{2^{m-1}v^2},
\]		
where $x,y,z$ play the roles of $a^{2^m}, b^{2^m}, c^{2^{m-1}}$, respectively. 

We next construct a cyclic extension $\langle x,y,z_0\rangle$
of $\langle x,y,z\rangle$ of order $2^{3m+\ell-1}$, where
$z_0^{2^{m-1}}=z$, by means of an automorphism $\Omega$ 
of $\langle x,y,z\rangle$ that is conjugation by $z_0$. This is achieved by 
		\[
    x\mapsto x^\alpha,\;
    y\mapsto y^\gamma,\;
    z\mapsto z,
    \]
		where $\gamma=1-2^m v$ is the inverse of $\beta$ modulo $2^{2m}$. We see that $\langle x,y,z_0\rangle$ has defining relations:
    \[
		xy = yx,\;
    x^{z_0} = x^\alpha,\;
    {}^{z_0}y=y^\beta,\;
		x^{2^{\ell}}=z_0^{2^{2m-1}},\; x^{2^{\ell+1}}=1,\; x^{2^{\ell-1}}y^{2^{\ell-1}}=1,\;
		z_0^{2^\ell k}= x^{2^{m-1}u^2}y^{2^{m-1}v^2}.
	\]
	
	We next construct a cyclic extension $\langle x_0,y,z_0\rangle$
of $\langle x,y,z_0\rangle$ of order $2^{4m+\ell-1}$, where
$x_0^{2^{m}}=x$, by means of an automorphism $\Psi$
of $\langle x,y,z_0\rangle$ that fixes $x$ and such that
$\Psi^{2^{m}}$ is conjugation by~$x$. For this purpose,
we consider the assignment
	 \[
    x\mapsto x,\;
    y\mapsto z_0^{-2^m}y^{1+2^{m-1}v}=z^{-2} y^{1+2^{m-1}v},\;
    z_0\mapsto z_0x^{-u}.
    \]
		
		Let us verify that the defining relations of $\langle x,y,z_0\rangle$ are preserved.
		This is clear for $xy = yx$, $x^{z_0} = x^\alpha$, $x^{2^{\ell+1}}=1$, and ${}^{z_0}y=y^\beta$. 
		In regards to $x^{2^{\ell}}=z_0^{2^{2m-1}}$, we have
		\[
    (z_0x^{-u})^{2^{2m-1}}
    = z_0^{2^{2m-1}} x^{-u(\alpha^{2^{2m-1}}-1)/(\alpha-1)}=1,
    \]
	since $(\alpha^{2^{2m-1}}-1)/(\alpha-1)\equiv  0\mod 2^{2m-1}$ and $2m-1\geq \ell+1$.
	As for $x^{2^{\ell-1}}y^{2^{\ell-1}}=1$, we have
	$$
	(z^{-2} y^{1+2^{m-1}v})^{2^{\ell-1}}=z^{-2^{\ell}}y^{2^{\ell-1}(1+2^{m-1}\ell)}=1,
	$$
since $\ell\geq m+1$ and $\ell+m-2\geq \ell+1$, that is, $m\geq 3$. Regarding $z_0^{2^\ell k}= x^{2^{m-1}u^2}y^{2^{m-1}u^2}$, we have
\[
    (z_0x^{-u})^{2^{\ell k }}
    = z_0^{2^{\ell k}} x^{-u(1 + \alpha + \cdots + \alpha^{2^{\ell k}-1})}=z_0^{2^{\ell k}}x^{2^{\ell}},
    \]
since $v_2((\alpha^{\ell k}-1)/(\a-1))=\ell$ and $x^{2^{\ell+1}}=1$, and
$$
(z^{-2}y^{1+2^{m-1}v})^{2^{m-1}u^2}=z^{-2^{m}}y^{2^{m-1}u^2}=x^{2^{\ell}}y^{2^{m-1}u^2}
$$
since $2m-2\geq \ell+1$, that is, $\ell\leq 2m-3$.

Thus the above assignment extends to a surjective endomorphism and hence an automorphism $\Psi$ of $\langle x,y,z_0\rangle$.
Proposition \ref{abc2abe} ensures that $\Psi^{2^m}$ is conjugation by $x$. 
This produces the required extension, where $\Psi$ is conjugation by $x_0$. We see that $\langle x_0,y,z_0\rangle$ has defining relations:
    \[
    y^{x_0} = z_0^{-2^m}y^{1+2^{m-1}v},\;
    x_0^{z_0} = x_0^\alpha,\;
    {}^{z_0}y = y^\beta,\;
		\]
		\[
    x_0^{2^{\ell+m}}=z_0^{2^{2m-1}},\; x_0^{2^{\ell+m+1}}=1,\; x_0^{2^{\ell+m-1}}y^{2^{\ell-1}}=1,\;
		z_0^{2^\ell k}= x_0^{2^{2m-1}u^2}y^{2^{m-1}v^2}.
    \]

		We finally construct a cyclic extension $\langle x_0,y_0,z_0\rangle$
of $\langle x_0,y,z_0\rangle$ of order $2^{5m+\ell-1}$, where
$y_0^{2^{m}}=y$, by means of an automorphism $\Pi$
of $\langle x_0,y,z_0\rangle$ that fixes $y$ and such that
$\Pi^{2^{m}}$ is conjugation by~$y$. For this purpose,
we consider the assignment
    \[
    x_0\mapsto x_0z_0,\;
    y\mapsto y,\;
    z_0\mapsto y^v z_0.
    \]
	Let us verify that the defining relations of $\langle x_0,y,z_0\rangle$ are preserved. 
	
$\bullet$ ${}^{z_0}y = y^\beta$. This is clear.
	
$\bullet$ $x_0^{2^{\ell+m+1}}=1$. Since $\ell+m+1\geq 2m$ and $v_2((\a^{2^{\ell+m+1}}-1)/(\a-1))=\ell+m+1$,
$$
(x_0z_0)^{2^{\ell+m+1}}=z_0^{2^{\ell+m+1}}x_0^{\a(\a^{2^{\ell+m+1}}-1)/(\a-1)}=1.
$$

$\bullet$  $x_0^{2^{\ell+m}}=z_0^{2^{2m-1}}$. We have
$$
(x_0z_0)^{2^{\ell+m}}=z_0^{2^{\ell+m}}x_0^{\a(\a^{2^{\ell+m}}-1)/(\a-1)}=x_0^{2^{\ell+m}},
$$
since $\ell+m\geq 2m$, $v_2((\a^{2^{\ell+m+1}}-1)/(\a-1))=\ell+m$, and $x_0^{2^{\ell+m+1}}=1$.
Moreover,
$$
(y^v z_0)^{2^{2m-1}}=y^{v(\b^{2^{2m-1}}-1)/(\b-1)}z_0^{2^{2m-1}}=z_0^{2^{2m-1}},
$$
since $v_2((\b^{2^{2m-1}}-1)/(\b-1))=2m-1\geq \ell+1$.

$\bullet$  $x_0^{2^{\ell+m-1}}y^{2^{\ell-1}}=1$. We have
$$
(x_0z_0)^{2^{\ell+m-1}}=z_0^{2^{\ell+m-1}}x_0^{\a(\a^{2^{\ell+m-1}}-1)/(\a-1)}=x_0^{2^{\ell+m-1}},
$$
since $\ell+m-1\geq 2m$ and $\a(\a^{2^{\ell+m-1}}-1)/(\a-1)\equiv 2^{\ell+m-1}\mod 2^{\ell+m+1}$, using $m\geq 3$.

$\bullet$  $z_0^{2^\ell k}= x_0^{2^{2m-1}u^2}y^{2^{m-1}v^2}$. We have,
$$
(y^v z_0)^{2^{\ell}k}=y^{v(\b^{2^{\ell}k}-1)/(\b-1)}z_0^{2^{\ell}k}=y^{2^{\ell}}z_0^{2^{\ell}k},
$$
since $v_2(\b^{2^{\ell}k}-1)/(\b-1))=\ell$ and $y^{2^{\ell+1}}=1$. Here $y^{2^{\ell}}=x_0^{2^{\ell+m}}=z_0^{2^{2m-1}}$, so
$$
(y^v z_0)^{2^{\ell}k}=z_0^{2^{2m-1}}z_0^{2^{\ell}k}.
$$
On the other hand, using $z_0^{2^{2m}}=1$, we find that
$$
(x_0z_0)^{2^{2m-1}u^2}=z_0^{2^{2m-1}}x_0^{\a(\a^{2^{2m-1}u^2}-1)/(\a-1)}=z_0^{2^{2m-1}}x_0^{2^{2m-1}u^2},
$$
since $\a(\a^{2^{2m-1}u^2}-1)/(\a-1)\equiv 2^{2m-1}u^2\mod 2^{\ell+m+1}$, using $3m-2\geq \ell+m+1$, that is, $2m-3\geq\ell$.
	
$\bullet$  $y^{x_0} = z_0^{-2^m}y^{1+2^{m-1}v}$. We have
    \[
    y^{x_0z_0}
    = (z_0^{-2^m} y^{1+2^{m-1}v})^{z_0}
    = z_0^{-2^m} y^{\gamma (1+2^{m-1}v)}
    = z_0^{-2^m} y^{\gamma + 2^{m-1}v},
    \]
    since $\beta 2^{m-1}v\equiv 2^{m-1}v\mod 2^{2m-1}$ with $2m-1\geq \ell+1$. On the other hand,
    \[
    (y^v z_0)^{-2^m} y^{1+2^{m-1}v}
    = (y^{v(1+\beta +\cdots +\beta^{2^m-1})}z_0^{2^m})^{-1} y^{1+2^{m-1}v},
    \] 
		where $(\beta^{2^m}-1)/(\beta-1)\equiv 2^m\mod 2^{2m-1}$ and $2m-1\geq\ell+1$, so
    \[
    (y^vz_0)^{-2^m} y^{1+2^{m-1}v}
    = (y^{2^m v}z_0^{2^m})^{-1} y^{1+2^{m-1}v}
    = z_0^{-2^m} y^{-2^m v} y^{1+2^{m-1}v}
    = z_0^{-2^m} y^{\gamma + 2^{m-1}v}.
    \]
    
$\bullet$  $x_0^{z_0} = x_0^\alpha$. The proof of Theorem \ref{teo14} applies.

That $\Pi^{2^{m}}$ is conjugation by~$y$ follows as in the proof of \cite[Theorem 11.4]{MS}. The rest of the proof goes as usual.
\end{proof}

We continue the general case $m=n$, $m<\ell<2m$, $\ell\leq 2m-3$, $m\geq 3$, and proceed to make more explicit calculations. 
By (\ref{qe}), we have
$$
b^{\d_\b}a^{\d_\a}=a^{\d_\a}b^{\d_\b}z, \quad z\in \langle a^{2^{3m-1}}, b^{2^{3m-1}}\rangle,
$$
so
\begin{equation}\label{qe59}
(a^{\d_\a}b^{\d_\b})^i=a^{\d_\a i} b^{\d_\b i} z^{i(i-1)/2},\quad i\geq 1, z\in \langle a^{2^{3m-1}}, b^{2^{3m-1}}\rangle.
\end{equation}
Recall that $\a-\b=2^\ell k$, so that $u-v=2^{\ell-m}k$, where $2\nmid k$. 
As indicated at the beginning of this section, we have $a^{2^{3m}}=1=b^{2^{3m}}$. Since $2m-\ell\geq 2$, 
raising (\ref{qe}) to the $2^{(2m-\ell)}$th power and taking (\ref{dadb}) and (\ref{qe59}) into account yields
\begin{equation}\label{qe4}
c^{2^{2m}}=a^{2^{4m-\ell-1} u^2/k}b^{p^{4m-\ell-1} v^2/k}.
\end{equation}

We next obtain an alternative expression for $c^{2^{2m}}$ by squaring (\ref{prin2}).
To achieve this, recall that $v_2(\l_\a)=3m-2=v_2(\l_\b)$ and $a^{2^{3m}}=1=b^{2^{3m}}$, so $[a^{\l_\a},c]=1=[\b^{\l_\b},c]$. 
Hence by (\ref{prin2}), 
\begin{equation}\label{kici}
c^{2\d_\a}=b^{-2(\b^{\b-\a}-1)}a^{2\l_\a},\; c^{2\d_\b}=b^{-2\l_\b}a^{2(\a^{\a-\b}-1)}.
\end{equation}
To unravel (\ref{kici}), recall that $c^{2^{3m-1}}=1$, as indicated at the beginning of this section,
so (\ref{dadb}) gives
$$
c^{\d_\a}=c^{2^{2m-1} u^2}, c^{\d_\b}=c^{2^{2m-1} v^2}.
$$
But $u^2\equiv v^2\mod 2^{\ell-m+1}$, and $c^{2^{m+\ell+1}}=1$, as indicated at the beginning of this section,
so
$$
c^{2\d_\b}=c^{2^{2m} u^2}=c^{2\d_\a}.
$$
On the other hand, whether $\a>\b$ or $\a<\b$, we see that
$$
a^{\a^{\a-\b}-1}=a^{2^{m+\ell} uk}, b^{\b^{\b-\a}-1}=b^{-2^{m+\ell} vk}.
$$
Thus (\ref{kici}) gives
\begin{equation}\label{cook}
c^{2^{2m} u^2}=c^{2\d_\a}=b^{-2(\b^{\b-\a}-1)}a^{2\l_\a}=b^{2^{m+\ell+1} vk}z_1,\quad
z_1\in\langle a^{2^{3m-1}}\rangle,
\end{equation}
\begin{equation}\label{cook2}
c^{2^{2m} u^2}=c^{2\d_\b}=b^{-2\l_\b}a^{2(\a^{\a-\b}-1)}=z_2 a^{2^{m+\ell+1} uk},\quad
z_2\in\langle b^{2^{3m-1}}\rangle.
\end{equation}
Therefore by (\ref{qe4}), (\ref{cook}), and (\ref{cook2}), we have
\begin{equation}\label{ct0}
b^{2^{m+\ell+1} vk}a^{2^{3m-1}s}=b^{2^{3m-1}t} a^{2^{m+\ell+1} uk}=c^{2^{2m} u^2}=a^{2^{4m-\ell-1} u^4/k}b^{p^{4m-\ell-1} u^2 v^2/k}.
\end{equation}

Suppose first that $2\ell+2<3m$. Thus, setting $i=m+\ell+1$ and $j=4m-\ell-1$, we have $i<j< 3m-1$. Then (\ref{ct0}) yields
$$
a^{2^i}\in \langle b^{2^j}\rangle, b^{2^i}\in \langle b^{2^j}\rangle,
$$
which easily implies $a^{2^{m+\ell+1}}=1=b^{2^{m+\ell+1}}$.

Suppose next that $2\ell+2=3m$. Then $m+\ell+1=4m-\ell-1$.
Since $\ell<2m-2$, we still have $m+\ell+1<3m-1$. Since $k$, $u$, and $v$ are odd, (\ref{ct0}) now gives
$$
a^{2^{m+\ell+1}}\in \langle b^{2^{m+\ell+2}}\rangle, b^{2^{m+\ell+1}}\in \langle a^{2^{m+\ell+2}}\rangle,
$$
which easily implies $a^{2^{m+\ell+1}}=1=b^{2^{m+\ell+1}}$. 

Suppose for the remainder of this section that $2\ell+2\leq 3m$. By the above,
\begin{equation}\label{feo}
a^{2^{m+\ell+1}}=1=b^{2^{m+\ell+1}}.
\end{equation}
Thus by (\ref{cook}) or (\ref{cook2}),
\begin{equation}\label{cook4}
c^{2^{2m}}=1.
\end{equation}
As $m+\ell+1\leq 3m-2$, we can use (\ref{prin2}), (\ref{feo}), and (\ref{cook4}) to deduce
\begin{equation}\label{cook5}
c^{2^{2m-1}}=b^{2^{\ell+m}}=a^{2^{\ell+m}}\in Z, a^{2^{3m-1}}=1=b^{2^{3m-1}},
\end{equation}
as well as (\ref{qe}) and (\ref{feo}) to infer $c^{\a-\b}=a^{\d_\a}b^{\d_\b}=b^{\d_\b}a^{\d_\a}$. Raising this to the power $2^{2m-\ell-1}$ and 
appealing to  (\ref{feo}), (\ref{cook4}), and  (\ref{cook5}) yields
$$
c^{2^{2m-1}}=a^{2^{\ell+m}}b^{2^{\ell+m}}=a^{2^{\ell+m+1}}=1,
$$
so (\ref{cook5}) gives
$$
a^{2^{\ell+m}}=1=b^{2^{\ell+m}}=c^{2^{2m-1}}.
$$
Since $a^{2^{m+\ell}}=1=b^{2^{m+\ell}}=c^{2^{2m}}$, 
the general observations at the beginning of this section imply that the class of $G_2$ is at most 5, with
$$
a^{2^{2m-1}}, b^{2^{2m-1}}, c^{2^{\ell}}\in Z, a^{2^{\ell}}, b^{2^{\ell}}, c^{2^{m-1}}\in Z_2,
a^{2^{m}}, b^{2^{m}}\in Z_3, c\in Z_4, Z_5=G_2.
$$
Using $a^{2^{m+\ell}}=1=b^{2^{m+\ell}}$ and raising $c^{\a-\b}=a^{\d_\a}b^{\d_\b}=b^{\d_\b}a^{\d_\a}$
to the $2^{(\ell-m+1)}$th power, we get $c^{2^{2\ell-m+1}}=1$. Here
$b^{2^{2m-1}}\in\langle a\rangle\langle c\rangle$ by (\ref{qe}). Therefore 
$G_2=\langle a\rangle\langle c\rangle\langle b\rangle$  yields
$$|G_2|\leq 2^{m+\ell}2^{2\ell-m+1}2^{2m-1}=2^{2m+3\ell}.$$

\begin{theorem}\label{teo19}  Suppose that $n=m<\ell<2m$, $2\ell+2\leq 3m$, $\ell\leq 2m-3$, and $m\geq 3$. Then $e=2m+3\ell, f=5$,
$o(a)=2^{\ell+m}=o(b)$, and $o(c)=2^{2\ell-m+1}$.
\end{theorem}

\begin{proof}
Consider the abelian group of order $2^{3\ell-m+1}$ generated by $x,y,z$ subject to the defining relations $[x,y]=[x,z]=[y,z]=1$,
as well as
\[
    x^{2^{\ell}}=y^{2^{\ell}}=z^{2^{2\ell-2m+2}}=1,\;
		z^{2^{\ell-(m-1)} k}= x^{2^{m-1}u^2}y^{2^{m-1}v^2}=x^{2^{m-1}u^2}y^{2^{m-1}u^2},
\]		
where $x,y,z$ play the roles of $a^{2^m}, b^{2^m}, c^{2^{m-1}}$, respectively (where we used $u^2\equiv v^2\mod 2^{\ell-m+1}$ and
$y^{2^{\ell}}=1$).

We next construct a cyclic extension $\langle x,y,z_0\rangle$
of $\langle x,y,z\rangle$ of order $2^{3\ell}$, where
$z_0^{2^{m-1}}=z$, by means of an automorphism $\Omega$
of $\langle x,y,z\rangle$ that fixes $z$ and such that
$\Omega^{2^{m-1}}$ is conjugation by~$z$, that is, the trivial automorphism. In order to achieve this goal,
we consider the assignment
		\[
    x\mapsto x^\alpha,\;
    y\mapsto y^\gamma,\;
    z\mapsto z,
    \]
		where $\gamma=1-2^m v$ is the inverse of $\beta$ modulo $2^{2m}$, noting that $2\ell-2m+2\leq 2m$, that is, $\ell\leq 2m-1$,
		which ensures that the defining relations of $\langle x,y,z\rangle$ are preserved. 
		Thus the above assignment extends to an endomorphism $\Omega$ of $\langle x,y,z\rangle$ which is clearly surjective
		and hence an automorphism of $\langle x,y,z\rangle$. Let us verify that $\Omega^{2^{m-1}}$ acts trivially on $x,y,z$.
		This is obviously true for $z$, and since $\alpha^{2^{m-1}}\equiv 1 \mod 2^{2m-1}$ and $\gamma^{2^{m-1}}\equiv 1 \mod 2^{2m-1}$,
		with $2m-1\geq\ell$, it is also true of $x$ and $y$. This produces the required extension, where $\Omega$ is conjugation by $z_0$. We readily verify that $\langle x,y,z_0\rangle$ has defining relations:
    \[
		xy = yx,\;
    x^{z_0} = x^\alpha,\;
    {}^{z_0}y=y^\beta,\;
		x^{2^{\ell}}=y^{2^{\ell}}=z_0^{2^{2\ell-m+1}}=1,\;
		z_0^{2^\ell k}= x^{2^{m-1}u^2}y^{2^{m-1}u^2}.
	\]
	
	We next construct a cyclic extension $\langle x_0,y,z_0\rangle$
of $\langle x,y,z_0\rangle$ of order $2^{3\ell+m}$, where
$x_0^{2^{m}}=x$, by means of an automorphism $\Psi$
of $\langle x,y,z_0\rangle$ that fixes $x$ and such that
$\Psi^{2^{m}}$ is conjugation by~$x$. For this purpose,
we consider the assignment
	 \[
    x\mapsto x,\;
    y\mapsto z_0^{-2^m}y^{1+2^{m-1}v}=z^{-2} y^{1+2^{m-1}v},\;
    z_0\mapsto z_0x^{-u}.
    \]
		
		Let us verify that the defining relations of $\langle x,y,z_0\rangle$ are preserved.
		This is clear for $xy = yx$, $x^{z_0} = x^\alpha$, and $x^{2^{\ell}}=1$. Regarding ${}^{z_0}y=y^\beta$,
		we have
		\[
    {}^{(z_0x^{-u})}(z^{-2}y^{1+2^{m-1}v})= z^{-2} y^{\beta (1+2^{m-1}v)}=z^{-2\beta} y^{\beta (1+2^{m-1}v)}=
		(z^{-2}y^{1+2^{m-1}v})^\beta,
    \]
		as $2\beta\equiv 2\mod 2^{m+1}$ and therefore $2\beta\equiv 2\mod 2^{2(\ell-m+1)}$, since $m+1\geq 2\ell-2m+2$,
		that is, $3m+1\geq 2\ell+2$. As for $z_0^{2^{2\ell-m+1}}=1$, we have
		\[
    (z_0x^{-u})^{2^{2\ell-m+1}}
    = z_0^{2^{2\ell-m+1}} x^{-u(1 + \alpha + \cdots + \alpha^{2^{2\ell-m+1}-1})}=1,
    \]
	since $(\alpha^{2^{2\ell-m+1}}-1)/(\alpha-1)\equiv  0\mod 2^{2\ell-m+1}$ and $2\ell-m+1\geq \ell$, that is, $\ell\geq m-1$.
	In regards to $y^{2^{\ell}}=1$, we have
	$$
	(z^{-2} y^{1+2^{m-1}v})^{2^{\ell}}=z^{-2^{\ell+1}}y^{2^\ell(1+2^{m-1}\ell)}=1,
	$$
since $\ell+1\geq 2\ell-2m+2$, that is, $2m\geq \ell+1$. Regarding $z_0^{2^\ell k}= x^{2^{m-1}u^2}y^{2^{m-1}u^2}$, we have
\[
    (z_0x^{-u})^{2^{\ell k }}
    = z_0^{2^{\ell k}} x^{-u(1 + \alpha + \cdots + \alpha^{2^{\ell k}-1})}=z_0^{2^{\ell k}},
    \]
since $v_2((\alpha^{\ell k}-1)/(\a-1))=\ell$, and
$$
(z^{-2}y^{1+2^{m-1}v})^{2^{m-1}u^2}=y^{2^{m-1}u^2},
$$
since $m\geq 2\ell-2m+2$, that is, $3m\geq 2\ell+2$, and $2m-2\geq \ell$.

Thus the above assignment extends to a surjective endomorphism and hence an automorphism $\Psi$ of $\langle x,y,z_0\rangle$.
By Proposition \ref{abc2abe}, $\Psi^{2^{m}}$
		is conjugation by $x$. This produces the required extension, where $\Psi$ is conjugation by $x_0$. We readily verify that $\langle x_0,y,z_0\rangle$ has defining relations:
    \[
    y^{x_0} = z_0^{-2^m}y^{1+2^{m-1}v},\;
    x_0^{z_0} = x_0^\alpha,\;
    {}^{z_0}y = y^\beta,\;
    z_0^{2^{2\ell-m+1}} = x_0^{2^{\ell+m}}=y^{2^\ell}=1,\;
		z_0^{2^\ell k}= x_0^{2^{2m-1}u^2}y^{2^{m-1}u^2}.
    \]

		We finally construct a cyclic extension $\langle x_0,y_0,z_0\rangle$
of $\langle x_0,y,z_0\rangle$ of order $2^{3\ell+2m}$, where
$y_0^{2^{m}}=y$, by means of an automorphism $\Pi$
of $\langle x_0,y,z_0\rangle$ that fixes $y$ and such that
$\Pi^{2^{m}}$ is conjugation by~$y$. For this purpose,
we consider the assignment
    \[
    x_0\mapsto x_0z_0,\;
    y\mapsto y,\;
    z_0\mapsto y^v z_0.
    \]
	Let us verify that the defining relations of $\langle x_0,y,z_0\rangle$ are preserved. The first, third,
	and fourth relations are easily verified. Regarding the second relation, the proof of Theorem \ref{teo14} applies.
 As for the fifth relation, namely $z_0^{2^\ell k}= x_0^{2^{2m-1}u^2}y^{2^{m-1}u^2}$,
				we have
				$$
				(y^vz_0)^{2^\ell k}=y^{v(\b^{2^{\ell k}}-1)/(\b-1)}z_0^{2^\ell k}=z_0^{2^\ell k},
				$$
				since $(\b^{2^{\ell k}}-1)/(\b-1)\equiv 0\mod 2^{\ell}$. Also,
				$$
				(x_0z_0)^{2^{2m-1}u^2}=x_0^{\a^{2^{2m-1}u^2}-1)/(\a-1)} z_0^{2^{2m-1}u^2}=x_0^{2^{2m-1}u^2},
				$$
				since $2m-1\geq 2\ell-m+1$ in the case of $z_0$ and $3m-2\geq \ell+m$, in the case of $x_0$.
				
				The fact that $\Pi^{2^m}$ is conjugation by $y$ can be seen as in the proof of \cite[Theorem 11.4]{MS}.
				The rest of the proof goes as usual.
	\end{proof}

Reviewing all sections when $p=2$, it turns out that we always have $a^{2^{3m-1}}=1=b^{2^{3n-1}}$.

\section{The conditions $\a>1$, $\b>1$ are unnecessary}\label{gc}

We resume here the general case $\a,\b\neq 1$ and $m,n>0$ (without assuming that $\a>1$ or $\b>1$).

\begin{theorem}\label{pos} There are integers $\a_0,\b_0>1$ such that $G(\a,\b)_p\cong G(\a_0,\b_0)_p$.
\end{theorem}


\begin{proof} All unexplained notation is taken from \cite{M}. In view of the isomorphism
$G(\a,\b)\cong G(\b,\a)$, we may assume without loss that $\a\geq\b$. If $\b>1$ there is nothing to do,
so we may suppose that $\b<0$.

If $\a>1$ then \cite[Eq. (2.28)]{M} yields
$$
a^{\g_\a(\a-1)^2}=1=b^{\eta(\b-1)^2},
$$
while if $\a<0$ the argument given in \cite[p. 606]{M} leads to
$$
\a^{\xi(\a-1)^2}=1=b^{\eta(\b-1)^2}.
$$

Suppose first that $p>3$. Then
$$
v_p(\g_\a(\alpha-1)^2))=4m=v_p(\xi(\alpha-1)^2)),\; v_p(\eta(\b-1)^2)=4n
$$
by \cite[Propositions 2.1 and 4.1]{MS}. It follows from Corollary \ref{pres} that
$$
G(\a,\b)_p=\langle a,b\,|\, a^{[a,b]}=a^{\a},\, b^{[b,a]}=b^{\b}, a^{p^{4m}}=1=b^{p^{4n}}\rangle.
$$
Set $\a_0=\a+p^{4m}x$ and
$\b_0=\b+p^{4n}y$, where $x,y\in\N$ are large enough so that $\a_0,\b_0>1$. Then $v_p(\a_0-1)=m$ and
$v_p(\b_0-1)=n$, so
$$
G(\a,\b)_p=\langle a,b\,|\, a^{[a,b]}=a^{\a_0},\, b^{[b,a]}=b^{\b_0}, a^{p^{4m}}=1=b^{p^{4n}}\rangle=G(\a_0,\b_0)_p.
$$

Suppose next that $p=2$. Then
$$
v_2(\g_\a(\alpha-1)^2))=4m-1=v_2(\xi(\alpha-1)^2)),\; v_2(\eta(\b-1)^2)=4n-1
$$
by \cite[Propositions 2.1 and 4.1]{MS}. It follows from Corollary \ref{pres} that
$$
G(\a,\b)_2=\langle a,b\,|\, a^{[a,b]}=a^{\a},\, b^{[b,a]}=b^{\b}, a^{2^{4m-1}}=1=b^{2^{4n-1}}\rangle.
$$
Set $\a_0=\a+2^{4m-1}x$ and
$\b_0=\b+2^{4n-1}y$, where $x,y\in\N$ are large enough so that $\a_0,\b_0>1$. Then $v_2(\a_0-1)=m$ and
$v_2(\b_0-1)=n$, so
$$
G(\a,\b)_2=\langle a,b\,|\, a^{[a,b]}=a^{\a_0},\, b^{[b,a]}=b^{\b_0}, a^{2^{4m-1}}=1=b^{2^{4n-1}}\rangle=G(\a_0,\b_0)_2.
$$

Suppose finally that $p=3$. By Proposition \ref{nuevaca}, we have
\begin{equation}\label{cenab}
a^{(\alpha-1)^2\m_\a}=1=b^{(\beta-1)^2\m_\b}.
\end{equation}

Suppose first that $\a>0$ and $\b<0$. We then have
$$
a^{\g_\a(\a-1)^2}=1=a^{\m_\a(\a-1)^2},
$$
by \cite[Eq. (2.28)]{M} and (\ref{cenab}). Thus, if $\a\not\equiv -2\mod 9$
$$
a^{3^{4m}}=1
$$
by \cite[Proposition 2.1]{MS}, while when $\a\equiv -2\mod 9$, we have
$$
a^{243}=1
$$
by \cite[Proposition 2.1]{MS} if $\a\not\equiv -2\mod 27$ and by \cite[Proposition 2.2]{MS} if $\a\equiv -2\mod 27$.

On the other hand, by \cite[Eq. (2.28)]{M} and (\ref{cenab}), we have
$$
b^{\eta(\b-1)^2}=1=b^{\m_\b(\b-1)^2}.
$$
If $\b\not\equiv -2\mod 9$, then $v_3(\eta(\b-1)^2)=4m$ by \cite[Proposition 4.1]{MS}, in which case
$$
b^{3^{4m}}=1,
$$
while when $\b\equiv -2\mod 9$, we have
$$
b^{243}=1
$$
by \cite[Proposition 2.2]{MS} if $\b\equiv -2\mod 27$ and by \cite[Proposition 4.1]{MS} if $\b\not\equiv -2\mod 27$.

If $\b\not\equiv -2\mod 9$, set $\b_0=\b+3^{4m}x$, where $x\in \N$ is large enough so that $\b_0>1$.
If $\b\equiv -2\mod 9$, set $\b_0=\b+243x$, where $x\in \N$ is large enough so that $\b_0>1$. Then $\b_0\equiv \b\mod 27$
and $v_3(\b_0-1)=m$. Thus, if $\a\not\equiv -2\mod 9$ and $\b\not\equiv -2\mod 9$, then
$$
G(\a,\b)_3=\langle a,b\,|\, a^{[a,b]}=a^{\a},\, b^{[b,a]}=b^{\b_0}, a^{3^{4m}}=1=b^{3^{4m}}\rangle=G(\a,\b_0)_3;
$$
if $\a\not\equiv -2\mod 9$ and $\b\equiv -2\mod 9$, then
$$
G(\a,\b)_3=\langle a,b\,|\, a^{[a,b]}=a^{\a},\, b^{[b,a]}=b^{\b_0}, a^{3^{4m}}=1=b^{243}\rangle=G(\a,\b_0)_3;
$$
if $\a\equiv -2\mod 9$ and $\b\equiv -2\mod 9$, then
$$
G(\a,\b)_3=\langle a,b\,|\, a^{[a,b]}=a^{\a},\, b^{[b,a]}=b^{\b_0}, a^{243}=1=b^{243}\rangle=G(\a,\b_0)_3;
$$
and if $\a\equiv -2\mod 9$ and $\b\not\equiv -2\mod 9$, then
$$
G(\a,\b)_3=\langle a,b\,|\, a^{[a,b]}=a^{\a},\, b^{[b,a]}=b^{\b_0}, a^{243}=1=b^{3^{4m}}\rangle=G(\a,\b_0)_3.
$$

Suppose next that $\a<0$ and $\b<0$. Then by the argument given in \cite[p. 606]{M} and by (\ref{cenab}), we have
$$
a^{\xi (\a-1)^2}=1=a^{\m_\a(\a-1)^2}.
$$
Thus, if $\a\not\equiv -2\mod 9$
$$
a^{3^{4m}}=1
$$
by Proposition \cite[Proposition 4.1]{MS}, while when $\a\equiv -2\mod 9$, we have
$$
a^{243}=1
$$
by \cite[Proposition 2.2]{MS} if $\a\equiv -2\mod 27$ and by \cite[Proposition 4.1]{MS} if $\a\not\equiv -2\mod 27$.

Likewise, by the argument given in \cite[p. 606]{M} and (\ref{cenab}), we have
$$
b^{\eta(\b-1)^2}=1=b^{\m_\b(\b-1)^2}.
$$
If $\b\not\equiv -2\mod 9$, then $v_3(\eta(\b-1)^2)=4m$ by \cite[Proposition 4.1]{MS}, in which case
$$
b^{3^{4m}}=1,
$$
while when $\b\equiv -2\mod 9$, we have
$$
b^{243}=1
$$
by \cite[Proposition 2.2]{MS} when $\b\equiv -2\mod 27$ and by \cite[Proposition 4.1]{MS} when $\b\not\equiv -2\mod 27$.

If $\a\not\equiv -2\mod 9$, set $\a_0=\b+3^{4m}x$, where $x\in \N$ is large enough so that $\a_0>1$.
If $\a\equiv -2\mod 9$, set $\a_0=\b+243x$, where $x\in \N$ is large enough so that $\a_0>1$.

If $\b\not\equiv -2\mod 9$, set $\b_0=\b+3^{4m}y$, where $y\in \N$ is large enough so that $\b_0>1$.
If $\b\equiv -2\mod 9$, set $\b_0=\b+243y$, where $y\in \N$ is large enough so that $\b_0>1$.

Then $\a_0\equiv \a\mod 27$, $\b_0\equiv \b\mod 27$, $v_3(\a_0-1)=m$,
and $v_3(\b_0-1)=m$. Thus, if $\a\not\equiv -2\mod 9$ and $\b\not\equiv -2\mod 9$, then
$$
G(\a,\b)_3=\langle a,b\,|\, a^{[a,b]}=a^{\a_0},\, b^{[b,a]}=b^{\b_0}, a^{3^{4m}}=1=b^{3^{4m}}\rangle=G(\a_0,\b_0)_3;
$$
if $\a\not\equiv -2\mod 9$ and $\b\equiv -2\mod 9$, then
$$
G(\a,\b)_3=\langle a,b\,|\, a^{[a,b]}=a^{\a_0},\, b^{[b,a]}=b^{\b_0}, a^{3^{4m}}=1=b^{243}\rangle=G(\a_0,\b_0)_3;
$$
if $\a\equiv -2\mod 9$ and $\b\equiv -2\mod 9$, then
$$
G(\a,\b)_3=\langle a,b\,|\, a^{[a,b]}=a^{\a_0},\, b^{[b,a]}=b^{\b_0}, a^{243}=1=b^{243}\rangle=G(\a_0,\b_0)_3;
$$
and if $\a\equiv -2\mod 9$ and $\b\not\equiv -2\mod 9$, then
\[
G(\a,\b)_3=\langle a,b\,|\, a^{[a,b]}=a^{\a_0},\, b^{[b,a]}=b^{\b_0}, a^{243}=1=b^{3^{4m}}\rangle=G(\a_0,\b_0)_3.\qedhere
\]
\end{proof}

We proceed to use Theorem \ref{pos} to show that all our structural results, from Theorem \ref{teo1} to Theorem~\ref{teo19}
inclusive, are valid without assuming that $\a>1$ or $\b>1$.

Indeed, let $\a_0$ and $\b_0$ be as defined in the proof of Theorem \ref{pos}. Note that $\a_0=1+p^{m} u_0$ and $\b_0=1+p^{n} v_0$,
where $p\nmid u_0,v_0$, that is, $v_p(\a-1)=v_p(\a_0-1)$ and $v_p(\b-1)=v_p(\b_0-1)$.
Recall that $\a=1+p^m u$, $b=1+p^n v$, $\ell=v_p(\a-\b)$,
and that $\a-\b=p^\ell k$ when $\a\neq\b$. Set $\ell_0=v_p(\a_0-\b_0)$ and write
$\ell_0=p^\ell k_0$ when $\a_0\neq \b_0$.

Suppose first that $p>3$, or $p=3$ and $\a,\b\not\equiv 7\mod 9$. In view of the isomorphism $G(\a,\b)\cong G(\b,\a)$,
we may assume without loss that $m\geq n$ when studying the structure of $G(\a,\b)_p$.
Assume first that $\ell=n$, which is equivalent to $\ell_0=n$.
Then Theorem \ref{teo1} gives the structure of $G(\a_0,\b_0)_p$
in terms of $m$ and $n$ only. Therefore, Theorem \ref{teo1} is true without assuming $\a>1$ or $\b>1$.
Assume next that $\ell\neq n$.
Then $\ell_0\neq n$ and $\ell,\ell_0> m=n$. If $\ell\geq 2m$, which is equivalent to $\ell_0\geq 2m$,
then Theorem \ref{teo2} gives the structure of $G(\a_0,\b_0)_p$
in terms of $m$, so Theorem \ref{teo2} is true without assuming $\a>1$ or $\b>1$. Assume next that $m<\ell<2m$,
which is equivalent to $m<\ell_0<2m$, in which case $\ell_0=\ell$. If $2\ell<3m$ (resp.  $2\ell>3m$) then
Theorem \ref{teo3} (resp. Theorem \ref{teo4}) gives the structure of $G(\a_0,\b_0)_p$ in terms of $m$ and $\ell$,
so Theorem \ref{teo3} (resp. Theorem \ref{teo4}) is true without assuming $\a>1$ or $\b>1$. It remains to consider the case $2\ell=3m$.
Setting
 $s=v_p(2k^2-u^3)$ and $s_0=v_p(2k_0^2-u_0^3)$, we have $0\leq s<m/2 \Leftrightarrow
0\leq s_0<m/2$, in which case
$s_0=s$, and $s\geq m/2 \Leftrightarrow s_0\geq m/2$. Theorem \ref{teo5} gives the structure of $G(\a_0,\b_0)_p$
in terms of $m$ and $s$ when $0\leq s<m/2$, and it terms of $m$ only when $s\geq m/2$. Thus Theorem \ref{teo5} is also valid
without assuming $\a>1$ or $\b>1$.

Suppose next that $p=3$. Note that $\a_0\equiv \a\mod 27$ and $\b_0\equiv \b\mod 27$.
Hence $\a,\b\equiv 7\mod 9\Leftrightarrow\a_0,\b_0\equiv 7\mod 9$,
in which case either $\a\equiv\b\mod 27$, which means $\a_0\equiv\b_0\mod 27$, or $v_3(\a-\b)=2$,
which means $v_3(\a_0-\b_0)=2$. Thus, Theorems \ref{teo6} and \ref{teo7}
are true without assuming $\a>1$ or $\b>1$. Moreover, $\a\equiv 7\mod 9$ and $\b\equiv 4\mod 9$ (resp. $\b\equiv 1\mod 9$)
means that $\a_0\equiv 7\mod 9$ and $\b_0\equiv 4\mod 9$ (resp. $\b_0\equiv 1\mod 9$), so Theorem \ref{teo8} (resp. Theorem \ref{teo9})
is true without assuming $\a>1$ or $\b>1$.

Suppose finally that $p=2$. Assume first that $m=1$ or $n=1$.
Since Theorems \ref{teo10}, \ref{teo11}, and \ref{teo12} give the structure of $G(\a,\b)_2$
in terms of $m$ and $n$ only, they are valid without assuming $\a>1$ or $\b>1$. Assume next that $m,n>1$.
In view of the isomorphism $G(\a,\b)\cong G(\b,\a)$, we may assume without loss that $m\geq n$.
If $\ell=n$, which is equivalent to $\ell_0=n$, then Theorem \ref{teo13} gives the structure of $G(\a_0,\b_0)_2$
in terms of $m$ and $n$, so Theorem \ref{teo13} is true without assuming $\a>1$ or $\b>1$.
Assume in what follows that $m,n>1$ and $\ell\neq n$. Then $\ell_0\neq n$ and $\ell,\ell_0> m=n$.
If $\ell\geq 2m$, which is equivalent to $\ell_0\geq 2m$,
then Theorem \ref{teo14} gives the structure of $G(\a_0,\b_0)_2$
in terms of $m$, so Theorem \ref{teo14} is true without assuming $\a>1$ or $\b>1$.
Assume next that $m<\ell<2m$,
which is equivalent to $m<\ell_0<2m$, in which case $\ell_0=\ell$.
If $\ell=2m-1$, then Theorem \ref{teo15} gives the structure of $G(\a_0,\b_0)_2$
in terms of $m$, so Theorem \ref{teo15} is true without assuming $\a>1$ or $\b>1$.
This settles the case $m=2$, so we assume henceforth that $\ell\leq 2m-2$ and $m\geq 3$.
If $\ell=2m-2$, then Theorem \ref{teo16} gives the structure of $G(\a_0,\b_0)_2$
in terms of $m$, so Theorem \ref{teo16} is true without assuming $\a>1$ or $\b>1$.
We may assume henceforth that $\ell\leq 2m-3$. If $2\ell+2=3m+1$ then necessarily $m\geq 5$ is
odd and setting $s=v_2(u^3-k^2)$ and $s_0=v_2(u_0^3-k_0^2)$,
we have $s<(m-3)/2 \Leftrightarrow
s_0<(m-3)/2$, in which case
$s_0=s$, and $s\geq (m-3)/2 \Leftrightarrow s_0\geq (m-3)/2$; Theorem \ref{teo17} gives the structure of $G(\a_0,\b_0)_2$
in terms of $m$ and $s$ when $s<(m-3)/2$, and it terms of $m$ only when $s\geq (m-3)/2$;
thus Theorem \ref{teo17} is true without assuming $\a>1$ or $\b>1$. If $2\ell+2>3m+1$ (resp. $2\ell+2\leq 3m$),
then Theorem~\ref{teo18} (resp. Theorem~\ref{teo19}) gives the structure of $G(\a_0,\b_0)_2$
in terms of $m$ and $\ell$, so Theorem \ref{teo18}  (resp. Theorem \ref{teo19}) is true without assuming $\a>1$ or $\b>1$.

\medskip

\noindent{\bf Acknowledgments.} We thank V. Gebhardt and A. Previtali for GAP and Magma calculations, and J. Cruickshank and
A. Montoya Ocampo for proofreading parts of the paper and helpful comments.


\end{document}